\documentclass[10pt]{article}

\usepackage{amsmath}    
\usepackage{graphicx}   
\usepackage{verbatim}   
\usepackage{color}      
\usepackage{hyperref}   

\usepackage{amssymb}
\usepackage{amsthm}

\newtheorem{Def}{Definition}[section]
\newtheorem{Eg}{Example}[section]
\newtheorem{Prop}[Def]{Proposition}
\newtheorem{Lem}[Def]{Lemma}
\newtheorem{Thm}[Def]{Theorem}
\newtheorem{Cor}[Def]{Corollary}
\newtheorem{Ass}[Def]{Assumption}
\theoremstyle{definition}
\newtheorem{Rem}[Def]{Remark}

\newcommand{\p}{\mathbb{P}}
\newcommand{\e}{\mathbb{E}}
\newcommand{\real}{\mathbb{R}}
\newcommand{\n}{\mathbb{N}}

\newcommand{\sign}{\text{sign}}
\newcommand{\1}{{\bf 1}}
\newcommand{\mf}{\mathcal{F}}

\newcommand{\rd}{\mathrm{d}}
\allowdisplaybreaks

\begin{document}

\title{
	Probability density function of SDEs with unbounded and path--dependent drift coefficient
}
\author{
	Dai Taguchi
	\footnote{
		Research Institute for Interdisciplinary Science,
		Okayama University
		3-1-1
		Tsushima-naka,
		Kita-ku
		Okayama
		700-8530,
		Japan,
		email:\texttt{ dai.taguchi.dai@gmail.com}
	}
	\quad and \quad
	Akihiro Tanaka
	\footnote{
		Graduate School of Engineering Science,
		Osaka University,
		1-3, Machikaneyama-cho, Toyonaka,
		Osaka, 560-8531, Japan
		/
		Sumitomo Mitsui Banking Corporation,
		1-1-2, Marunouchi, Chiyoda-ku, 
		Tokyo, 100-0005, Japan,
		Email: \texttt{tnkaki2000@gmail.com}
	}
}
\maketitle
\begin{abstract}
	In this paper, we first prove that the existence of a solution of SDEs under the assumptions that the drift coefficient is of linear growth and path--dependent, and diffusion coefficient is bounded, uniformly elliptic and H\"older continuous.
	We apply Gaussian upper bound for a probability density function of a solution of SDE without drift coefficient and local Novikov condition, in order to use Maruyama--Girsanov transformation.
	The aim of this paper is to prove the existence with explicit representations (under linear/super--linear growth condition), Gaussian two--sided bound and H\"older continuity (under sub--linear growth condition) of a probability density function of a solution of SDEs with path--dependent drift coefficient.
	As an application of explicit representation, we provide the rate of convergence for an Euler--Maruyama (type) approximation, and an unbiased simulation scheme.\\
	\textbf{2010 Mathematics Subject Classification}: 65C30; 62G07; 35K08; 60H35\\
	\textbf{Keywords}:
	Probability density function;
	Maruyama--Girsanov theorem;
	Gaussian two--sided bound;
	Parametrix method;
	Euler--Maruyama scheme;
	Unbiased simulation;
	H\"older continuity
\end{abstract}

\tableofcontents

\section{Introduction}
Let $X^{x}=(X_t^{x})_{t \geq 0}$ be a solution of path--dependent $d$-dimensional stochastic differential equations (SDEs)
\begin{align}\label{SDE_1}
	\rd X_t^{x}
	=b(t,X^{x})\rd t
	+\sigma(t,X_t^{x})\rd W_t,
	~t \geq 0,~X_0^{x}=x\in \real^d,
\end{align}
where $W=(W_t)_{t \geq 0}$ is a $d$-dimensional standard Brownian motion on a probability space $(\Omega, \mathcal{F},\p)$ with a filtration $(\mathcal{F}_t)_{t\geq 0}$ satisfying the usual conditions, and the drift coefficient $b:[0,\infty) \times C([0,\infty);\real^d) \to \real^d$ and diffusion matrix $\sigma:[0,\infty) \times \real^d \to \real^{d \times d}$ are measurable.

The existence and regularity of a probability density function (pdf) of $X_t^x$ with respect to Lebesgue measure have been studied by many authors.
If the drift $b:[0,\infty) \times \real^d \to \real^d$ is bounded H\"older continuous and diffusion matrix $\sigma$ is bounded, uniformly elliptic and H\"older continuous, then
it is well-known that by using Levi's parametrix method (see, \cite{Fr64}), there exists the fundamental solution of parabolic type partial differential equations (Kolmogorov equation), and by Feynman-Kac formula, it is a pdf of a solution of associated SDEs (see also, \cite{KoMa,KoKoMe,LeMe,McSi67}).
Note that the parametrix method can be applied to the case of $L^p([0,T] \times \real^d)$-valued drift with $p \geq d+2$ \cite{Po}, H\"older continuous (unbounded) drift \cite{DeKr02}, Brownian motion with signed measure belonging to the Kato class \cite{KiSo} and Hyperbolic Brownian motion with drift \cite{IdIm}.
Moreover, Qian and Zheng \cite{QiZh02, QiZh03} provided a sharp two--sided bounds for a pdf of a Brownain motion with bounded drift coefficient by using pinned diffusion arguments (or regular conditional probability) and by choosing a parametrix as the pdf of bang--bang diffusion process.
On the other hand, the existence, Gaussian two--sided bound and H\"older continuity for the fundamental solution of the parabolic equations in divergence form were proved by Aronson  \cite{Ar67} and Nash \cite{Na58}.
It is worth noting that
Fabes and Kenig \cite{FaKe81} provided an example of diffusion coefficient $\sigma$ such that the law of $X_t=x+\int_{0}^{t} \sigma(s,X_s)\rd W_s$ is purely singular with respect to Lebesgue measure.

There are several methods to study a pdf as a probabilistic approach.
One of the most useful tool is the Malliavin calculus (see, \cite{IkWa, Nu06,Shi04}).
Kusuoka and Stroock \cite{KuSt85} proved 
if the coefficients are Fr\'echet differentiable and diffusion coefficient uniformly elliptic, then a solution of path--dependent SDEs admits a smooth pdf.
This results were extended by Takeuchi \cite{Take16} for the existence and smoothness of a joint pdf of finite dimensional distribution of SDEs.
By using the Malliavin calculus,
De~Marco \cite{DM11} studied a local existence of a pdf under local smoothness conditions on the coefficients, 
Kohatsu-Higa and Tanaka \cite{KoTa12} 
studied the existence of a pdf of additive functionals of SDEs with bounded measurable drift coefficient,
Hayashi, Kohatsu-Higa and Y\^uki \cite{HaKoYu13, HaKoYu14} studied the H\"older continuity,
and recently, Olivera and Tudor \cite{OlTu} proved that by using It\^o--Tanaka trick or Zvonkin transform (e.g. \cite{Ve,Zv}), a solution of SDE with H\"older continuous drift (unbounded) can be transformed by diffeomorphism to an equivalent equation with differentiable coefficients and a pdf.
On the other hand, by using the stochastic control method, Sheu \cite{Sh91} study the Gaussian two--sided bound for a pdf of a solution of time--homogeneous Markovian SDEs $\rd X_t=b(X_t)\rd t+\sigma(X_t) \rd W_t$.

The existence of a pdf of a solution of SDEs with path--dependents and non-smooth coefficients have been studied recently.
Fournier and Printems \cite{FoPr10} proved the existence of a pdf of a solution of one-dimensional SDEs with path--dependent coefficients, stochastic heat equations and L\'evy driven SDEs, by using ``one-step" Euler--Maruyama scheme and Fourier transform approach.
As an extension of the approach in \cite{FoPr10}, Bally and Caramellino \cite{BaCa} provided an interpolation method, and proved the existence of a pdf for a solution of multi-dimensional path--dependent SDEs.

On the other hand as a perturbation approach, the parametrix method and Maruyama--Girsanov theorem are also useful tool in order to prove the existence of a pdf of a solution of SDEs with path--dependents coefficients.
Frikha and  Li \cite{FrLi} studied the existence of a weak solution of SDEs with path--dependent coefficients and its pdf by using the parametrix method.
In the case of bounded and path--dependent drift coefficient, Makhlouf \cite{Ma15} and Kusuoka \cite{Ku17} studied the existence, explicit representation, Gaussian two--sided bound and H\"older continuity (see also \cite{BaKr16,BaKr17}).
In particular, Makhlouf (Theorem 3.1 in \cite{Ma15}) showed that the following representation holds for a pdf, denoted by $p_t(x,\cdot)$, of Brownian motion with random drift $\rd X_t=b_t \rd t + \rd W_t$, $X_0=x$:
\begin{align}\label{rep_density_0}
	p_t(x,y)
	=g_t(x,y)
	+\int_{0}^{t}
		\e\left[
			\langle 
				\nabla g_{t-s}(X_s,y),
				b_s
			\rangle
		\right]
	\rd s,
	~\text{a.e.},~y \in \real^d,
\end{align}
where $g_t(x,y):=\frac{\exp\left(-|y-x|^2/2t\right)}{(2 \pi t)^{d/2}}$, which is an analogue of the parametrix method.
On the other hand, Kusuoka  \cite{Ku17} showed that a pdf of a solution of path--dependent SDE \eqref{SDE_1} with bounded drift coefficient, denoted by $p_t(x,\cdot)$, has the following representation:
\begin{align}\label{rep_density_1}
	p_t(x,y)
	=q(0,x;t,y)
	\e
	\left[
	Z_t(1,Y^{0,x})
	~\middle|~Y_t^{0,x}=y
	\right],~\text{a.e.},~y \in \real^d,
\end{align}
where $q(0,x;t,\cdot)$ is the pdf of a solution of SDE without drift: $\rd Y_t^{0,x}=\sigma(t,Y_t^{0,x}) \rd W_t$, $Y^{0,x}_0=x$, $Z_t(1,Y^{0,x})$ is the Girsanov density (see, Theorem \ref{main_0} below), and $\e[~\cdot~|~Y_t^{0,x}=y ]$ is the expectation of a regular conditional probability given $Y_t^{0,x}=y$ for $y \in \real^d$.
This representation is an analogue of Maruyama's result on the proof of Girsanov's theorem (see, Theorem 1 in \cite{Mar54}).
Note that these representations were also shown by Qian and Zheng \cite{QiZh04} (see, Lemma 2.3 and Theorem 2.4 in \cite{QiZh04}) for time--homogeneous diffusion processes with drift coefficient $b:\real^d \to \real^d$ satisfying at most linear growth and Novikov condition.

The aim of this paper is to prove that the existence of a weak solution of SDEs under the assumptions that the drift coefficient is of linear growth and path--dependent, and diffusion coefficient is bounded, uniformly elliptic and H\"older continuous.
We apply the Gaussian upper bound for the pdf of a solution of SDE $\rd Y_t^{0,x}=\sigma(t,Y_t^{0,x}) \rd W_t$, $Y^{0,x}_0=x$ and ``local" Novikov condition, in order to use Maruyama--Girsanov transformation.
We will also show that the existence of a pdf and the representations \eqref{rep_density_0} and \eqref{rep_density_1} hold under linear growth condition on $b$.
By using these representations, we show that Gaussian two--sided bound and H\"older continuity hold under sub--linear growth condition on $b$.

As an application of the representation of a pdf, we consider the rate of convergence for a pdf of the Euler--Maruyama scheme, which is a standard discrete approximation for a solution of SDEs (see, \cite{KP,Ma55}).
It is worth noting that Maruyama's proof of Girsanov's theorem in \cite{Mar54} is based on this approximation.
For the Euler--Maruyama scheme $X^{(x,n)}$, many authors studies the strong error $\e[\sup_{0 \leq t <T} |X_t^{x}-X_t^{(x,n)}|^p]^{1/p}$, for some $p\geq 1$ and the weak error $|\e[f(X_T^{x})]-\e[f(X_T^{(x,n)})]|$ for some measurable function $f:\real^d \to \real$, (see, e.g. \cite{BT,BaSh18,KP,KLY,LeSz17b,Ma55,MeTa,MP,NT1,NT2,TT,Yan}, see also \cite{AnKo,BaKo15,DoOuWa17,HeTaTo17,IdIm,KoTaZh,KoYu} as an unbiased simulation scheme based on the Euler--Maruyama scheme with random grid and parametrix expansions/Malliavin weight).
On the other hand, recently, the convergence rate for a pdf of the Euler--Maruyama scheme studied by using the parametrix method and Malliavin calculus (see, e.g. \cite{G08,KoKoMe,Gu,NT2}).
In this paper, we provide the rate of convergence by using the representation of \eqref{rep_density_0} for pdfs of Brownian motion with path--dependent drift and its Euler--Maruyama approximation.

Recently, if the coefficients of SDE grow super--linearly, then the standard Euler--Maruyama approximation does not converge to a solution of the equation (see, Theorem 2.1 in \cite{HuJeKl11}).
In order to approximate a solution of Markov type SDE with super--linear growth coefficients, several tamed Euler--Maruyama approximations are proposed (see e.g., \cite{HuJeKl12,Sa13,Sa16}).
In this paper, inspired by \cite{HuJeKl12,Sa13,Sa16}, we use arguments of Fournier and Printems \cite{FoPr10} with ``one-step" tamed Euler--Maruyama approximation in order to prove absolute continuity of the law of $X_t^x$ with respect to the Lebesgue measure under super--linear growth, Khasminskii and one-sided Lipschitz condition on the coefficients.
It is worth noting that Romito \cite{Ro18} studied the existence and Besov regularity of the probability density function of a solution of SDEs with locally bounded drift, and locally H\"older continuous, elliptic diffusion coefficient, by using localization argument and one-step Euler--Maruyama scheme.
Therefore, the result proved in section \ref{Sec_super} is included in the result of Romito (see, Theorem 4.1 in \cite{Ro18}).
On the other hand, in our paper, we use directly one-step tamed Euler scheme, so the approach of proof is different.

This paper is structured as follows.
In section \ref{Sec_sde}, we prove the existence and weak uniqueness of a solution of SDE \eqref{SDE_1} with linear growth and path--dependent drift coefficient.
In section \ref{Sec_pdf}, we study a pdf of a solution of SDEs \eqref{SDE_1}.
We first provide two representations for a pdf in subsection \ref{Sub_exist}, and prove a Gaussian two--sided bound in subsection \ref{Sub_GB}.
In subsection \ref{Sub_para}, we provide third representation for a pdf of a solution of SDEs with path--dependent and bounded drift, and as an application of this representation, in subsection \ref{Sub_sharp}, we prove a sharp bounds for a pdf of Brownian motion with drift which is inspired by \cite{QiZh02,QiZh03} (see also, \cite{BaKr16}).
In subsection \ref{Sub_comp}, we consider a comparison property of pdfs, and in subsection \ref{Sub_EM}, as an application of a comparison property for pdfs of a solution of SDE and its Euler--Maruyama scheme, we provide its rate of convergence.
In subsection \ref{Sub_expa}, we consider the parametrix method for a pdf of a solution of Markovian SDEs with unbounded drift, and in subsection \ref{Sub_unbi} as an application of the parametrix method, we provide an unbiased simulation scheme introduced by Bally and Kohatsu-Higa \cite{BaKo15}.
In subsection \ref{Sub_Hol}, we study H\"older continuity of a pdf by using the explicit representation.
In section \ref{Sec_super}, we prove the existence of a pdf of a solution of one--dimensional SDEs under super--linear growth conditions on the coefficients.
In a short Appendix, we provide an explicit calculation for beta type integrals.


\subsection*{Notations}
We give some basic notations and definitions used throughout this paper.
We consider the elements of $\real^d$ are column vectors and for $x\in\real^d$, we denote $x=(x^1,\ldots,x^d)^{\top}$.
Let $T>0$ be fixed.
We denote by $C([0,\infty);\real^d)$ the space of continuous functions $w:[0,\infty) \to \real^d$ with metric $\rho$ defined by $\rho(w,w')=\sum_{k=1}^{\infty} 2^{-k} (\max_{0\leq t \leq k} |w_t-w_t'| \wedge 1).$
Let us denote by $\mathcal{B}(C([0,\infty);\real^d))$ the topological $\sigma$--field on $C([0,\infty);\real^d)$, and $\mathcal{B}_t(C([0,\infty);\real^d))$  the sub--$\sigma$--field defined by $\{\rho_t^{-1}(A)~|~A \in \mathcal{B}(C([0,\infty);\real^d))\}$, where $\rho_t(w)(s):=w(t \wedge s)$, (see, Chapter IV in \cite{IkWa}).
For $w \in C([0,\infty);\real^d)$, define $w_t^{*}:=\sup_{0 \leq s \leq t} |w_s|$.
Let $C_b^{\infty}(\real^d;\real^q)$ be the space of $\real^q$-valued functions such that all the derivatives are bounded.
For an invertible $d \times d$-matrix $A=(A_{i,j})_{1 \leq i,j\leq d}$, we denote $|A|^2:=\sum_{i,j=1}^{d} A_{i,j}^2$ and
$g_A(x,y)
=\frac{\exp(-\frac{1}{2} \langle A^{-1}(y-x),y-x \rangle)}
{(2\pi)^{d/2} \sqrt{\det A}},$
and $g_{c}(x,y)=g_{cI}(x,y)$, for $c \in \real \backslash \{0\}$ where the matrix $I$ is the identity matrix.
We denote the sign function by $\mbox{sgn}(x):=-\1_{(-\infty,0]}(x)+\1_{(0,\infty)}(x)$ for $x \in \real$, and the gamma function by $\Gamma(x):=\int_{0}^{\infty} t^{x-1}e^{-t} \rd t$ for $x \in (0,\infty)$.
We will use Hermite polynomials associated with the Gaussian density of order $1$ and $2$ denoted by $H^i$ and $H^{i,j}$, that is, for an invertible matrix $A$,
$
H_A^i(y):= -(A^{-1}y)^i \text{ and } H_A^{i,j}(y):= (A^{-1}y)^i (A^{-1}y)^j - (A^{-1})_{i,j}
$.

\section{SDEs with unbounded and path--dependent drift}\label{Sec_sde}

In this section, we will show that weak existence and uniqueness in law for a solution of SDE \eqref{SDE_1} on $[0,T]$.

It is known that if $\sigma$ is identity matrix $I_d$ and $b:[0,T] \times C([0,T];\real^d) \to \real^d$ is a progressively measurable functional on $C([0,T];\real^d)$ satisfying the linear growth condition $|b(t,w)|\leq K(1+w^*_t)$, for all $(t,w) \in [0,T] \times C([0,T];\real^d)$, then there exists a unique weak solution of SDE \eqref{SDE_1} (see, Proposition 5.3.6 and Remark 5.3.8 in \cite{KS}).
The idea of proof is based on Maruyama--Girsanov transform with ``local" Novikov condition (see, Corollary 3.5.14 and 3.5.16 in \cite{KS}).
In order to extend this result to non-constant diffusion matrix, we use the Gaussian upper bound for a pdf of a solution of SDE without drift coefficient (see, \eqref{SDE_2} below).

We need the following assumptions on the coefficients $b$ and $\sigma$.

\begin{Ass}\label{Ass_1}
	
	We suppose that the coefficients $b=(b^{(1)},\ldots,b^{(d)})^{\top}:[0,\infty) \times C([0,\infty);\real^d) \to \real^d$ and $\sigma=(\sigma_{i,j})_{1 \leq i,j \leq d}:[0,\infty)\times \real^d \to \real^{d\times d}$ satisfy the following conditions:
	\begin{itemize}
		\item[(i)]
		The drift coefficient $b$ is $\mathcal{B}([0,\infty)) \otimes \mathcal{B}(C([0,\infty);\real^d))/\mathcal{B}(\real^d)$-measurable and for each fixed $t>0$, the map $C([0,\infty);\real^d) \ni w\mapsto b(t,w) \in \real^d$ is $\mathcal{B}_t(C([0,\infty);\real^d))/\mathcal{B}(\real^d)$-measurable (see, Chapter IV, Definition 1.1 in \cite{IkWa}), and is of linear growth, that is, for each $T>0$, there exists $K(b,T)>0$ such that for any $(t,w) \in [0,T] \times C([0,T];\real^d)$,
		\begin{align*}
			|b(t,w)|
			\leq
			K(b,T)(1+w^{*}_t).
		\end{align*}
			
		\item[(ii)]
		$a:=\sigma \sigma^{\top}$ is $\alpha$-H\"older continuous in space and $\alpha/2$-H\"older continuous in time with $\alpha \in (0,1]$, that is,
		\begin{align*}
		\|a\|_{\alpha}:=
		\sup_{t \in [0,\infty), x \neq y}\frac{|a(t,x)-a(t,y)|}{|x-y|^{\alpha}}
		+\sup_{x\in \real^d, t \neq s}\frac{|a(t,x)-a(s,x)|}{|t-s|^{\alpha/2}}
		<\infty.
		\end{align*}
		
		\item[(iii)]
		The diffusion coefficient $\sigma$ is bounded and uniformly elliptic, that is, there exist $\underline{a}, \overline{a}>0$ such that for any $(t,x,\xi) \in [0,\infty) \times \real^d \times \real^d$,
		\begin{align*}
			\underline{a}|\xi|^2 \leq \langle a(t,x) \xi,\xi \rangle \leq \overline{a} |\xi|^2.
		\end{align*}
	\end{itemize}
\end{Ass}

\begin{Eg}\label{Eg_0}
	Suppose that $\nu:\mathbb{S}:=[0,\infty) \times \real^d \times [0,\infty) \times (\real^{d})^{\n} \times \real^{\ell} \to \real^d$ is measurable and satisfies that there exists $K>0$ and $\theta:=\{\theta_i\}_{i \in \n} \in [0,\infty)^{\n}$ such that $\|\theta\|_{\ell_1}:=\sum_{i \in \n} \theta_i<\infty$ and for any $\chi=(t,w,z,\{u_i\}_{i\in \n},v) \in \mathbb{S}$,
	\begin{align*}
		|\nu(\chi)|
		\leq K(1+\|\chi\|_{\theta}),
	\end{align*}
	where $\|\chi\|_{\theta}^2:=t^2+|w|^2+|z|^2+\sum_{i \in \n} \theta_i |u_i|^2+|v|^2$.
	We define $A_t : C([0,\infty);\real^d) \to \mathbb{S}$ by
	\begin{align}\label{Eg_1}
		A_t(w)
		=
		\left(
			t,
			w_t,
			\max_{0 \leq s \leq t} \zeta(s,w_s),
			\{w_{\tau_i(t)}\}_{i\in \n},
			\int_{0}^{t} c(s,w_s) \rd s,
		\right) \in \mathbb{S},
		~
		w \in C([0,\infty);\real^d),
	\end{align}
	where $\tau_i(t):=(t-\tau_i)\1_{(\tau_i,\infty)}(t)$, $\tau_i >0$, $i \in \n$,  $\zeta :[0,\infty) \times \real^{d} \to [0,\infty)$ and $c:[0,\infty) \times \real^d \to \real^{\ell}$ are measurable and of linear growth. 
	Then since $\zeta$ and $c$ are of linear growth, there exists $C>0$ such that for any $w \in C([0,\infty);\real^d)$, $\|A_t(\omega)\|_{\theta}\leq C(1+(1+\|\theta\|_{\ell_1})w^{*}_t)$, thus $b=\nu \circ A_{\cdot}$ satisfies Assumption \ref{Ass_1} (i).
\end{Eg}

Recall that we fixed $T>0$ arbitrarily.
Let us consider the following Markovian SDE without drift coefficient: for given $s \in [0,T)$,
\begin{align}\label{SDE_2}
	Y_t^{s,x}
	=x
	+\int_{s}^{t}
		\sigma(r,Y_r^{s,x})
	\rd W_r,~
	t \in [s,T].
\end{align}
If the diffusion coefficient $\sigma$ satisfies the Assumption \ref{Ass_1} (ii) and (iii), then a weak existence and uniqueness in law holds (see, e.g. Theorem 4.1 and 5.6 in \cite{StVa}).
Moreover, from Theorem 5.4 in \cite{Fr75}, $Y_t^{s,x}$ admits the pdf (fundamental solution) $q(s,x;t,\cdot)$ for any $t \in (s,T]$ and $q(s,x;t,y)$ is the solution of the Kolmogorov backward equation:
\begin{align}\label{pde_fund_0}
	(\partial_s + L_s)q(s,x;t,y)=0,
	~\lim_{s \uparrow t}\int_{\real^d} f(y) q(s,x;t,y) \rd y
	= f(x),~f \in C_b^{\infty}(\real^d;\real),
\end{align}
where $L_s$ is a differential operator defined by
\begin{align*}
	L_sf(x)
	:=\frac{1}{2}\sum_{i,j=1}^d a_{i,j}(s,x)\frac{\partial^2 f}{\partial x_i \partial x_j}(x),
\end{align*}
(see page 149 in \cite{Fr75}).
Moreover, for any $f \in C_b^{\infty}(\real^d;\real)$, the function $u(s,x;t):=\e\left[f\left( Y_{t}^{s,x} \right) \right]$
is a solution to the following partial differential equation:
\begin{align}
	\begin{split} \label{heat_eq_1}
		(\partial_s+L_s) u(s,x;t)
		&=0,~(s,x)\in [0,t) \times \real^d,\\
		u(t,x;t)
		&=f(x),~ x\in \real^d,
\end{split}
\end{align}
(see, Theorem 5.3 in \cite{Fr75}).

The following lemma shows that $q(s,x;t,\cdot)$ and its derivative with respect to $x$ satisfy the Gaussian bounded.

\begin{Lem} \label{bound_drev_1}
	Suppose that Assumption \ref{Ass_1}-(ii) and (iii) holds.
	Then for all $t \in (s,T]$, $Y_t^{s,x}$ admits the pdf $q(s,x;t,\cdot)$ with respect to Lebesgue measure, and there exist $\widehat{C}_{\pm}>0$ and $\widehat{c}_{\pm}>0$ such that for any $(t,x,y) \in (s,T] \times \real^d \times \real^d$,
	\begin{align}
	\widehat{C}_{-} g_{\widehat{c}_{-}(t-s)}(x,y)
	\leq 
	q(s,x;t,y)
	&\leq \widehat{C}_{+} g_{\widehat{c}_{+}(t-s)}(x,y),\label{bound_qt}\\
	|\partial_{x_i} q(s,x;t,y)|
	&\leq \frac{\widehat{C}_{+}}{(t-s)^{1/2}} g_{\widehat{c}_{+}(t-s)}(x,y).\label{bound_drev_2}
	\end{align}
\end{Lem}

The upper bound \eqref{bound_qt} and \eqref{bound_drev_2} are shown in \cite{Fr64}, Theorem 9.4.2.
For the lower bound \eqref{bound_qt}, we refer section 4.2 in \cite{LeMe}, and \cite{Sh91} for time independent case, see also Chapter 7, section 6 in \cite{Bass}.

By using the Gaussian upper bound \eqref{bound_qt}, we prove that the existence and uniqueness in law holds for SDE \eqref{SDE_1} under linear growth condition on $b$.

\begin{Thm}\label{main_0}
	Suppose that Assumption \ref{Ass_1} holds.
	Then SDE \eqref{SDE_1} has a weak solution and uniqueness in law holds on $[0,T]$.
	In particular, for any measurable functional $f:C([0,T];\real^d) \to \real$ such that the expectation $\e[f(Y^{0,x}) Z_T(1,Y^{0,x})]$ exists, it holds that
	\begin{align}\label{Girsanov_0}
		\e[f(X^{x})]
		=\e[f(Y^{0,x}) Z_T(1,Y^{0,x})],
	\end{align}
	where for $q\in \real$, $Z(q,Y^{0,x})=(Z_t(q,Y^{0,x}))_{t \in [0,T]}$ is a martingale defined by
	\begin{align*}
		Z_t(q,Y^{0,x})
		&:=
		\exp\left(
			\sum_{j=1}^{d}
				\int_{0}^{t}
					q \mu^j(s,Y^{0,x})
				\rd W_s^{j}
			-\frac{1}{2}
				\int_{0}^{t}
					|q\mu(s,Y^{0,x})|^2
				\rd s
		\right),\\
		\mu(t,w)
		&:=
		\sigma(t,w_t)^{-1} b(t,w),~(t,w) \in [0,T] \times C([0,T];\real^d).
	\end{align*}
\end{Thm}

Before proving Theorem \ref{main_0}, we prove the following lemma which shows that the random variable $X_T^{x,*}$ is $L^p$-integrable for all $p>0$.

\begin{Lem}\label{moment_0}
	Assume that $F:[0,\infty) \times C([0,\infty);\real^d) \to \real^d$ is $\mathcal{B}([0,\infty)) \otimes \mathcal{B}(C([0,\infty);\real^d))/\mathcal{B}(\real^d)$-measurable and for each fixed $t>0$, the map $C([0,\infty);\real^d) \ni w\mapsto F(t,w) \in \real^d$ is $\mathcal{B}_t(C([0,\infty);\real^d))/\mathcal{B}(\real^d)$-measurable, and is of linear growth, that is, for each $T>0$, there exists $K(F,T)>0$ such that for any $(t,w) \in [0,T] \times C([0,T];\real^d)$,
	\begin{align*}
	|F(t,w)|
	\leq
	K(F,T)(1+w^{*}_t).
	\end{align*}
	Suppose that Assumption \ref{Ass_1} holds and let $X^{x}$ be a solution of SDE \eqref{SDE_1}.
	Then for any $p>0$, there exists $C_{b,\sigma}(p,F,T)>0$ such that
	\begin{align*}
	\e\left[
	\sup_{t \in [0,T]}
	|F(t,X^{x})|^p
	\right]^{1/p}
	\leq C_{b,\sigma}(p,F,T) (1+|x|).
	\end{align*}
\end{Lem}
\begin{proof}
	It is suffices to show the statement for $p \geq 2$.
	Since $b$ is of linear growth and $\sigma$ is bounded, applying Jensen's inequality and Burkholder-Davis-Gundy's inequality, there exist $\widetilde{C}_1,\widetilde{C}_2>0$ such that for all $t \in [0,T]$,
	\begin{align*}
	\e\left[
	|X_t^{x,*}|^p
	\right]
	\leq
	\widetilde{C}_1(1+|x|^p)
	+\widetilde{C}_2\int_{0}^{t}
	\e\left[
	|X_s^{x,*}|^p
	\right]
	\rd s.
	\end{align*}
	Hence Gronwall's inequality implies the statement.
\end{proof}

\begin{proof}[Proof of Theorem \ref{main_0}]
	The proof is based on Corollary 3.5.16, Proposition 5.3.6, Remark 5.3.8 and Proposition 3.10 in \cite{KS}.
	
	We first prove that $Z(q,Y^{0,x})$ is a martingale for all $q \in \real$.
	Since usual Novikov condition may be fail in this setting, we apply a local Novikov condition as follows.
	From Corollary 3.5.14 in \cite{KS}, it suffices to prove that for any fixed $T>0$, there exist $n(T) \in \n$ and a sequence $\{t_0,\ldots,t_{n(T)}\}$ such that $0=t_0<t_1< \cdots <t_{n(T)}=T$ and 
	\begin{align*}
		\e\left[
			\exp
				\left(
					\frac{1}{2}
					\int_{t_{n-1}}^{t_n}
						|q\mu(s,Y^{0,x})|^2
					\rd s
				\right)
		\right]
		<\infty,~\text{for all } n=1,\ldots,n(T).
	\end{align*}
	Let $M_t^{x}:=Y_t^{0,x}-x$.
	Then $M^{x}=(M_t^{x})_{t \in [0,T]}$ is a martingale since $\sigma$ is bounded.
	By Assumption \ref{Ass_1} (i) and (iii), we have
	\begin{align*}
		\int_{t_{n-1}}^{t_n}
			|q\mu(s,Y^{0,x})|^2
		\rd s
		\leq
		(t_n-t_{n-1}) \underline{a} |qK(b,T)|^2(1+|x|+M^{x,*}_T)^2.
	\end{align*}
	Note that $U_t^{x}:=\exp(\frac{1}{4}(t_n-t_{n-1}) \underline{a} |qK(b,T)|^2(1+|x|+|M_t^{x}|)^2)$ is a sub-martingale, so by using Doob's inequality (see, Theorem 1.3.8 (iv) in \cite{KS}), we have
	\begin{align*}
		\e\left[
			\exp
				\left(
					\frac{1}{2}
					\int_{t_{n-1}}^{t_n}
						|q \mu(r,Y^{0,x})|^2
					\rd r
				\right)
		\right]
		&\leq
		\e\left[
			\exp\left(
				\frac{1}{2}(t_n-t_{n-1}) \underline{a} |qK(b,T)|^2(1+|x|+M^{x,*}_T)^2
			\right)
		\right]\\
		&=
		\e\left[
			|U_T^{x,*}|^2
		\right]
		\leq
		4\e[|U_T^{x}|^2].
	\end{align*}
	Using the Gaussian upper bound \eqref{bound_qt}, we have
	\begin{align*}
		\e[|U_T^{x}|^2]
		&=\int_{\real^d}
			\exp\left(
				\frac{1}{2}(t_n-t_{n-1}) \underline{a} |qK(b,T)|^2(1+|x|+|y-x|)^2
			\right)
			q(0,x;T,y)
		\rd y\\
		&\leq
		\widehat{C}_{+}
		\int_{\real^d}
			\exp\left(
			\frac{1}{2}(t_n-t_{n-1}) \{2 \underline{a} |qK(b,T)|^2\} \{(1+|x|)^2+|y-x|^2\}
			\right)
			g_{\widehat{c}_{+} T}(x,y)
		\rd y.
	\end{align*}
	We choose $n(T) \in \n$ and the sequence $\{t_0,\ldots,t_{n(T)}\}$ satisfying
	\begin{align*}
		t_n-t_{n-1}
		\leq \frac{1}{2 \underline{a} |qK(b,T)|^2 \widehat{c}_{+} T},
	\end{align*}
	which provides $\e[|U_T^{x}|^2]<\infty$.
	This concludes that $Z(q,Y^{0,x})$ is a martingale.
	
	We define the new measure $\mathbb{Q}$ on the measurable space $(\Omega, \mathcal{F}_T)$ as
	\begin{align*}
		\frac{\rd \mathbb{Q}}{\rd \p}
		=Z_T(1,Y^{0,x}).
	\end{align*}
	Since $Z(1,Y^{0,x})$ is a martingale, the measure $\mathbb{Q}$ is a probability measure on $(\Omega, \mathcal{F}_T)$.
	Moreover, from Maruyama--Girsanov theorem, $B=(B_t=(B_t^1,\ldots,B_t^d)^{\top})_{0\leq t \leq T}$, which is defined by for each $j=1,\ldots,d$,
	\begin{align*}
		B_t^j
		:=W_t^j
		-\left\langle
			W^j, \sum_{\ell=1}^{d} \int_{0}^{\cdot} \mu^{\ell} (s,Y^{0,x}) \rd W_s^{\ell}
		\right \rangle_t
		=W_t^j
		-\int_0^t
			\mu^j (s,Y^x)
		\rd s,
	\end{align*}
	is a $d$-dimensional standard Brownian motion on the probability space $(\Omega,\mathcal{F}_T,\mathbb{Q})$.
	Hence we have
	\begin{align*}
		\rd Y_t^{0,x}
		=
		\sigma(t,Y_t^{0,x}) \rd W_t
		=
		b(t,Y^{0,x}) \rd t
		+\sigma(t,Y_t^{0,x}) \rd B_t,
	\end{align*}
	thus, $Y^x$ is a solution of SDE \eqref{SDE_1} with $Y_0^{0,x}=x$ under the probability measure $\mathbb{Q}$.
	
	Next, we prove the uniqueness in law.
	The proof is based on Proposition 5.3.10 in \cite{KS}.
	Let $(X^{i},W^{i})$, $(\Omega^{i}, \mathcal{F}^{i}, \p^{i})$, $\{\mathcal{F}^{i}_{t}\}_{t \geq 0}$, $i=1,2$ be two solution of SDE \eqref{SDE_1}.
	For each $k \geq 1$, let
	\begin{align*}
		\tau^{i}_{k} := T \wedge 
		\inf
		\left\{
			t \in [0,T]~;~ \int_{0}^{t} |\mu (s,X^i)|^2 \rd s =k
		\right\} .
	\end{align*}
	From Lemma \ref{moment_0}, $\tau^{i}_{k} \to T$ as $k \to \infty$, almost surely.
	Then, for each $k \in \n$ and $i=1,2$
	$(Z_{t\wedge \tau_k^i}(-1,X^i))_{t \in [0,T]}$ is a martingale on $(\Omega^{i}, \mf_{T}^{i}, \p^{i} )$.
	For each $i=1,2$, we define new a measure on $(\Omega^i, \mf_T^i )$ as
	\begin{align*}
	\frac{\rd \mathbb{Q}_k^{i}}{\rd \p^{i}}
	=Z_{T \wedge \tau_k^i}(-1,X^{i}).
	\end{align*} 
	Then from Maruyama--Girsanov theorem,  for $i=1,2$, $(B_{t \wedge \tau_k^i}^i=(B_{t \wedge \tau_k^i}^{i,1},\ldots,B_{t \wedge \tau_k^i}^{i,d})^{\top})_{0\leq t \leq T}$, which is defined by for each $j=1,\ldots,d$,
	$
	B_t^{i,j}
	:=
	W_t^{i,j}
	-\int_0^t
	\mu^j (s,X^i)
	\rd s,
	$
	are $d$-dimensional standard Brownian motion on the probability space $(\Omega^{i}, \mathcal{F}^{i}_{T}, \mathbb{Q}_{k}^{i})$, and then, for  each $k\geq1$ and $i=1,2$, the process $(X_{t \wedge \tau_k^i}^i, B_{t \wedge \tau_k^i}^i)_{0 \leq t \leq T}$  is a solution of \eqref{SDE_2} under $\mathbb{Q}_k^{i}$. By the same way of Proposition 3.10 in \cite{KS} , the uniqueness in law for SDE \eqref{SDE_2} implies the uniqueness in law for SDE \eqref{SDE_1}.  
	
	
	Finally, we prove \eqref{Girsanov_0}.
	By the uniqueness in law of $X^{x}$, for any measurable functional $f:C([0,T];\real^d) \to \real$ such that the expectation $\e[f(Y^x) Z_T(1,Y^x)]$ exists, we have
	\begin{align*}
		\e[f(X^{x})]
		=
		\e_{\mathbb{Q}}[f(Y^{0,x})]
		=\e
		\left[
			f(Y^{0,x})
			\frac{\rd \mathbb{Q}}{\rd \mathbb{P}}
		\right]
		=\e
		\left[
			f(Y^{0,x})
			Z_T(1,Y^{0,x})
		\right],
	\end{align*}
	which concludes the proof.
\end{proof}

\section{PDF of a solution of SDEs with unbounded and path--dependent  drift}\label{Sec_pdf}

In this section, we show that the existence, representation, Gaussian two--sided bound and H\"older continuity for a pdf of a solution of SDEs with unbounded and  path--dependent drift coefficient.

\subsection{Existence and representations}\label{Sub_exist}

We obtain the existence and representations for a  pdf of a solution of SDE \eqref{SDE_1} under linear growth condition on $b$.

\begin{Thm}\label{main_1}
	Suppose Assumption \ref{Ass_1} holds.
	Then for any $(t,x) \in (0,T] \times \real^d $, $X_t^x$ admits a pdf, denoted by $p_t(x,\cdot)$, with respect to Lebesgue measure and it has the following representations
	\begin{align}
		p_t(x,y)
		&=q(0,x;t,y)
		+ \int_0^t
			\e\left[
				\langle
					\nabla_x q(s,X_s^x;t,y), b(s,X^x)
				\rangle
			\right]
		\rd s,
		~\text{a.e.},~y \in \real^d,\label{density_1}\\
		&=q(0,x;t,y)
		\e
		[
			Z_t(1,Y^{0,x})
			~|~Y_t^{0,x}=y
		],~\text{a.e.},~y \in \real^d\label{density_2},
	\end{align}
	where $\e[~\cdot~|~Y_t^{0,x}=y ]$ is the expectation of a regular conditional probability given $Y_t^{0,x}=y$ for $y \in \real^d$.
\end{Thm}

\begin{proof}
	We first show the second representation \eqref{density_2}.
	From Theorem \ref{main_0}, $Z(1,Y^{0,x})$ is a martingale, thus for any $f \in C_b^{\infty}(\real^d;\real)$, it holds that
	\begin{align*}
	\e[f(X_t^x)]
	=\e[f(Y^{0,x}_t) Z_T(1,Y^{0,x})]
	=\e[f(Y^{0,x}_t) \e[ Z_T(1,Y^{0,x})~|~\mathcal{F}_t]]
	=\e[f(Y^{0,x}_t) Z_t(1,Y^{0,x})].
	\end{align*}
	On the other hand, from Theorem 1.3.3 in \cite{IkWa}, there exists a regular conditional probability given $Y_t^{0,x}=y$ for $y \in \real^d$, denoted by $\p(~\cdot~|~Y_t^{0,x}=y)$, such that
	\begin{align*}
	\e[f(Y_t^{0,x}) Z_t(1,Y^{0,x})]
	=\int_{\real^d}
		f(y)
		\e
		[
			Z_t(1,Y^{0,x})
			~|~Y_t^{0,x}=y
		]
	\p(Y_t^{0,x} \in \rd y),
	\end{align*}
	where $\e[~\cdot~|~Y_t^{0,x}=y ]$ is the expectation with respect to $\p(~\cdot~|~Y_t^{0,x}=y)$.
	Therefore, it holds that for each fixed $(t,x) \in (0,T] \times \real^d$,
	\begin{align*}
		p_t(x,y)
		=q(0,x;t,y)
		\e
		[
			Z_t(1,Y^{0,x})
			~|~Y_t^{0,x}=y
		]
		\in [0,\infty)
		,~\text{a.e.},~y \in \real^d,
	\end{align*}
	which is the second representation \eqref{density_2}.
	
	Now we show the first representation \eqref{density_1}.
	It suffices to prove that for any $f \in C_b^{\infty}(\real^d; \real)$,
	\begin{align}\label{Ito_6}
		\e[f(X_t^{x})]
		=
		\int_{\real^d} f(y)
			\left\{
				q(0,x;t,y)
				+
				\int_0^t
					\e\left[ \langle \nabla_x q(s,X_s^x;t,y), b(s,X) \rangle \right]
				\rd s
			\right\}
		\rd y.
	\end{align}
	By the definition of $u(s,x;t)$, we have
	\begin{align} \label{expectation_2}
		\e\left[f\left(Y_t^{0,x} \right) \right]
		&=u(0,x;t),
		\quad
		\e\left[f(X_t^x)\right]
		=\e\left[u(t,X_t^x;t)\right].
	\end{align}
	By using It\^o's formula, it holds that for any $\varepsilon \in (0,t)$,
	\begin{align*}
		u(t-\varepsilon,X_{t-\varepsilon}^x;t)
		=&u(0,x;t)
		+\int_0^{t-\varepsilon} (\partial_s+L_s) u(s,X_s^x;t) \rd s 
		+\int_0^{t-\varepsilon} \langle \nabla_x u(s,X_s^x;t),b(s,X^x) \rangle \rd s \notag \\
		&+ \sum_{i,j=1}^{d}\int_0^{t-\varepsilon} \sigma_{i,j}(s,X_s^x) \partial_{x_i} u(s,X_s^x;t) \rd W_s^j.
	\end{align*}
	Since $u(s,x;t)$ is a solution to the heat equation \eqref{heat_eq_1}, it holds that
	\begin{align*}
		u(t-\varepsilon,X_{t-\varepsilon}^x;t)
		=&u(0,x;t)
		+\int_0^{t-\varepsilon} \langle \nabla_x u(s,X_s^x;t),b(s,X) \rangle \rd s \notag \\
		&+ \sum_{i,j=1}^{d}\int_0^{t-\varepsilon} \sigma_{i,j}(s,X_s^x) \partial_{x_i} u(s,X_s^x;t) \rd W_s^j.
	\end{align*}
	Since for $i=1,\ldots,d$, by using \eqref{bound_drev_2} and dominated convergence theorem,
	\begin{align}\label{deriva_u}
		\partial_{x_i} u(s,x;t)
		&=\partial_{x_i} \e\left[f\left(Y_{t}^{s,x} \right) \right]
		=\int_{\real^d} f(y) \frac{\partial}{\partial x_i} q(s,x;t,y)\rd y,
	\end{align}
	and it holds that for any $s \in [0,t)$ and $x \in \real^d$,
	\begin{align*}
		|\partial_{x_i} u(s,x;t)|
		\leq \frac{C\|f\|_{\infty}}{(t-s)^{1/2}},
	\end{align*}
	for some $C>0$.
	Therefore, the martingale property implies that the expectation of $\int_0^{t-\varepsilon} \sigma_{i,j}(s,X_s^x) \partial_{i} u(s,X_s^x;t) \rd W_s^j$ equals to zero, and by using Schwarz's inequality and Lemma \ref{moment_0} with $F=b$ and $p=1$,
	\begin{align}\label{main_1_drift}
		\e\left[
			\int_0^{t}
				\left|
					\left\langle
						\nabla_x u(s,X_s^x;t), b(s,X^x)
					\right \rangle
				\right|
			\rd s
		\right]
		&\leq
		\sqrt{d}C\|f\|_{\infty}
		\e\left[
			\int_0^{t}
				\frac{|b(s,X^x)|}{(t-s)^{1/2}}
			\rd s
		\right]
		<\infty.
	\end{align}
	Hence, by taking the expectation and Fubini's theorem, we have from \eqref{expectation_2} and \eqref{main_1_drift}
	\begin{align}\label{Ito_4}
		\e[u(t-\varepsilon,X_{t-\varepsilon}^{x};t)] 
		&=
		\e\left[f\left(Y_t^{0,x}\right) \right]
		+\int_0^{t-\varepsilon} \e\left[ \langle \nabla_{x} u(s,X_s^x;t),b(s,X^x) \rangle \right]\rd s.
	\end{align}
	Since $\sup_{(s,x) \in [0,t] \times \real^d} |u(s,x;t)| \leq  \|f\|_{\infty}$, by using the dominated convergence theorem and \eqref{main_1_drift}, \eqref{Ito_4},
	\begin{align*}
		\e[f(X_t^{x})]
		&=
		\e\left[
			\lim_{\varepsilon \to 0+} u(t-\varepsilon,X_{t-\varepsilon}^{x};t)
		\right]
		=
		\lim_{\varepsilon \to 0+}
			\e\left[
				u(t-\varepsilon,X_{t-\varepsilon}^{x};t)
			\right]\\
		&=\e\left[f\left(Y_t^{0,x}\right) \right]
		+\int_0^{t} \e\left[ \langle \nabla_{x} u(s,X_s^x;t),b(s,X^x) \rangle \right]\rd s.
	\end{align*}
	Finally, from Lemma \ref{moment_0} with $F=b$ and $p=1$,
	\begin{align*}
		&\sum_{i=1}^{d}
		\int_{0}^{t}
		\int_{\real^d}
			\e\left[
				|f(y)|
				\left|
					\frac{\partial}{\partial x_i} q(s,X_s^x;t,y)
				\right|
				|b^i(s,X^x)|
			\right]
		\rd y
		\rd s\\
		&\leq
		\sqrt{d}
		\widehat{C}_{+}\|f\|_{\infty}
		\int_{0}^{t}
		\int_{\real^d}
			\frac{1}{(t-s)^{1/2}}
			\e\left[
				g_{\widehat{c}_{+}(t-s)}(X_s^x,y)
				|b(s,X^{x})|
			\right]
		\rd y
		\rd s\\
		&=
		\sqrt{d}
		\widehat{C}_{+}\|f\|_{\infty}
		\int_{0}^{t}
		\frac{1}{(t-s)^{1/2}}
			\e\left[
				|b(s,X^{x})|
			\right]
		\rd s
		<\infty.
	\end{align*}
	Therefore, from \eqref{deriva_u} and Fubini's theorem we obtain \eqref{Ito_6}, which is the first representation \eqref{density_1}.
\end{proof}

\subsection{Gaussian two--sided bound and continuity of pdf}\label{Sub_GB}
In this subsection, we prove the Gaussian two--sided bound and continuity for a pdf of a solution of SDE \eqref{SDE_1} under the following {\it sub--linear growth condition} on the drift coefficient $b$.

\begin{Ass}\label{Ass_2}
	We suppose that the drift coefficient $b$ satisfies the following condition :
	for any $\delta,t>0$, there exists $K_t(\delta)>0$ such that $K_t(\delta)$ is increasing with respect to $t$
	and for all $t>0$ and $w \in C([0,t];\real^d)$,
	\begin {align*}
		|b(t,w)|
		\leq
		\delta | w^{*}_t|+K_t(\delta).
\end{align*}

\end{Ass}

\begin{Rem}
	\begin{itemize}
		\item[(i)]
		Let $f:\real^d \to \real^d$ be a measurable function and of sub--linear growth, that is, $f$ is bounded on any compact subset of $\real^d$ and $|f(x)|=o(|x|)$ as $|x| \to \infty$, which is equivalent to the condition that for any $\delta>0$, there exists a constant $K(\delta)>0$ such that $|f(x)|\leq \delta |x|+K(\delta)$.
		Therefore, if $b$ satisfies Assumption \ref{Ass_2}, then we say that $b$ is of sub--linear growth.
		
		
		
		
		\item[(ii)]
		Suppose that $b:[0,\infty) \times C([0,\infty);\real^d) \to \real^d$ satisfies the following growing condition: there exists $K>0$ and $\beta \in (0,1)$ such that
		\begin{align*}
			|b(t,w)|
			\leq
			K (1+|w^{*}_t|^{\beta}),
			~\text{ for all } (t,w) \in [0,\infty) \times C([0,\infty);\real^d).
		\end{align*}
		Then $b$ satisfies Assumption \ref{Ass_2} with $K_{t}(\delta)=K\{1+(K/\delta)^{\beta/(1-\beta)}\}$ for all $t>0$.
		Indeed it holds that
		\begin{align*}
			K (1+|w^{*}_t|^{\beta})
			\leq
			\left\{ \begin{array}{ll}
			\displaystyle
			\delta |w^{*}_t|+K
			 &\text{ if } |w_t^{*}|>(K/\delta)^{1/(1-\beta)},  \\
			\displaystyle
			K\{1+(K/\delta)^{\beta/(1-\beta)}\}
			&\text{ if } |w_t^{*}| \leq (K/\delta)^{1/(1-\beta)}.
			\end{array}\right.
		\end{align*}
	\end{itemize}
\end{Rem}

Under the sub--linear growth condition on $b$, we prove a Gaussian two--sided bound and a continuity for a pdf of $X_t^x$.

\begin{Thm}\label{main_2}
	Suppose Assumption \ref{Ass_1} and \ref{Ass_2} hold.
	Let $p_1,p_2,p_3>1$ with $p_1 \in (1,\frac{d}{d-1})$ and $1/p_1+1/p_2+1/p_3=1$.

	\begin{itemize}
		\item[(i)]
		For each $(t,x) \in (0,T] \times \real^d$, the right hand side of \eqref{density_1} is continuous with respect to $y$, that is, $p_t(x,\cdot)$ has a continuous version.
		
		\item[(ii)]
		There exist $C_{\pm}\equiv C_{\pm}(p_1) >0$ such that for any $(t,x) \in (0,T] \times \real^d$ and a.e. $y \in \real^d$, it holds that
		\begin{align*}
			p_t(x,y)
			&\geq
			\frac
				{C_{-}g_{2^{-1} \widehat{c}_{-} t}(x,y)}
				{
					1
					+
					\sup_{0 \leq s \leq t}
					\e
					\left[
						Z_s(1,Y^{0,x})^{-p_2}
					\right]^{1/p_2}
					\max_{i=1,2}
					\e
					\left[
						|b(s,Y^{0,x})|^{i p_3}
					\right]^{1/p_3}
					},
		\end{align*}
		and
		\begin{align*}
			p_t(x,y)
			&\leq
			C_{+}
			\left(
				1
				+
				\sup_{0 \leq s \leq t}
				\e
					\left[
						Z_s(1,Y^{0,x})^{p_2}
					\right]^{1/p_2}
				\max_{i=1,2}
				\e
				\left[
					|b(s,Y^{0,x})|^{i p_3}
				\right]^{1/p_3}
			\right)
			g_{p_1 \widehat{c}_{+} t}(x,y).
		\end{align*}
		
		\item[(iii)]
		Let $p_t(x,\cdot)$ be a continuous version of a pdf of $X_t^x$ for $(t,x) \in (0,T] \times \real^d$.
		For $r \in \real$, we define $t_r$ by
		\begin{align}
			t_r
			&:=\min\left\{
			T, \frac{1}{2K(b,T) \sqrt{3\underline{a} (2r^2-r) \widehat{c}_{+}}}
			\right\} \label{def_tr}.
		\end{align}
		Then there exist $C_{\pm}>0$ and $c_{\pm}>0$ such that for any $(t,x,y) \in (0,T] \times \real^d \times \real^d$, it holds that if $t \in (0,t_{-p_{2}}]$, then
		\begin{align*}
			\frac
				{C_{-} g_{2^{-1} \widehat{c}_{-} t}(x,y)}
				{(1+|x|^2) \exp\left(c_{-}(1+|x|^2)t\right)}
				\leq 
				p_t(x,y)
				\leq
				C_{+} (1+|x|^2) 
				\exp\left(c_{+}(1+|x|^2)t\right)
				g_{p_1 \widehat{c}_{+} t}(x,y),
		\end{align*}
		and if $t \in (t_{-p_{2}},t_{p_{2}}]$, then
		\begin{align*}
		\frac
		{C_{-}g_{2^{-1} \widehat{c}_{-}t}(x,y)}
		{
			(1+|x|^2)
			\exp
			\left(
			\frac
			{|x|^2}
			{8p_2 \widehat{c}_{+}T}
			\right)
		}
		\leq 
		p_t(x,y)
		\leq
		C_{+} (1+|x|^2) 
		\exp\left(c_{+}(1+|x|^2)t\right)
		g_{p_1 \widehat{c}_{+} t}(x,y),
		\end{align*}
		and if $t \in (t_{p_{2}},T]$, then
		\begin{align*}
			\frac
				{C_{-}g_{2^{-1} \widehat{c}_{-}t}(x,y)}
				{
					(1+|x|^2)
					\exp
						\left(
							\frac
								{|x|^2}
								{8p_2 \widehat{c}_{+}T}
						\right)
				}
		\leq 
		p_t(x,y)
		\leq
		C_{+}(1+|x|^2)
		\exp
			\left(
				\frac
				{|x|^2}
				{8p_2\widehat{c}_{+}T}
			\right)
		g_{p_1 \widehat{c}_{+}t}(x,y).
		\end{align*}
		
	\end{itemize}
\end{Thm}

\begin{Rem}
	Note that if $b$ is bounded, then from Theorem \ref{main_2} (ii) and \eqref{bdd_G_2_2} below, $p_t(x,y)$ satisfies the Gaussian two--sided bound uniformly with respect to $x \in \real^d$, that is, there exist $C_{\pm}>0$ and $c_{\pm}$ such that for any $x,y \in \real^d$ and $t \in (0,T]$,
	\begin{align}\label{unif_GB_0}
		C_{-} g_{c_{-}t}(x,y)
		\leq 
		p_t(x,y)
		&\leq
		C_{+} g_{c_{+}t}(x,y),
	\end{align}
	(see, also Theorem 2.5 in \cite{Ku17}).
	However, if $b$ is of sub--linear growth, then $C_{\pm}$ in \eqref{unif_GB_0} might be depend on the initial value $x\in \real^d$.
	Note that an Ornstein--Uhlenbeck process $\rd X_t=\kappa X_t \rd t+\rd W_t$, $X_0=x$, the law of $X_t$ admits a pdf.
	However, it does not satisfies the Gaussian two-sided bound uniformly in $x \in \real$, (see, section 6.2 in \cite{Ku17}).
\end{Rem}

Since $q(s,x;t,y)$ is continuous in $y \in \real^d$ and satisfies the Gaussian two--sided bound, in order to obtain continuity and two--sided bound of $p_t(x,y)$, we need to consider
\begin{align*}
	\int_0^t \e\left[ \langle \nabla_x q(s,X_s^x;t,y), b(s,X^x) \rangle \right] \rd s
	\quad\text{and}\quad
	\e[
		Z_t(1,Y^{0,x})
		~|~Y_t^{0,x}=y
	].
\end{align*}

We first introduce the following lemma which shows that the moment of the Maruyama--Girsanov density $Z_t(1,Y^{0,x})$ is finite.

\begin{Lem}\label{bdd_Girsanov_1}
	Suppose Assumption \ref{Ass_1} and \ref{Ass_2} hold.
	Recall that $t_r$ is defined by \eqref{def_tr} for $r \in \real$.
	For any $r  \in \real$, $(t,x) \in (0,T] \times \real^d$, it holds that
	\begin{align}\label{bdd_Girsanov_01}
		&\sup_{0 \leq s \leq t}
		\e[Z_s(1,Y^{0,x})^{r}] \notag\\
		&\leq
		\left\{ \begin{array}{ll}
		\displaystyle
		1, 
		&\text{ if } 2r^2-r \leq 0,~t \in (0,T], \\
		\displaystyle
			2^{1+d/4}
			\widehat{C}_{+}
			\exp\left(
			\frac{3}{2}
			K(b,T)^2
			\underline{a}
			(2r^2-r)
			t
			(1+|x|^2)
			\right),
		&\text{ if } 2r^2-r  > 0,~t \in (0,t_r],\\
		\displaystyle
			2^{1+d/4}
			\left(\frac{T}{t_r}\right)^{d/4}
			\widehat{C}_{+}^{1/2}
			\exp
			\left(
			\frac{3}{2}\underline{a} 
			(2r^2-r)
			|K_T(\delta_{r,T})|^2
			t
			\right)
			\exp\left(
			\frac{|x|^2}{8\widehat{c}_{+}T}
			\right),
		&\text{ if } 2r^2-r  > 0,~t \in (t_r,T],
			\end{array}\right.
	\end{align}
	where for $r \in \real$ and $t>0$,
	\begin{align*}
		\delta_{r,t}
		&:=
		\frac
		{1}
		{2t\sqrt{3\widehat{c}_{+} \underline{a}(2r^2-r)}}.
	\end{align*}
\end{Lem}
\begin{proof}
	For each $r \in \real$ and $s \in [0,t]$, by using Schwartz's inequality, we have
	\begin{align}\label{bdd_G_2_2}
		&
		\e[Z_s(1,Y^{0,x})^{r}] \notag\\
		&=
		\e
			\left[
				\exp
					\left(
						r \sum_{j=1}^{d}\int_0^s \mu^j(u,Y^{0,x}) \rd W_u^j
						- r^2 \int_0^s \left|\mu(u,Y^{0,x}) \right|^2 \rd u
						+(r^2-\frac{r}{2})\int_0^s \left|\mu(u,Y^{0,x})\right|^2 \rd u
					\right)
			\right] \notag\\
		&\leq
		\e\left[Z_s(2r,Y^{0,x})\right]^{1/2}
		\e
		\left[
		\exp\left( (2r^2-r)\int_0^s \left|\mu(u,Y^{0,x}) \right|^2 \rd u\right)
		\right]^{1/2}.
	\end{align}
	From Theorem \ref{main_0}, $Z(2r,Y^{0,x})$ is martingale, thus $\e\left[Z_s(2r,Y^{0,x})\right]=1$.
	If $2r^2-r \leq 0$, then \eqref{bdd_G_2_2} is bounded by $1$.
	
	Now we assume that $2r^2-r>0$.
	For $t \in (0,t_r]$, by using a linear growth condition on $b$, we have
	\begin{align*}
	\sup_{0 \leq s \leq t}
	\e[Z_s(1,Y^{0,x})^{r}]
	\leq
	\e
	\left[
	\exp
	\left(
	3K(b,T)^2\underline{a} 
	(2r^2-r)t
	(1+|x|^2+|M_{t}^{x,*}|^2)
	\right)
	\right]^{1/2}.
	\end{align*}
	Since the map $z \mapsto \exp(\frac{3}{2}K(b,T)^2\underline{a} (2r^2-r)t (1+|x|^2+|z|^2))$ is a convex,
	\begin{align*}
	U_t^x:=\exp\left(\frac{3}{2}K(b,T)^2\underline{a} (2r^2-r)t(1+|x|^2+|M_{t}^{x}|^2)\right)
	\end{align*}
	is a sub-martingale, and using Doob's maximal inequality, we have
	\begin{align*}
	\e
	\left[
	\exp
	\left(
	3K(b,T)^2\underline{a} 
	(2r^2-r)t
	(1+|x|^2+|M_{t}^{x,*}|^2)
	\right)
	\right]
	=\e[|U_t^{x,*}|^2]
	\leq 4\e[|U_t^x|^2].
	\end{align*}
	Hence it follows from the Gaussian upper bound \eqref{bound_qt} that
	\begin{align*}
	\e[|U_t^x|^2]
	&\leq
	\widehat{C}_{+}
	\exp\left(3K(b,T)^2\underline{a} (2r^2-r)t(1+|x|^2)\right)
	\int_{\real^d}
	\exp
	\left(
	3K(b,T)^2\underline{a} (2r^2-r)t|y-x|^2
	\right)
	g_{\widehat{c}_{+}t}(x,y)
	\rd y\\
	&=
	2^{d/2}\widehat{C}_{+}
	\exp\left(3K(b,T)^2\underline{a} (2r^2-r)t(1+|x|^2)\right)\\
	&\quad \times
	\int_{\real^d}
	\exp
	\left(
	\left\{
	3K(b,T)^2\underline{a} (2r^2-r)t
	-\frac
	{1}
	{4\widehat{c}_{+}t}
	\right\}
	|y-x|^2
	\right)
	g_{2\widehat{c}_{+}t}(x,y)
	\rd y.
	\end{align*}
	Therefore, since $t \leq t_r$ if and only if $3K(b,T)^2\underline{a} (2r^2-r)t-\frac{1}{4\widehat{c}_{+}t}\leq 0$, we obtain the statement for $t \in (0,t_r]$.
	
	For $t \in (t_r,T]$, using Assumption \ref{Ass_1} (iii), Assumption \ref{Ass_2} and \eqref{bdd_G_2_2}, we have
	\begin{align*}
		\sup_{0 \leq s \leq t}
		\e[Z_s(1,Y^{0,x})^{r}]
		&\leq
		\e
		\left[
		\exp\left( (2r^2-r)\int_0^t \left|\mu(s,Y^{0,x}) \right|^2 \rd s\right)
		\right]^{1/2}\\
		&\leq
		\e
		\left[
			\exp
			\left(
				3\underline{a} 
				(2r^2-r)
				\int_0^t
					\delta^2(|x|^2+|M_{s}^{x,*}|^2)
					+
					\left|K_T(\delta)\right|^2
				\rd s
			\right)
		\right]^{1/2}\\
		&\leq
		\exp\left(
			\frac{3}{2}\underline{a} 
			(2r^2-r)
			|K_T(\delta)|^2 t
		\right)
		\e
		\left[
			\exp
			\left(
				3\underline{a} 
				(2r^2-r)
				\delta^2
				t
				(|x|^2+|M_{t}^{x,*}|^2)
			\right)
		\right]^{1/2}.
	\end{align*}
	Note that the map
	$
		z
		\mapsto
		\exp\left(
			\frac{3}{2}\underline{a}(2r^2-r)\delta^{2} t( |x|^2+|z|^{2})
		\right)
	$
	is convex, thus 
	\begin{align*}
		V_t^x
		:=
		\exp\left(
		\frac{3}{2}\underline{a}(2r^2-r)\delta^{2} t( |x|^2+|M_t^{x}|^{2})
		\right)
	\end{align*}
	is a sub-martingale, and using Doob's maximal inequality, we have
	\begin{align*}
		\e
		\left[
		\exp\left(3\underline{a} 
		(2r^2-r) \delta^{2} t\left( \left|x\right|^2+\left| M_{t}^{x,*} \right|^2\right)
		\right)
		\right]
		=
		\e\left[
			|V^{x,*}_t|^2
		\right]
		\leq 4 \e\left[|V_t^x|^2\right].
	\end{align*}
	Recall that $M_t^{x}=Y_t^{0,x}-x$ and $Y_t^{0,x}$ has the pdf which satisfies the Gaussian upper bound \eqref{bound_qt}, we have
	\begin{align*}
		\e[|V_t^x|^2]
		&\leq
		\left(\frac{T}{t_r}\right)^{d/2}
		\widehat{C}_{+}
		\int_{\real^d}
			\exp\left(
				3\underline{a}(2r^2-r)\delta^{2} T(|x|^2+|y-x|^{2})
			\right)
			g_{\widehat{c}_{+}T}(x,y)
		\rd y\\
		&=
		\left(\frac{2T}{t_r}\right)^{d/2}
		\widehat{C}_{+}
		\exp\left(
			3\underline{a}(2r^2-r)\delta^{2} T |x|^2
		\right)\\
		&\quad \times
		\int_{\real^d}
		\exp\left(
			\left\{
			3\underline{a}(2r^2-r)\delta^{2} T
			-\frac{1}{4\widehat{c}_{+}T}
			\right\}
			|y-x|^{2}
		\right)
		g_{2\widehat{c}_{+}T}(x,y)
		\rd y.
	\end{align*}
	By choosing $\delta=\delta_{r,T}$, $3\underline{a}(2r^2-r)\delta^{2} T
	-\frac{1}{4\widehat{c}_{+}T} = 0$, thus we obtain the statement for $t \in (t_r,T]$.
\end{proof}

\begin{Rem}
	Note that the sub--linear growth condition is necessary in order to show the $q$-th moment of $Z_t(1,Y^{0,x})$, (see, e.g. Remark 3.3 in \cite{NT2}).
\end{Rem}

The following lemma is useful for proving a Gaussian two--sided bound.

\begin{Lem}\label{key_1}
	Let $t \in (0,T]$, $p>1$ and $p_1,p_2,p_2>1$ with $1/p_1+1/p_2+1/p_3=1$.
	Suppose that Assumption \ref{Ass_1} and \ref{Ass_2} hold and $F:[0,\infty) \times C([0,\infty);\real^d) \to \real^d$ satisfies assumptions in Lemma \ref{moment_0}.
	\begin{itemize}
		\item[(i)]
		For any $(s,x,y) \in [0,t) \times \real^d \times \real^d$ it holds that
		\begin{align}\label{Girsanov_3}
			&
			\e
				\left[
					\left|
						\nabla_x q(s,X_s^x;t,y)
					\right|
					\left|
						F(s,X^x)
					\right|
				\right] \notag\\
			&\leq
			\frac
				{
					\sqrt{d}
					\widehat{C}_{+}
					C_{0,\sigma}(p_3,F,T)
					(1+|x|)
					\e
					\left[
						Z_t(1,Y^{0,x})^{p_2}
					\right]^{1/p_2}
				}
				{\sqrt{t-s}}
			\e\left[
				\left|g_{\widehat{c}_+(t-s)}(Y_s^{0,x},y) \right|^{p_1}
			\right]^{1/{p_1}}.
		\end{align}
		In particular, if $F$ is bounded and $p_3 =\infty$,
		\begin{align*}
			\e
				\left[
					\left|
						\nabla_x q(s,X_s^x;t,y)
					\right|
					\left|
						F(s,X^x)
					\right|
				\right]
			\leq
			\frac{\sqrt{d}\widehat{C}_{+}\|F\|_{\infty}}{\sqrt{t-s}}
			\e
			\left[
			\left|g_{\widehat{c}_+(t-s)}(Y^{0,x}_s,y)\right|^{p_1}
			\right]^{1/p_1}
			\e
			\left[
			Z_t(1,Y^{0,x})^{p_2}
			\right]^{1/p_2}.
		\end{align*}
		
		\item[(ii)]
		For any $c>0$ and $(s,x,y) \in [0,t) \times \real^d \times \real^d$, it holds that
		\begin{align}\label{esti_1}
			\e\left[\left|g_{c(t-s)}(Y^{0,x}_s,y) \right|^{p}\right]^{1/p}
			\leq
			\frac{\widehat{C}_{+}^{\frac{1}{p}}}{c^{\frac{d(p-1)}{2p}}}
			\frac{\{p(c \vee \widehat{c}_{+})\}^{\frac{d}{2}}}{(c \wedge \widehat{c}_{+})^{\frac{d}{2p}}}
			\left(
				\frac
					{t}
					{t-s}
			\right)^{\frac{d(p-1)}{2p}}
			g_{p(c \vee \widehat{c}_+) t}(x,y).
		\end{align}
			
	\end{itemize}
	
\end{Lem}
\begin{proof}
	(i).
	For any $s \in [0,t)$, using Theorem \ref{main_0}, H\"older's inequality and \eqref{bound_drev_2}, we have,
	\begin{align}\label{Girsanov_2}
		&
		\e
			\left[
				\left|
					\nabla_x q(s,X_s^{x};t,y)
				\right|
				\left|
					F(s,X^{x})
				\right|
			\right]
		=
		\e\left[
			|\nabla_x q(s,Y_s^{0,x};t,y)|
			|F(s,Y^{0,x})|
			Z_t(1,Y^{0,x})
		\right] \notag\\
		&\leq
		\frac{\sqrt{d}\widehat{C}_{+}}{\sqrt{t-s}}
			\e\left[
				g_{\widehat{c}_+(t-s)}(Y^{0,x}_s,y)
				|F(s,Y^{0,x})|
				Z_t(1,Y^{0,x})
			\right]
		\notag\\
		&\leq 
			\frac{\sqrt{d}\widehat{C}_{+}}{\sqrt{t-s}}
			\e
			\left[
				\left|g_{\widehat{c}_+(t-s)}(Y^{0,x}_s,y)\right|^{p_1}
			\right]^{1/p_1}
			\e
			\left[
				Z_t(1,Y^{0,x})^{p_2}
			\right]^{1/p_2}
			\e[|F(s,Y^{0,x})|^{p_3}]^{1/p_3}.
	\end{align}
	By using Lemma \ref{moment_0} with $b\equiv 0$,
	we conclude the proof of (i).
	
	(ii)
	By using the upper bounds \eqref{bound_qt} and Chapman--Kolmogorov equation, it holds that
	\begin{align*}
		&\e\left[
			\left|
					g_{c(t-s)}(Y^{0,x}_s,y)
			\right|^p
		\right]
		=\int_{\real^d}
			q(0,x;s,z)
			\left|
				g_{c(t-s)}(z,y)
			\right|^p
		\rd z
		\leq
		\widehat{C}_+
		\int_{\real^d}
			g_{\widehat{c}_+ s}(x,z)
			\left|
				g_{c(t-s)}(z,y)
			\right|^p
		\rd z\\
		&\leq
			\left(
				\frac
					{c \vee \widehat{c}_{+}}
					{c \wedge \widehat{c}_{+}}
			\right)^{d/2}
			\frac
				{\widehat{C}_+}
				{\{2\pi c (t-s)\}^{\frac{d(p-1)}{2}}}
			\int_{\real^d}
				g_{(c \vee \widehat{c}_+) s}(x,z)
				g_{(c \vee \widehat{c}_+)(t-s)}(z,y)
			\rd z\\
		&=
		\left(
			\frac
				{c \vee \widehat{c}_{+}}
				{c \wedge \widehat{c}_{+}}
		\right)^{d/2}
		\frac
			{\widehat{C}_+}
			{\{2\pi c (t-s)\}^{\frac{d(p-1)}{2}}}
		g_{(c \vee \widehat{c}_+) t}(x,y),
	\end{align*}
	which concludes the statement.
\end{proof}

For the proof of the Gaussian two--sided bound, we need the following lemma, which is an analogy of Lemma 2.3 in \cite{Ku17}.

\begin{Lem}\label{lower_0}
	Let $r \in \real$ and $p_1,p_2,p_3>1$ with $p_1 \in (1,\frac{d}{d-1})$, $1/p_1+1/p_2+1/p_3=1$.
	Suppose Assumption \ref{Ass_1} and \ref{Ass_2} hold.
	Then there exists $C_{r,p_1}>0$ such that for all $(t,x,y) \in (0,T] \times \real^d \times \real^d$,
	\begin{align*}
		&\sup_{0 \leq s <t}
		\e
		\left[
			q(s;Y_s^{0,x};t,y)
			Z_s(1,Y^{0,x})^{r}
		\right] \notag\\
		&\leq
		\widehat{C}_{+} g_{\widehat{c}_{+}t}(x,y)
		+
		C_{r,p_1}
		\sup_{0 \leq s \leq t}
		\e
			\left[
				Z_s(1,Y^{0,x})^{rp_2}
			\right]^{1/p_2}
		\max_{i=1,2}
		\e
		\left[
			b(s,Y^{0,x})^{i p_3}
		\right]^{1/p_3}
		 g_{p_1 \widehat{c}_{+} t}(x,y).
		\end{align*}
\end{Lem}
\begin{proof}
	Let $s \in [0,t)$.
	By It\^o's formula, $Z(1,Y^x)^{r}$ satisfies the following linear SDE
	\begin{align*}
		Z_s(1,Y^x)^{r}
		=1
		+\frac{r(r-1)}{2}
			\int_{0}^{s}
				|\mu(r,Y^{0,x})|^2
				Z_u(1,Y^{0,x})^{r}
			\rd u
		+r \sum_{j=1}^{d}
			\int_{0}^{s}
				\mu^j(u,Y^{0,x})
				Z_u(1,Y^{0,x})^{r}
			\rd W_u^j
	\end{align*}
	and by using \eqref{pde_fund_0}, we have
	\begin{align*}
		q(s,Y_s^{0,x};t,y)
		=q(0,x;t,y)
		+\sum_{i,j=1}^{d}
			\int_{0}^{s}
				\partial_{x_i} q (u,Y_u^{0,x};t,y)
				\sigma_{i,j}(u,Y_u^{0,x})
			\rd W_u^j.
	\end{align*}
	Hence by using integration by parts formula it holds that
	\begin{align*}
		q(s,Y_s^{0,x};t,y)Z_s(1,Y^{0,x})^{r}
		&=q(0,x;t,y)
		+
		r 
		\int_{0}^{s}
			\langle
				\nabla_x q(u,Y_u^{0,x};t,y), b(u,Y^{0,x})
			\rangle
			Z_u(1,Y^{0,x})^{r}
		\rd u\\
		&\quad
		+
		\frac{r(r-1)}{2}
		\int_{0}^{s}
			q(u,Y_u^{0,x};t,y)
			|\mu(u,Y^{0,x})|^2
			Z_u(1,Y^{0,x})^{r}
		\rd u\\
		&\quad
		+
		M^{1}_s
		+
		M^{2}_s,
	\end{align*}
	where
	\begin{align*}
		M^{1}_s
		&:=
		\sum_{i,j=1}^{d}
		\int_{0}^{s}
			Z_u(1,Y^{0,x})^{r}
			\partial_{x_i} q (u,Y_u^{0,x};t,y)
			\sigma_{i,j}(u,Y_u^{0,x})
		\rd W_u^j,\\
		M^{2}_s
		&:=
		r \sum_{j=1}^{d}
		\int_{0}^{s}
			q(u,Y_u^{0,x};t,y)
			\mu^j(u,Y^{0,x})
			Z_u(1,Y^{0,x})^{r}
		\rd W_u^j.
	\end{align*}
	By taking expectation, it holds that
	\begin{align*}
		\e[q(s,Y_s^{0,x};t,y)Z_s(1,Y^{0,x})^{r}]
		&=q(0,x;t,y)
		+r \e\left[
			\int_{0}^{s}
				\langle
					\nabla_x q(u,Y_u^{0,x};t,y), b(u,Y^{0,x})
				\rangle
				Z_u(1,Y^{0,x})^{r}
			\rd u
		\right]\\
		&\quad
			+\frac{r(r-1)}{2}
				\e\left[
					\int_{0}^{s}
						q(u,Y_u^{0,x};t,y)
						|\mu(u,Y^{0,x})|^2
						Z_u(1,Y^{0,x})^{r}
					\rd u
				\right],
	\end{align*}
	where we use the fact that the expectations of $M^{1}_{s}$ and $M^{2}_{s}$ are zero.
	Indeed, since $s \in [0,t)$, there exists $t_{0} \in [0,t)$ such that $s \leq t_{0}<t$.
	By using the moment estimate on $Z_u(1,Y^{0,x})^r$ (see, Lemma \ref{bdd_Girsanov_1}), and the upper bound for $\partial_{x_i} q(u,x;t,y)$ (see, \eqref{bound_drev_2}),
	\begin{align*}
		\e\left[
			\langle M^{1} \rangle_s
		\right]
		&\leq
		\frac
			{
				d
				\overline{a}
				\widehat{C}_{+}^{2}
			}
			{
				(2\pi \widehat{c}_{+})^{d}
			}
		\int_{0}^{t_{0}}
			\frac{\e\left[
				Z_u(1,Y^{0,x})^{2r}
				\right]}{(t-u)^{d+1}}
		\rd u
		\leq
		\frac
		{
			d
			\overline{a}
			\widehat{C}_{+}^{2}
		}
		{
			(2\pi \widehat{c}_{+})^{d}
		}
		\sup_{0 \leq u \leq t_{0}}
		\e\left[
			Z_u(1,Y^{0,x})^{2r}
		\right]
		\int_{0}^{t_{0}}
			\frac{1}{(t-u)^{d+1}}
		\rd u
		<\infty,
	\end{align*}
	and the moment estimate on $b(u,Y^{0,x})$ (see, Lemma \ref{moment_0}), the upper bound for $q(u,x;t,y)$ (see, \eqref{bound_qt}) and Schwarz's inequality
	\begin{align*}
		\e\left[
			\langle M^{2} \rangle_s
		\right]
		&\leq
		\frac{r^{2}\widehat{C}_{+}^{2}}{(2\pi \widehat{c}_{+})^{d}}
		\int_{0}^{t_{0}}
			\frac
				{
					\e\left[
						|\mu(u,Y^{0,x})|^{2}
						Z_u(1,Y^{0,x})^{2r}
					\right]
				}
				{
					(t-u)^{d}
				}
		\rd u\\
		&\leq
		\frac{r^{2}\widehat{C}_{+}^{2}}{(2\pi \widehat{c}_{+})^{d}}
		\sup_{0 \leq u \leq t_{0}}
		\e\left[
			|\mu(u,Y^{0,x})|^{4}
			\right]^{1/2}
		\sup_{0 \leq u \leq t_{0}}
			\e\left[
				Z_u(1,Y^{0,x})^{4r}
		\right]^{1/2}
		\int_{0}^{t_{0}}
			\frac{1}{(t-u)^{d}}
		\rd u
		<\infty.
	\end{align*}
	Hence $M^{1}=(M^{1}_s)_{s \in [0,t_{0}]}$ and $M^{2}=(M^{2}_s)_{s \in [0,t_{0}]}$ are martingale, and thus the expectations of $M^{1}_s$ and $M^{2}_s$ are zero.
	By using \eqref{bound_qt}, Schwarz's inequality and H\"older's inequality with $1/p_1+1/p_2+1/p_3=1$, we have
	\begin{align*}
		&\e[q(s,Y_s^{0,x};t,y)Z_s(1,Y^{0,x})^{r}]\\
		&\leq
		\widehat{C}_{+}g_{\widehat{c}_{+}t}(x,y)
		+\int_{0}^{t}
			\frac
				{r \sqrt{d}\widehat{C}_{+}}
				{\sqrt{t-u}}
			\e\left[
				g_{\widehat{c}_{+}(t-u)}(Y_u^{0,x},y)
				|b(u,Y^x)|
				Z_u(1,Y^x)^{r}
			\right]
		\rd u\\
		&\quad
		+\frac{r(r-1)\widehat{C}_{+} \underline{a}}{2}\int_{0}^{t}
				\e\left[
					g_{\widehat{c}_{+}(t-u)}(Y_u^{0,x},y)
					|b(u,Y^{0,x})|^2
				Z_u(1,Y^{0,x})^{r}
			\right]
		\rd u\\
		&\leq 
		\widehat{C}_{+}g_{\widehat{c}_{+}t}(x,y)
		+\int_{0}^{t}
			\frac
				{r \sqrt{d}\widehat{C}_{+}}
				{\sqrt{t-u}}
			\e\left[
				g_{\widehat{c}_{+}(t-u)}(Y_u^{0,x},y)^{p_1}
			\right]^{1/p_{1}}
			\e\left[
				Z_u(1,Y^{0,x})^{rp_2}
			\right]^{1/p_2}
			\e\left[
				|b(u,Y^{0,x})|^{p_3}
			\right]^{1/p_3}
		\rd u\\
		&\quad
		+\frac{r(r-1)\widehat{C}_{+} \underline{a}}{2}
		\int_{0}^{t}
			\e\left[
				g_{\widehat{c}_{+}(t-u)}(Y_u^{0,x},y)^{p_1}
			\right]^{1/p_1}
			\e\left[
				Z_u(1,Y^{0,x})^{rp_2}
			\right]^{1/p_2}
			\e\left[
				|b(u,Y^{0,x})|^{2p_3}
			\right]^{1/p_3}
		\rd u.
	\end{align*}
	Therefore, it follows from Lemma \ref{moment_0}, Lemma \ref{bdd_Girsanov_1} and \eqref{esti_1} that
	\begin{align*}
		&\e[q(s,Y_s^{0,x};t,y)Z_s(1,Y^{0,x})^{r}]\\
		&\leq
		\widehat{C}_{+}g_{\widehat{c}_{+}t}(x,y)
		+
		C
		\sup_{0 \leq s \leq t}
		\left(
			\e\left[
				Z_s(1,Y^{0,x})^{rp_2}
			\right]^{1/p_2}
			\max_{i=1,2}
			\e
			\left[
				b(s,Y^{0,x})^{i p_3}
			\right]^{1/p_3}
		\right)\\
		&\quad\quad\quad\quad\quad\quad\quad\quad
		\times
		\int_0^t
			\left\{
				\frac
				{1}
				{(t-u)^{\frac{d(p_1-1)}{2p_1}+\frac{1}{2}}}
				+
				\frac
				{1}
				{(t-u)^{\frac{d(p_1-1)}{2p_1}}}
			\right\}
		\rd u
		g_{p_1 \widehat{c}_{+}t}(x,y),
	\end{align*}
	for some $C>0$.
	Since $p_1 \in (1,\frac{d}{d-1})$ implies $\frac{d(p_1-1)+p_1}{2p_1}<1$, we conclude the statement.
\end{proof}

\begin{proof}[Proof of Theorem \ref{main_2}]
	(Continuity).
	By the construction of the fundamental solution of parabolic type PDE, $q(0,x;t,\cdot)$ is continuous (see \cite{Fr64}), thus it is suffices to prove that for a given $y_0 \in \real^d$,
	\begin{align}\label{cont_0}
		\int_0^t
			\e\left[
				\langle
					\nabla_x q(s,X_s^x;t,y_0), b(s,X^x)
				\rangle
			\right]
		\rd s
		=
		\lim_{y \to y_0}
		\int_0^t
			\e\left[
				\langle
					\nabla_x q(s,X_s^x;t,y), b(s,X^x)
				\rangle
			\right]
		\rd s.
	\end{align}
	
	We first show that for each $s \in [0,t)$,
	\begin{align}\label{cont_1}
		\e\left[
			\langle
				\nabla_x q(s,X_s^{x};t,y_0), b(s,X^{x})
			\rangle
		\right]
		=
		\lim_{y \to y_0}
		\e\left[
			\langle
				\nabla_x q(s,X_s^{x};t,y), b(s,X^{x})
			\rangle
		\right].
	\end{align}
	For any $(s,y) \in [0,t) \times \real^d$, by using \eqref{bound_drev_2}, we have
	\begin{align*}
		\left|\langle
			\nabla_x q(s,X_s^{x};t,y), b(s,X^{x})
		\rangle\right|
		\leq
		\frac{\sqrt{d}\widehat{C}_{+}|b(s,X^{x})|}{(2\pi \widehat{c}_{+})^{\frac{d}{2}}(t-s)^{\frac{d+1}{2}}}.
	\end{align*}
	Since $\nabla_{x} q(0,x;t,\cdot)$ is continuous (e.g., page 20 of \cite{Fr64}), from Lemma \ref{moment_0} and dominated convergence theorem, we obtain \eqref{cont_1}.
	
	Let $p_1,p_2,p_3>1$ with $p_1 \in (1,\frac{d}{d-1})$, $1/p_1+1/p_2+1/p_3=1$.
	By using \eqref{Girsanov_3} with $F\equiv b$ and \eqref{esti_1} in Lemma \ref{key_1}, we have
	\begin{align*}
		&\int_0^t
			\sup_{y \in \real^d}
			\e\left[
				|\langle
					\nabla_x q(s,X_s^{x};t,y), b(s,X^{x})
				\rangle|
			\right]
		\rd s \notag\\
		&\leq
		C_{p_3} (1+|x|)
		\sup_{0 \leq s \leq t}
		\e[Z_s(1,Y^{0,x})^{p_2}]^{1/p_{2}}
		\int_0^t
			\frac{1}
			{(t-s)^{\frac{d(p_1-1)}{2p_1}+\frac{1}{2}}}
		\rd s
		\sup_{y \in \real^d}
		g_{p_1\widehat{c}_{+}t}(x,y)
		<\infty,
	\end{align*}
	for some $C_{p_3}>0$.
	Thus again using dominated convergence theorem, we conclude \eqref{cont_0}.
	
	
	(Upper and lower bound).
	The proof of upper and lower bound are based on \cite{Ku17} and the second representation \eqref{density_2}.
	By using Schwarz's inequality, it holds that
	\begin{align*}
		1
		&=
		\e
		\left[
			Z_t(1,Y^{0,x})^{1/2}
			Z_t(1,Y^{0,x})^{-1/2}
		~\middle|~Y_t^{0,x}=y
		\right]^2\\
		&\leq
		\e
			\left[
				Z_t(1,Y^{0,x})
			~\middle|~Y_t^{0,x}=y
		\right]
		\e
		\left[
			Z_t(1,Y^{0,x})^{-1}
		~\middle|~Y_t^{0,x}=y
		\right]
		\leq \infty,
		~\text{a.e.},~y \in \real^d,
	\end{align*}
	which implies that
	\begin{align*}
		0
		\leq
		\frac{1}
		{
			\e
			\left[
			Z_t(1,Y^{0,x})^{-1}
			~\middle|~Y_t^{0,x}=y
			\right]
		}
		\leq
		\e
		\left[
			Z_t(1,Y^{0,x})
			~\middle|~Y_t^{0,x}=y
		\right],
		~\text{a.e.},~y \in \real^d.
	\end{align*}
	Therefore, from \eqref{density_2}, we have
	\begin{align}\label{lower_main_0}
		p_t(x,y)
		\geq 
		\frac
		{
			q(0,x;t,y)^2
		}
		{
			q(0,x;t,y)
			\e
			\left[
				Z_t(1,Y^{0,x})^{-1}
				~\middle|~Y_t^{0,x}=y
			\right]
		}
		\geq 0
		,
		~\text{a.e.},~y \in \real^d.
	\end{align}
	
	Now we show that for any $r \in \real$ and $s \in [0,t)$,
	\begin{align}\label{low_pr_1}
		q(0,x;t,y)
		\e
		\left[
		Z_s(1,Y^{0,x})^{r}
		~\middle|~Y_t^{0,x}=y
		\right]
		=
		\e
		\left[
		q(s,Y_s^{0,x};t,y)
		Z_s(1,Y^{0,x})^{r}
		\right],
		~\text{a.e.},~y \in \real^d.
	\end{align}
	It is sufficient to show that for any $A \in \mathcal{F}_s$, 
	\begin{align}\label{low_pr_2}
		q(0,x;t,y)
		\p
			\left(A
			~\middle|~Y_t^{0,x}=y
			\right)
		=
		\e
		\left[
			q(s,Y_s^{0,x};t,y)
		\1_{A}
		\right],
		~\text{a.e.},~y \in \real^d.
	\end{align}
	From Theorem 1.3.3 in \cite{IkWa}, we have for any $f \in C_b^{\infty}(\real^d;\real)$,
	\begin{align*}
		\int_{\real^d}
			f(y)
			\p
				\left(
					A~\middle|~Y_t^{0,x}=y
				\right)
			q(0,x;t,y)
		\rd y
		=\int_{\real^d}
			f(y)
			\p(A \cap \{Y_t^{0,x} \in \rd y\}).
	\end{align*}
	Using the Markov property of $Y^{0,x}$, for any $B \in \mathcal{B}(\real^d)$,
	\begin{align*}
		\p(A \cap \{Y_t^{0,x} \in B\})
		&=\e\left[
			\1_A
			\p\left(
				Y_t^{0,x} \in B
				~|~\mathcal{F}_s
			\right)
		\right]
		=
		\int_{B}
			\e\left[
				\1_A
				q(s,Y_s^{0,x};t,y)
			\right]
		\rd y.
	\end{align*}
	Hence, we have
	\begin{align*}
		\int_{\real^d}
			f(y)
			\p
			\left(
				A~\middle|~Y_t^{0,x}=y
			\right)
			q(0,x;t,y)
		\rd y
		=\int_{\real^d}
			f(y)
			\e\left[
				\1_A
				q(s;Y_s^{0,x};t,y)
			\right]
		\rd y.
	\end{align*}
	This concludes \eqref{low_pr_2}.
	
	Applying Lemma \ref{lower_0}, Fatou's lemma and \eqref{low_pr_1}, we have for any $r \in \real$,
	\begin{align}
		&q(0,x;t,y)
		\e
		\left[
			Z_{t}(1,Y^{0,x})^{r}
			~\middle|~Y_t^{0,x}=y
		\right]
		\leq
		q(0,x;t,y)
		\liminf_{s \to 0}
		\e
		\left[
			Z_{t-s}(1,Y^{0,x})^{r}
			~\middle|~Y_t^{0,x}=y
		\right] \notag\\
		&\leq
		q(0,x;t,y) \sup_{0 \leq s <t}
		\e
		\left[
			Z_s(1,Y^{0,x})^{r}
		~\middle|~Y_t^{0,x}=y
		\right]
		=
		\sup_{0 \leq s <t}
		\e
		\left[
			q(s,Y_s^{0,x};t,y)
			Z_s(1,Y^{0,x})^{r}
		\right] \notag\\
		&\leq
		\widehat{C}_{+} g_{\widehat{c}_{+}t}(x,y)
		+
		C_{r,p_1}
		\sup_{0 \leq s \leq t}
		\e
			\left[
				Z_s(1,Y^{0,x})^{rp_2}
			\right]^{1/p_2}
		\max_{i=1,2}
		\e
		\left[
			|b(s,Y^{0,x})|^{i p_3}
		\right]^{1/p_3}
		g_{p_1 \widehat{c}_{+} t}(x,y)
		\label{low_pr_3-}\\
		&\leq
		\frac
		{
			\widehat{C}_{+}
			+
			C_{r,p_1}
			\sup_{0 \leq s \leq t}
			\e
			\left[
			Z_s(1,Y^{0,x})^{rp_2}
			\right]^{1/p_2}
			\max_{i=1,2}
			\e
			\left[
				|b(s,Y^{0,x})|^{i p_3}
			\right]^{1/p_3}
		}
		{(2\widehat{c}_{+} t)^{d/2}}
		<\infty,\label{low_pr_3}
	\end{align}
	thus, $Z_t(1,Y^x)^{r}$ is $L^1$ integrable with respect to the expectation $\e
	[~\cdot~|~Y_t^{0,x}=y ]$ for any $r \in \real$.
	Hence \eqref{low_pr_3-} with $r=1$, we have
	\begin{align*}
		p_t(x,y)
		\leq
		\widehat{C}_{+} g_{\widehat{c}_{+}t}(x,y)
		+
		C_{1,p_1}
		\sup_{0 \leq s \leq t}
		\e
		\left[
		Z_s(1,Y^{0,x})^{p_2}
		\right]^{1/p_2}
		\max_{i=1,2}
		\e
		\left[
			|b(s,Y^{0,x})|^{i p_3}
		\right]^{1/p_3}
		g_{p_1 \widehat{c}_{+} t}(x,y),
	\end{align*}
	a.e., $y \in \real^d$.
	Moreover, the Gaussian lower bound \eqref{bound_qt} for $q(0,x;t,y)$, and estimations \eqref{lower_main_0}, \eqref{low_pr_3} with $r=-1$, we obtain
	\begin{align*}
		p_t(x,y)
		&\geq 
		\frac
		{(2\pi \widehat{c}_{+}t)^{d/2}}
		{
			\widehat{C}_{+}
			+
			C_{-1,p_1}
			\sup_{0 \leq s \leq t}
			\e
			\left[
			Z_s(1,Y^{0,x})^{-p_2}
			\right]^{1/p_2}
			\max_{i=1,2}
			\e
			\left[
				|b(s,Y^{0,x})|^{i p_3}
			\right]^{1/p_3}
		}
		\frac{\widehat{C}_{-}^2\exp\left(
				-\frac{|y-x|^2}{\widehat{c}_{-}t}
			\right)
		}{(2\pi \widehat{c}_{-}t)^{d}}\\
		&=
		\frac
		{\widehat{c}_{+}^{d/2} \widehat{c}_{-}^{-d/2}  \widehat{C}_{-}^2}
		{
			\widehat{C}_{+}
			+
			C_{-1,p_1}
			\sup_{0 \leq s \leq t}
			\e
			\left[
				Z_s(1,Y^{0,x})^{-p_2}
			\right]^{1/p_2}
			\max_{i=1,2}
			\e
			\left[
				|b(s,Y^{0,x})|^{i p_3}
			\right]^{1/p_3}
		}
		g_{2^{-1}\widehat{c}_{-}t} (x,y),
	\end{align*}
	a.e., $y \in \real^d$.
	Therefore, we conclude the statement (ii).
	The continuity of $p_t(x,\cdot)$, Lemma \ref{moment_0} and Lemma \ref{bdd_Girsanov_1} with $r=\pm p_{2}$ implies the statement (iii).
\end{proof}


\subsection{SDEs with bounded and path--dependent drift}\label{Sub_para}

The parametrix method is a useful tool for studying a fundamental solution of parabolic type partial differential equations.
In this subsection, we apply the parametrix method to provide another representation formula for a pdf of solution of SDE with bounded and path--dependent drift.

Let $\widetilde{X}^{s,x}=(\widetilde{X}_t^{s,x})_{t \in [s,T]}$ be a solution of the following Markovian SDE of the form
\begin{align*}
	\widetilde{X}_t^{s,x}
	=x
	+\int_{s}^{t} \widetilde{b}(r,\widetilde{X}^{s,x}_r)\rd r
+\int_{s}^{t}  \sigma(r,\widetilde{X}_r^{s,x})\rd W_r,
\end{align*}
where $\widetilde{b}:[0,T] \times \real^d \to \real^d$ is a bounded and measurable.
Under Assumption \ref{Ass_1} on $\sigma$, by using the same way as a proof of Theorem \ref{main_1}, a pdf of $\widetilde{X}_t^{s,x}$, denoted by $\widetilde{p}(s,x;t,\cdot)$ satisfies
\begin{align*}
	\widetilde{p}(s,x;t,y)
	&=q(s,x;t,y)
	+\int_{s}^{t}
	\e\left[
	\langle
	\nabla_x q(r,\widetilde{X}^{s,x}_r;t,y), \widetilde{b}(r,\widetilde{X}^{s,x}_r)
	\rangle
	\right]
	\rd r \notag\\
	&=q(s,x;t,y)
	+\int_s^t \rd r
	\int_{\real^d} \rd z
	\langle
	\nabla_x q(r,z;t,y), \widetilde{b}(r,z)
	\rangle
	\widetilde{p}(s,x;r,z),
\end{align*}
which is an analogue of the parametrix method (see, page 4, (2.8) in \cite{Fr64} or (4.4) in \cite{LeMe}).
Moreover, we have the following ``formal" expansion holds
\begin{align}\label{para_expansion_0}
	\widetilde{p}(s,x;t,y)
	&=
	\sum_{n=0}^{\infty}
	q \otimes \widetilde{H}^{\otimes n} (s,x;t,y),
\end{align}
where
$
\widetilde{H}(r,x;t,y):=
\langle
\nabla_x q(r,x;t,y), \widetilde{b}(r,x)
\rangle
$
and the space and time convolution operator $\otimes$ is defined by
\begin{align*}
f \otimes g (s,x;t,y):=& \int_s^t \rd r \int_{\real^d} \rd z f(s,x;r,z) g(r,z;t,y),
\end{align*}
and we denote $f^{\otimes 1}=f$, $f^{\otimes k}=f^{\otimes (k-1)} \otimes f$ and $f \otimes g^{\otimes 0}=f$.
\begin{Rem}
	Deck and Kruse \cite{DeKr02} shows that if the drift $b$ is of H\"older growing condition: $\sup_{0 \leq t <T} |b(t,x)|\leq K(1+|x|^{\beta})$ with $\beta <\alpha \leq 1$, a similar expansion converges absolutely and uniformly for $(t,x,y) \in (0,T] \times \real^d \times \real^d$.
\end{Rem}


In order to provide another representation for $p_t(x,\cdot)$, we first show that if $\widetilde{b}$ is bounded measurable, then the expansion \eqref{para_expansion_0} converges absolutely and uniformly in $x,y \in \real^d$, and $\partial_{x_i}\widetilde{p}(s,x;t,y)$ exists for all $i=1,\ldots,d$.

We denote $|\widetilde{H}|(s,z;t,y):=|\widetilde{H}(s,z;t,y)|$.

\begin{Lem}\label{main_2_lem_1}
	Suppose Assumption \ref{Ass_1} holds and the drift coefficient $\widetilde{b}$ is bounded and measurable.
	Then it holds that
	\begin{align}\label{BM_drift_3}
		|\widetilde{H}|^{\otimes n}(s,z;t,y)
		\leq
		\frac
		{(\sqrt{d} \|\widetilde{b}\|_{\infty} \widehat{C}_{+})^n (t-s)^{(n-2)/2} \Gamma(1/2)^n}
		{\Gamma(n/2)}
		g_{\widehat{c}_{+}(t-s)} (z,y),~n \in \n,
	\end{align}
	and there exist $C_+$ and $c_+$ such that for each $k\in \{0,1\}$ and $i=1\ldots,d$,
	\begin{align}
		\sum_{n=0}^{\infty}
		\sup_{x,y \in \real}
		|\partial_{x_i}^{k} q| \otimes |\widetilde{H}|^{\otimes n} (s,x;t,y)
		&\leq
		\frac{C_+}{(t-s)^{k/2}} g_{c_{+}(t-s)}(x,y),\label{para_expa_02}\\
		\partial_{x_i}^{k}
		\widetilde{p}(s,x;t,y)
		&=
		\sum_{n=0}^{\infty}
		(\partial_{i}^{k} q) \otimes \widetilde{H}^{\otimes n} (s,x;t,y),
		\label{para_expa_01}
	\end{align}
	for all $0\leq s<t \leq T$ and $x,y \in \real^d$.
\end{Lem}
\begin{proof}
	We first show \eqref{BM_drift_3}.
	For $n=1$, from \eqref{bound_drev_2}, it holds that
	\begin{align*}
		|\widetilde{H}|(s,z;t,y)
		=
		|\widetilde{H}(s,z;t,y)|
		\leq
		\frac{\sqrt{d}\|\widetilde{b}\|_{\infty}\widehat{C}_+}{\sqrt{t-s}}
		g_{\widehat{c}_{+}(t-s)} (z,y).
	\end{align*}
	We assume that \eqref{BM_drift_3} holds for $n-1 \geq 1$, then from Chapman--Kolmogorov equation, change of variable $u=t_{1}-s$ and Lemma \ref{pa_App_1} with $m=1$, $t_{0}=t-s$, $a=1/2$ and $b=(n-3)/2>-1$, we have
	\begin{align*}
	|\widetilde{H}|^{\otimes n}(s,z;t,y)
	&=
	\int_{s}^{t} \rd t_1
	\int_{\real^d} \rd z_1
	|\widetilde{H}|^{\otimes (n-1)}(s,z;t_1,z_1)
	|\widetilde{H}|(t_1,z_1;t,y)
	\\
	&\leq
	\frac
	{(\sqrt{d} \|\widetilde{b}\|_{\infty}\widehat{C}_{+})^n \Gamma(1/2)^{n-1}}
	{\Gamma((n-1)/2)}
	\int_{s}^{t}\rd t_1
	\frac{(t_1-s)^{(n-1-2)/2}}{(t-t_1)^{1/2}}
	\int_{\real^{d}}
	\rd z_{1}
	g_{\widehat{c}_{+}(t_{1}-s)} (z_{1},z)
	g_{\widehat{c}_{+}(t-t_{1})} (y,z_{1})
	\\
	&=
	\frac
	{(\sqrt{d} \|\widetilde{b}\|_{\infty}\widehat{C}_{+})^n \Gamma(1/2)^{n-1}}
	{\Gamma((n-1)/2)}
	g_{\widehat{c}_{+}(t-s)} (z,y)
	\int_{0}^{t-s}
	\frac{u^{(n-3)/2}}{(t-s-u)^{1/2}}
	\rd u\\
	&=
	\frac
	{(\sqrt{d} \|\widetilde{b}\|_{\infty} \widehat{C}_{+})^n (t-s)^{(n-2)/2} \Gamma(1/2)^n}
	{\Gamma(n/2)}
	g_{\widehat{c}_{+}(t-s)} (z,y).
	\end{align*}
	Hence \eqref{BM_drift_3} holds for every $n \in \n$.
	
	Now we consider the expansion \eqref{para_expa_01} and upper bound \eqref{para_expa_02}.
	By using \eqref{bound_drev_2} and \eqref{BM_drift_3}, Chapman--Kolmogorov equation, change of variable $u=t-t_{1}$ and Lemma \ref{pa_App_1} with $m=1$, $t_{0}=t-s$, $a=k/2$ and $b=(n-2)/2>-1$, we have for each $n \in \n$, $k\in \{0,1\}$ and $i=1\ldots,d$,
	\begin{align}
		|\partial_{x_i}^{k} q| \otimes |\widetilde{H}|^{\otimes n} (s,x;t,y)
		&=\int_{s}^{t} \rd t_1 \int_{\real^d} \rd z_1
			\left|
			\partial_{x_i}^{k} q (s,x;t_1,z_1)
			\right|
			|\widetilde{H}|^{\otimes n} (t_1,z_1;t,y) \notag\\
		&\leq
		\frac
		{(\sqrt{d} \|\widetilde{b}\|_{\infty})^n \widehat{C}_{+}^{n+1} \Gamma(1/2)^{n}}
		{\Gamma(n/2)}
		\int_{s}^{t}
			\frac{(t-t_1)^{(n-2)/2}}{(t_1-s)^{k/2}}
		\rd t_1 
		\cdot
		g_{\widehat{c}_{+}(t-s)} (x,y) \notag\\
		&=
		\frac
		{(\sqrt{d} \|\widetilde{b}\|_{\infty})^n \widehat{C}_{+}^{n+1} \Gamma(1/2)^{n}}
		{\Gamma(n/2)}
		\int_{0}^{t-s}
		\frac{u^{(n-2)/2}}{(t-s-u)^{k/2}}
		\rd u
		\cdot
		g_{\widehat{c}_{+}(t-s)} (x,y) \notag\\
		&=
		\frac
		{(\sqrt{d} \|\widetilde{b}\|_{\infty})^n \widehat{C}_{+}^{n+1} (t-s)^{(n-k)/2} \Gamma(1/2)^{n} \Gamma(1-k/2)}
		{\Gamma((n+2-k)/2)}
		g_{\widehat{c}_{+}(t-s)} (x,y) \notag\\
		&\leq
		\frac
		{(\sqrt{d} \|\widetilde{b}\|_{\infty})^n \widehat{C}_{+}^{n+1} T^{n/2} \Gamma(1/2)^{n}\Gamma(1-k/2)}
		{\Gamma((n+2-k)/2)}
		\frac{g_{\widehat{c}_{+}(t-s)} (x,y)}{(t-s)^{k/2}}.
		\label{para_expa_03}
	\end{align}
	Hence we have
	\begin{align*}
		\sum_{n=0}^{\infty}
			\sup_{x,y \in \real}
			|\partial_{x_i}^k q| \otimes |\widetilde{H}|^{\otimes n} (s,x;t,y)
		<\infty,
	\end{align*}
	which concludes the statement.
\end{proof}

We obtain the following representation on $p_t(x,y)$.

\begin{Thm}\label{main_2_1}
	Suppose Assumption \ref{Ass_1} holds and $b,\widetilde{b}$ are bounded.
	Then for any $(t,x,y) \in (0,T] \times \real^d \times \real^d$, it holds that
	\begin{align}\label{density_3}
	p_t(x,y)
	=\widetilde{p}(0,x;t,y)
	+\int_0^t
	\e\left[
	\langle
	\nabla_x \widetilde{p}(s,X_s^x;t,y), b(s,X^x)-\widetilde{b}(s,X_s^x)
	\rangle
	\right]
	\rd s.
	\end{align}
\end{Thm}

\begin{proof}
	We define $\Delta(0,x;t,y):=p_t(x,y)-\widetilde{p}(0,x;t,y)$ and
	\begin{align*}
		\Lambda(0,x;t,y)
		:=\int_0^t
		\e\left[
		\langle
		\nabla_x \widetilde{p}(s,X_s^x;t,y), b(s,X^x)-\widetilde{b}(s,X^x_s)
		\rangle
		\right]
		\rd s.
	\end{align*}
	Then we show $\Delta(0,x;t,y)=\Lambda(0,x;t,y)$ for all $(t,x,y) \in (0,T] \times \real^d \times \real^d$.
	
	We first compute $\Delta(0,x;t,y)$.
	It follows from Theorem \ref{main_1} that
	\begin{align*}
		\widetilde{p}(0,x;t,y)
		&=q(0,x;t,y)
		+\int_0^t \rd s
		\int_{\real^d} \rd z
		\langle
		\nabla_x q(s,z;t,y), \widetilde{b}(s,z)
		\rangle
		\widetilde{p}(0,x;s,z).
	\end{align*}
	Hence $\Delta(0,x;t,y)$ satisfies the following linear equation,
	\begin{align}\label{BM_drift_1}
	\Delta(0,x;t,y)
	=\overline{\Delta}(0,x;t,y)
	+\Delta \otimes \widetilde{H}(0,x;t,y),
	\end{align}
	where the space and time convolution operator $\otimes$ is defined above and
	\begin{align*}
	\overline{\Delta}(0,x;t,y)
	&:=
	\int_0^t
		\e\left[
		\langle
		\nabla_x q(s,X_s^x;t,y), b(s,X^x)-\widetilde{b}(s,X^x_s)
		\rangle
		\right]
	\rd s.
	\end{align*}
	Since \eqref{BM_drift_1} is a linear equation, we have for any $N \in \n$,
	\begin{align*}
	\overline{\Delta}(0,x;t,y)
	&=\overline{\Delta}(0,x;t,y)
	+\overline{\Delta}\otimes \widetilde{H}(0,x;t,y)
	+\Delta\otimes \widetilde{H}^{\otimes 2}(0,x;t,y)\\
	&=
	\sum_{n=0}^{N-1}
	\overline{\Delta}\otimes \widetilde{H}^{\otimes n}(0,x;t,y)
	+\Delta\otimes \widetilde{H}^{\otimes N}(0,x;t,y).
	\end{align*}
	Now we estimate the upper bound of $|\Delta\otimes \widetilde{H}^{\otimes N}(0,x;t,y)|$.
	By using \eqref{BM_drift_3}, the Gaussian upper bound \eqref{unif_GB_0} for $p_t(x,y)$ and $\widetilde{p}(0,x;t,y)$ and Chapman--Kolmogorov equation, we have
	\begin{align*}
		|\Delta\otimes \widetilde{H}^{\otimes N}(0,x;t,y)|
		&\leq 
		\int_{0}^{t} \rd s
		\int_{\real^d} \rd z
			\{p_t(x,Y)+\widehat{p}(0,x;t,y)\}
			|\widetilde{H}|^{\otimes N}(s,z;t,y)\\
		&\leq
		\frac{2 C_{+} \widehat{C}_{+}^{N} T^{N/2} \Gamma(1/2)^N }{\Gamma(N/2)}
		g_{c_{+}t} (x,y) \to 0,
	\end{align*}
	as $N \to \infty$.
	Hence, we obtain that $\Delta(0,x;t,y)$ satisfies the expansion\footnote{
	The procedure to obtain this expansion is called the parametrix method, and $\overline{\Delta}(0,x;t,y)$ is called the parametrix (see, page 4, (2.8) in \cite{Fr64} or (4.4) in \cite{LeMe}).
	}
	\begin{align*}
	\Delta(0,x;t,y)
	=\sum_{n=0}^{\infty}
		\overline{\Delta}\otimes \widetilde{H}^{\otimes n}(0,x;t,y).
	\end{align*}
	
	Now we prove $\Lambda(0,x;t,y)=\Delta(0,x;t,y)$.
	Recall that $\partial_{x_i}\widetilde{p}(s,x;t,y)$ satisfies the expansion \eqref{para_expa_01}.
	In order to use Fubini's Theorem, we prove that the following integral is finite
	\begin{align*}
		&|\Lambda|(0,x;t,y)\\
		&:=
		\sum_{n=0}^{\infty}
		\sum_{i=1}^{d}
		\int_0^t \rd s
		\int_s^t \rd r
		\int_{\real^d} \rd w
		\e\left[
		|(\partial_{i} q)(s,X_s^x;r,w)|
			|b^i(s,X^x)-\widetilde{b}^i(s,X^x_s)|
		\right]
		|\widetilde{H}|^{\otimes n}(r,w;t,y).
	\end{align*}
	Since $b$ and $\widetilde{b}$ are bounded, by using \eqref{para_expa_03} and Gaussian upper bound \eqref{unif_GB_0} for $p_s(x,\cdot)$, there exists $C>0$ such that
	\begin{align*}
		&|\Lambda|(0,x;t,y)\\
		&\leq
		(\|b\|_{\infty} \vee \|\widetilde{b}\|_{\infty})
		\sum_{n=0}^{\infty}
		\sum_{i=1}^{d}
		\int_0^t \rd s
		\int_s^t \rd r
		\int_{\real^d} \rd w
			\int_{\real^d} \rd v
		|(\partial_{i} q)(s,v;r,w)| |\widetilde{H}|^{\otimes n}(r,w;t,y)
		p_s(x,v)\\
		&\leq
		\sum_{n=0}^{\infty}
			\frac{C^n}{\Gamma((n+1)/2)}
			\int_0^t \rd s
				\frac{1}{{\sqrt{t-s}}}
					\int_{\real^d} \rd v
						g_{\widehat{c}_{+}(t-s)} (v,y)
						g_{\widehat{c}_{+}s}(x,v)
		<\infty.
	\end{align*}
	Therefore, using Fubini's Theorem, we have
	\begin{align*}
	&\Lambda(0,x;t,y)\\
	&=
	\sum_{n=0}^{\infty}
	\sum_{i=1}^{d}
	\int_0^t \rd s
	\int_s^t \rd r
	\int_{\real^d} \rd w
	\e\left[
	(\partial_{i} q)(s,X_s^x;r,w) \{b^i(s,X^x)-\widetilde{b}^i(s,X^x_s)\}
	\right]
	\widetilde{H}^{\otimes n}(r,w;t,y)\\
	&=
	\sum_{n=0}^{\infty}
	\int_0^t \rd r
	\int_{\real^d} \rd w
	\int_0^r \rd s
	\e\left[
	\langle \nabla_{x}q(s,X_s^x;r,w), b(s,X^x)-\widetilde{b}(s,X^x_s)\rangle
	\right]
	\widetilde{H}^{\otimes n}(r,w;t,y)\\
	&=
	\sum_{n=0}^{\infty}
	\int_0^t \rd r
	\int_{\real^d} \rd w
	\overline{\Delta}(0,x;r,w)
	\widetilde{H}^{\otimes n}(r,w;t,y)
	=
	\sum_{n=0}^{\infty}
	\overline{\Delta}\otimes \widetilde{H}^{\otimes n}(0,x;t,y)
	=
	\Delta(0,x;t,y),
	\end{align*}
	which concludes the proof.
\end{proof}

\subsection{Sharp bounds for a pdf of Brownian motion with bounded drift}\label{Sub_sharp}

Inspired by \cite{QiZh02,QiZh03}, we consider a sharp two--sided bound for a Brownian motion with path--dependent and bounded drift coefficient of the form
\begin{align}\label{BM_drift_0}
	X_t^x
	=x
	+\int_{0}^{t}
		b(s,X^x)
	\rd s
	+W_t,
	~x \in \real^d,
	~t \in [0,T],
\end{align}
by using representation \eqref{density_3} and bang--bang diffusion processes.

We define a $d$-dimensional bang--bang diffusion process $Y^{x,\alpha, \beta}=(Y_t^{x,\alpha, \beta})_{t \in [0,T]}$ with parameter $\alpha=(\alpha_1,\ldots,\alpha_d)^{\top}$, $\beta=(\beta_1,\ldots,\beta_d)^{\top} \in \real^d$, which satisfies the following SDE
\begin{align*}
	Y_t^{x,\alpha,\beta}
	=x
	+\int_{0}^{t}
	\beta \mbox{sgn}(\alpha-Y_s^{x,\alpha,\beta})
	\rd s
	+W_t,
\end{align*}
where $\beta \mbox{sgn}(x):=(\beta_1 \mbox{sgn}(x_1),\ldots,\beta_d \mbox{sgn}(x_d))^{\top}$, for each $x \in \real^d$.
Then it follows from Theorem 2 in \cite{QiZh02} (see also (6.5.14) in \cite{KS}) that for any $t \in (0,T]$, $Y_t^{x,\alpha,\beta}$ admits a pdf, denoted by $q^{\alpha,\beta}_t(x,\cdot)$ which satisfies
\begin{align*}
	q^{\alpha,\beta}_t(x, \alpha)
	=
	\prod_{i=1}^{d}
	\frac{2}{\sqrt{2 \pi t}}
	\int_{|x_i-\alpha_i|/\sqrt{t}}^{\infty}
	z_i \exp\left(
	-\frac{(z_i-\beta_i \sqrt{t})^2}{2}
	\right)
	\rd z_i.
\end{align*}

\begin{Thm}\label{main_BM_0}
	Suppose Assumption \ref{Ass_1} holds and the drift coefficient $b$ is bounded.
	Then a pdf of a solution of \eqref{BM_drift_0}, denoted by $p_{t}(x,\cdot)$ satisfies the following two--sided estimates: for any $(t,x,y) \in [0,T] \times \real^d \times \real^d$,
	\begin{align*}
		q^{y,-\|b\|_{\infty}}_t(x, y)
		\leq
		p_{t}(x,y)
		\leq q^{y,\|b\|_{\infty}}_t(x, y).
	\end{align*}
\end{Thm}
\begin{proof}
	Let $x,y \in \real^d$ be fixed.
	Using Theorem \ref{main_2_1} with $\widetilde{p}=q^{y,\pm\|b\|_{\infty}}$, we have
	\begin{align*}
		p_{t}(x,y)
		-q^{y,\pm\|b\|_{\infty}}_t(x, y)
		=\int_{0}^{t}
			\e\left[
				\langle
					\nabla_x
					q^{y,\pm\|b\|_{\infty}}_{t-s}(X_s^x, y)
					,
					b(s,X^x)
					-(\pm \|b\|_{\infty})
					\mbox{sgn}(y-X_s^x)
				\rangle
			\right]
		\rd s.
	\end{align*}
	On the other hand, it holds that for any $s \in [0,t)$, $z \in \real^d$ and $w \in C([0,\infty);\real^d)$,
	\begin{align*}
		\partial_{z_i} q^{y,\|b\|_{\infty}}_{t-s}(z, y)
		(b^i(s,w)-\|b\|_{\infty}\mbox{sgn}(y_i-z_i))
		&\leq 0,\\
		\partial_{z_i} q^{y,-\|b\|_{\infty}}_{t-s}(z, y)
		(b^i(s,w)+\|b\|_{\infty}\mbox{sgn}(y_i-z_i))
		&\geq 0,
	\end{align*}
	thus we conclude the statement.
\end{proof}

\subsection{Comparison property of pdfs}\label{Sub_comp}

In this subsection, we consider a comparison property of pdfs.
Let $X^x$ and $\widehat{X}^x$ be a solution of path--dependent SDE \eqref{SDE_1} with drift coefficient $b$ and $\widehat{b}$, respectively.
We denote by $p_{t}(x,\cdot)$ and $\widehat{p}_{t}(x,\cdot)$ pdf of $X_t^x$ and $\widehat{X}_t^x$ for $t \in (0,T]$, respectively.
Then we have the following comparison property of $p_{t}(x,\cdot)$ and $\widehat{p}_{t}(x,\cdot)$.

\begin{Thm}\label{main_comp_0}
	
	Suppose Assumption \ref{Ass_1} and \ref{Ass_2} hold for $b$, $\widehat{b}$ and $\sigma$.
	Let $p_1,p_2,p_3>1$ with $p_1 \in (1,\frac{d}{d-1})$ and $1/p_1+1/p_2+1/p_3=1$.
	\begin{itemize}
		\item[(i)]
		There exists $C_{+}\equiv C_{+}(p_1)>0$ such that for any $(t,x) \in (0,T] \times \real^d$ and a.e. $y \in \real^d$, it holds that 
		\begin{align}\label{main_comp_1}
		&|p_t(x,y)-\widehat{p}_{t}(x,y)| \notag\\
		&\leq
		C_{+}
		\sup_{0 \leq s \leq t}
		\left\{
			\e
				\left[
					Z_s(1,Y^{0,x})^{2p_2}
				\right]^{\frac{1}{2p_2}}
			+
			\e
				\left[
					\widehat{Z}_s(1,Y^{0,x})^{2p_2}
				\right]^{\frac{1}{2p_2}}
		\right\}
		\e
		\left[
		\widehat{b}(s,Y^{0,x})^{2p_2}
		\right]^{\frac{1}{2p_2}}
		\notag\\
		&\quad \times
		\e
		\left[
		|b(s,Y^{0,x})
		+
		\widehat{b}(s,Y^{0,x})|^{2p_3}
		\right]^{\frac{1}{2p_3}}
		\e
		\left[
		\left|
		b(s,Y^{0,x})
		-
		\widehat{b}(s,Y^{0,x})
		\right|^{2p_3}
		\right]^{\frac{1}{2p_3}}
		g_{p_1\widehat{c}_{+}t}(x,y).
		\end{align}
		
		\item[(ii)]
		Recall that $t_{p_2}$ is defined by \eqref{def_tr}.
		There exist $C_{+}>0$ and $c_{+}>0$ such that for any $(t,x) \in (0,T] \times \real^d$ and a.e. $y \in \real^d$, it holds that if $t \in (0,t_{2p_2}]$, then 
		\begin{align*}
			&|p_t(x,y)-\widehat{p}_{t}(x,y)| \notag\\
			&\leq
			C_{+} 
			\exp
			\left(
				c_{+}
				(1+|x|^2)t
			\right) 
			(1+|x|)^2
			\sup_{0 \leq s \leq t }
			\e
			\left[
			\left|
			b(s,Y^{0,x})
			-
			\widehat{b}(s,Y^{0,x})
			\right|^{2p_3}
			\right]^{\frac{1}{2p_3}}
			g_{p_1\widehat{c}_{+}t}(x,y),
		\end{align*}
		and if $t \in (t_{2p_2},T]$, then
		\begin{align*}
			&|p_t(x,y)-\widehat{p}_{t}(x,y)| \notag\\
			&\leq
			C_{+} 
			\exp\left(
				\frac{|x|^2}{16\widehat{c}_{+}p_{2}T}
			\right)
			(1+|x|)^2
			\sup_{0 \leq s \leq t }
			\e
				\left[
					\left|
						b(s,Y^{0,x})
						-
						\widehat{b}(s,Y^{0,x})
					\right|^{2p_3}
				\right]^{\frac{1}{2p_3}}
			g_{p_1\widehat{c}_{+}t}(x,y).
		\end{align*}
	\end{itemize}
\end{Thm}
\begin{proof}
	By using Theorem \ref{main_1} and \eqref{Girsanov_0}, we have
	\begin{align}
	|p_t(x,y)-\widehat{p}_{t}(x,y)|
	&
	\leq
	I_t^{(1)}(x,y)
	+
	I_t^{(2)}(x,y),\label{comp_1}
	\end{align}
	where $I_t^{(1)}(x,y)$ and $I_t^{(2)}(x,y)$ are defined by
	\begin{align*}
	I_t^{(1)}(x,y)
	:=&\int_0^t
	\e\left[
	\left|
	\nabla_x q(s,X_s^{0,x};t,y)
	\right|
	\left|
	b(s,X^{0,x})-\widehat{b}(s,X^{0,x})
	\right|
	\right]
	\rd s,\\
	I_t^{(2)}(x,y)
	:=&
	\int_0^t
	\e\left[
	\left|
	\nabla_x q(s,Y_s^{0,x};t,y)
	\right|
	\left|
	Z_s(1,Y^{0,x})-\widehat{Z}_s(1,Y^{0,x})
	\right|
	\left|
	\widehat{b}(s,Y_s^{0,x})
	\right|
	\right]
	\rd s,
	\end{align*}
	and for $q\in \real$, $\widehat{Z}(q,Y^{0,x})=(\widehat{Z}_t(q,Y^{0,x}))_{t \in [0,T]}$ is a martingale defined by
	\begin{align*}
	\widehat{Z}_t(q,Y^{0,x})
	&:=
	\exp\left(
	\sum_{j=1}^{d}
	\int_{0}^{t}
	q \widehat{\mu}^j(s,Y^{0,x})
	\rd W_s^{j}
	-
	\frac{1}{2}
	\int_{0}^{t}
	|q\widehat{\mu}(s,Y^{0,x})|^2
	\rd s
	\right),\\
	\widehat{\mu}(t,w)
	&:=
	\sigma(t,w_t)^{-1} \widehat{b}(t,w),~(t,w) \in [0,T] \times C([0,T];\real^d).
	\end{align*}
	By using the inequality \eqref{Girsanov_2} and Lemma \ref{key_1} 
	, we have
	\begin{align*}
		&I_t^{(1)}(x,y)\\
		&\leq
		\int_0^t
			\frac{\sqrt{d}\widehat{C}_{+}}{(t-s)^{1/2}}
			\e 
				\left[
					\left|
						g_{\widehat{c}_{+}t} (Y_s^{0,x}, y)
					\right|^{p_1}
				\right]^{1/p_1}
			\e\left[
			\left|
			Z_s(1,Y^{0,x})
			\right|^{p_2}
			\right]^{1/p_2}
			\e\left[
			\left|
			b(s,Y_s^{0,x})-\widehat{b}(s,Y_s^{0,x})
			\right|^{p_3}
			\right]^{1/p_3}
		\rd s\\
		&\leq
		C_{p_1}
		\sup_{0 \leq s \leq t}
		\e
		\left[
			Z_s(1,Y^{0,x})^{p_2}
		\right]^{1/p_2}
		\e
			\left[
				\left|
					b(s,Y^{0,x})-\widehat{b}(s,Y^{0,x})
				\right|^{p_3}
			\right]^{1/p_3}
		g_{\widehat{c}_{+}t}(x,y)
		\int_{0}^{t}
			\frac{1}{(t-s)^{\frac{d(p_1-1)}{2p_1}+\frac{1}{2}}}
		\rd s\\
		&\leq
		C_{p_1}'
		\sup_{0 \leq s \leq t}
		\e
		\left[
			Z_s(1,Y^{0,x})^{p_2}
		\right]^{1/p_2}
		\e
			\left[
				\left|
					b(s,Y^{0,x})-\widehat{b}(s,Y^{0,x})
				\right|^{p_3}
			\right]^{1/p_3}
		g_{\widehat{c}_{+}t}(x,y),
	\end{align*}
	for some $C_{p_1},C_{p_1}'>0$.
	For the second term of \eqref{comp_1}, by using H\"older's inequality, the elementary estimate $|e^x -e^y| \leq (e^x +e^y)|x-y|$, Minkowski's inequality and Lemma \ref{key_1}, 
	we have
	\begin{align*}
		&I_t^{(2)}(x,y)\\
		&\leq
		\int_0^t
			\e\left[
			\left|
			\nabla_x q(s,Y_s^{0,x};t,y)
			\right|^{p_1}
			\right]^{1/p_1}
			\e\left[
			\left|
			Z_s(1,Y^{0,x})+\widehat{Z}_s(1,Y^{0,x})
			\right|^{2p_2}
			\right]^{1/2p_2}
			\e\left[
			\left|
			\widehat{b}(s,Y^{0,x})
			\right|^{2p_2}
			\right]^{1/2p_2}\\
			& \quad \times
			\e\left[
			\left|
			\sum_{j=1}^{d}
			\int_{0}^{s}
			\mu^j(u,Y^{0,x})
			-
			\widehat{\mu}^j(u,Y^{0,x})
			\rd W_u^{j}
			-
			\frac{1}{2}
			\int_{0}^{s}
			|\mu(u,Y^{0,x})|^2
			-
			|\widehat{\mu}(u,Y^{0,x})|^2
			\rd u
			\right|^{p_3}
			\right]^{1/p_3}
		\rd s\\
		&\leq
		C_{p_1}
		g_{\widehat{c}_{+}t}(x,y)
		\sup_{0 \leq s \leq t}
		\left\{
			\e
				\left[
					Z_s(1,Y^{0,x})^{2p_2}
				\right]^{1/2p_2}
			+
			\e
				\left[
					\widehat{Z}_s(1,Y^{0,x})^{2p_2}
				\right]^{1/	2p_2}
		\right\}
		\e
			\left[
				\left|
					\widehat{b}(s,Y_s^{0,x})
				\right|^{2p_2}
			\right]^{1/2p_2}\\
		&\quad \times
		\left\{
			\e
				\left[
					\left|
						\sum_{j=1}^{d}
							\int_{0}^{s}
								\mu^j(u,Y^{0,x})
								-
								\widehat{\mu}^j(u,Y^{0,x})
							\rd W_u^{j}
					\right|^{p_3}
				\right]^{1/p_3}
			+
			\e
				\left[
					\left|
						\int_{0}^{s}
							|\mu(u,Y^{0,x})|^2
							-
							|\widehat{\mu}(u,Y^{0,x})|^2
						\rd u
					\right|^{p_3}
				\right]^{1/p_3}
		\right\}.
	\end{align*}
	From Burkholder-Davis-Gundy's inequality, Schwarz's inequality and Assumption \ref{Ass_1}, we obtain
	\begin{align*}
		&I_t^{(2)}(x,y)\\
		&\leq
		C_{p_1}
		g_{\widehat{c}_{+}t}(x,y)
		\sup_{0 \leq s \leq t}
		\left\{
			\e
				\left[
					Z_s(1,Y^{0,x})^{2p_2}
				\right]^{1/2p_2}
			+
			\e
				\left[
					\widehat{Z}_s(1,Y^{0,x})^{2p_2}
				\right]^{1/2p_2}
		\right\}
		\e
			\left[
				\widehat{b}(s,Y^{0,x})^{2p_2}
			\right]^{1/2p_2}
		\notag\\
		&\quad \times
		\e
			\left[
				|b(s,Y^{0,x})
				+
				\widehat{b}(s,Y^{0,x})|^{2p_3}
			\right]^{{2p_3}}
			\e
				\left[
					\left|
						b(s,Y^{0,x})
						-
						\widehat{b}(s,Y^{0,x})
					\right|^{2p_3}
			\right]^{1/2p_3},
	\end{align*}
	for some $C_{p_1}>0$.
	Hence we conclude \eqref{main_comp_1}.
	Lemma \ref{moment_0} and Lemma \ref{bdd_Girsanov_1} imply the statement (ii).
\end{proof}

\subsection{Application to an Euler--Maruyama type scheme}\label{Sub_EM}

In this subsection, we consider an Euler--Maruyama type scheme for Brownian motion with path--dependent drift coefficient of the form
\begin{align}\label{SDE_APP_0}
	\rd X_t^{x}
	=
	\nu(A_t(X^{x}))
	\rd t
	+
	\sigma
	\rd W_t,~
	t\in [0,T],~X_0=x\in \real^d,
\end{align}
where $\nu:\mathbb{S}:=[0,\infty) \times \real^d \times [0,\infty) \times (\real^{d})^{\n} \times \real^{\ell} \to \real^d$ and $A_t : C([0,\infty);\real^d) \to \mathbb{S}$ is defined by \eqref{Eg_1}, that is,
\begin{align*}
A_t(w)
=
\left(
t,
w_t,
\max_{0 \leq s \leq t} \zeta(s,w_s),
\{w_{\tau_i(t)}\}_{i\in \n},
\int_{0}^{t} c(s,w_s) \rd s,
\right) \in \mathbb{S},
~
w \in C([0,\infty);\real^d),
\end{align*}
for some measurable functions $\zeta :[0,\infty) \times \real^{d} \to [0,\infty)$ and $c:[0,T] \times \real^{d} \to \real^{\ell}$.

Let us define an Euler--Maruyama type scheme $X^{(x,n,m)}=(X^{x,n,m}_t)_{t \in [0,T]}$ for SDE \eqref{SDE_APP_0} for $n,m \in \n$ as follows:
\begin{align}\label{EM_0}
	\rd X_t^{(x,n,m)}
	=
	\nu(A_t^{(n,m)}(X^{(x,n,m)}))
	\rd t
	+
	\sigma
	\rd W_t,~
	t\in [0,T],~X_0^{(x,n,m)}=x\in \real^d,
\end{align}
where $A_t^{(n,m)}(w)$ for $w \in C([0,T];\real^d)$ is defined by
\begin{align*}
	A_t^{(n,m)}(w)
	&:=
	\left(
		\eta_n(t),
		w_{\eta_n(t)},
		\max_{0\leq s \leq t} \zeta(\eta_n(s),w_{\eta_n(s)}),
		\{\widehat{w}^{(m)}_{\eta_n(\tau_i(t))} \}_{i \in \n},
		\int_{0}^{\eta_n(t)}
			c(\eta_n(s),w_{\eta_n(s)})
		\rd s,
	\right)
	\in
	\mathbb{S},
\end{align*}
where $\eta_n(t)=kT/n$, if $t \in [kT/n,(k+1)T/n)$ and $\widehat{w}^{(m)}_{\eta_n(\tau_i(t))}$ is defined by
\begin{align*}
	\widehat{w}^{(m)}_{\eta_n(\tau_i(t))}
	:=\left\{ \begin{array}{ll}
	\displaystyle w_{\eta_n(\tau_i(t))} &\text{ if, } i \in \{1,\ldots,m\}\\
	\displaystyle 0 &\text{ if, } i \in \{m+1,m+2,\ldots \},
	\end{array}\right.
\end{align*}

\begin{Ass}\label{Ass_4}
	\begin{itemize}
		\item[(i)]
		$\nu$ is measurable and satisfies that there exists $\theta:=\{\theta_i\}_{i \in \n} \in [0,\infty)^{\n}$ with $\|\theta\|_{\ell_1}=\sum_{i \in \n} \theta_i<\infty$, and for any $\delta,t>0$, there exists $K_t(\delta)>0$ such that $K_t(\delta)$ is increasing with respect to $t$ and for any $\chi=(t,w,z,\{u_i\}_{i\in \n},v) \in \mathbb{S}$,
		\begin{align*}
			|\nu(\chi)|
			\leq
			\delta\|\chi\|_{\theta}+ K_{t}(\delta),
		\end{align*}
		where $\|\chi\|_{\theta}^2:=t^2+|w|^2+|z|^2+\sum_{i \in \n} \theta_i |u_i|^2+|v|^2$.
		
		\item[(ii)]
		$\nu$ is H\"older continuous on $\mathbb{S}$, that is, there exist $\beta \in (0,1]$ and $\|\nu\|_{\beta}>0$ such that for any $\chi=(t,w,z,\{u_i\}_{i \in \n},v)$, $\chi'=(t',w',z',\{u_i'\}_{i\in \n },v') \in \mathbb{S}$,
		\begin{align*}
			|\nu(\chi)-\nu(\chi')|
			\leq
			\|\nu\|_{\beta}
			\left(
				|t-t'|^{\beta/2}
				+|w-w'|^{\beta}
				+|z-z'|^{\beta}
				+\sum_{i \in \n } \theta_i |u_i-u_i'|^{\beta}
				+|v-v'|^{\beta}
			\right).
		\end{align*}

		\item[(iii)]
		The functions $\zeta$ and $c$ are $\gamma$-H\"older continuous in space and $\gamma/2$-H\"older continuous in time with $\gamma \in (0,1]$, that is,
		\begin{align*}
			\|\zeta\|_{\gamma}
			&:=
			\sup_{t \in [0,\infty), x \neq y}\frac{|\zeta(t,x)-\zeta(t,y)|}{|x-y|^{\gamma}}
			+\sup_{x\in \real^d, t \neq s}\frac{|\zeta(t,x)-\zeta(s,x)|}{|t-s|^{\gamma/2}}
			<\infty,\\
			\|c\|_{\gamma}
			&:=
			\sup_{t \in [0,\infty), x \neq y}\frac{|c(t,x)-c(t,y)|}{|x-y|^{\gamma}}
			+\sup_{x\in \real^d, t \neq s}\frac{|c(t,x)-c(s,x)|}{|t-s|^{\gamma/2}}
			<\infty.
		\end{align*}
		
		\item[(iv)]
		$\sigma$ is uniformly elliptic matrix, that is, there exists $\underline{a}$ such that for any $\xi \in \real^d$,
		\begin{align*}
			\underline{a}|\xi|^2 \leq \langle \sigma \sigma^{\top} \xi,\xi \rangle .
		\end{align*}
		
	\end{itemize}
\end{Ass}

Under Assumption \ref{Ass_4}, the drift coefficients $\nu \circ A_{\cdot}$ and $\nu \circ A_{\cdot}^{(n,m)}$ satisfy Assumption \ref{Ass_1} and \ref{Ass_2}.
Indeed, since $\zeta$ and $c$ are of linear growth, there exists $C>0$ such that $\|A_t(w)\|_{\theta} \vee \|A_t^{(n,m)}(w)\|_{\theta} \leq C(1+(1+\|\theta\|_{\ell_1})w^{*}_t)$.
Therefore, from Theorem \ref{main_1}, there exist pdfs $p_t(x,\cdot)$ and $p_t^{(n,m)}(x,\cdot)$ of $X_t^{x}$ and $X_t^{(x,n,m)}$, respectively such that for all $y \in \real^d$,
\begin{align*}
	p_t(x,y)
	&=
	g_{t\sigma}(x,y)
	+ \int_0^t
		\e\left[
			\langle
				\nabla_x g_{(t-s)\sigma} (X_s^x,y), \nu(A_s(X^{x}))
			\rangle
		\right]
	\rd s,\\
	p_t^{(n,m)}(x,y)
	&=
	g_{t \sigma}(x,y)
	+ \int_0^t
		\e\left[
			\left
				\langle
					\nabla_x g_{(t-s)\sigma} (X_s^{(x,n,m)},y), \nu(A_s^{(n,m)}(X^{(x,n,m)}))
				\right
			\rangle
		\right]
	\rd s.
\end{align*}

Under Assumption \ref{Ass_4}, we have the following error estimate.
\begin{Lem}\label{err_BM}
	Suppose Assumption \ref{Ass_4} holds.
	Then for any $p>1$, there exists $C_p>0$ such that
	\begin{align*}
		\sup_{0 \leq s \leq t}
		\e
			\left[
				\left|
					\nu(A_s(x+\sigma W))
					-\nu(A_s^{(n,m)}(x+\sigma W))
				\right|^p
			\right]^{1/p}
		\leq
		C_{p}
			\left\{
				\left(
					\frac{\log n}{n}
				\right)^{\beta \gamma/2}
				+
				\sum_{i=m+1}^{\infty} \theta_i
			\right\}.
	\end{align*}
\end{Lem}
\begin{proof}
	It suffices to show the statement for $p>1/\beta$.
	Since $\nu$ is $\beta$-H\"older continuous, we have for any $s \in [0,t]$,
	\begin{align*}
		&\frac
			{\left|
				\nu(A_s(x+\sigma W))
				-\nu(A_s^{(n,m)}(x+\sigma W))
			\right|^{p}}
			{6^{p-1}\|\nu\|_{\beta}^{p}}\\
		&\leq
			|s-\eta_n(s)|^{p\beta}
			+|\sigma|^{p\beta} |W_s-W_{\eta_n(s)}|^{p\beta}
			+\left|
				\max_{0 \leq u \leq s}\zeta(u,x+\sigma W_u)-\max_{0 \leq u \leq s}\zeta(\eta_n(u), x+\sigma W_{\eta_n(u)}) 
			\right|^{p\beta}\\
		&\quad
			+
				|\sigma|^{p\beta}
				\left(
					\sum_{i=1}^{m}
					\theta_i
					|W_{\tau_i(s)}-W_{\eta_n(\tau_i(s))}|^{\beta}
					+
					\sum_{i=m+1}^{\infty}
					\theta_i
					|x+W_{\tau_i(s)}|^{\beta}
				\right)^{p}
			+
				\left|
					\int_{\eta_n(s)}^{s}
						c(u,x+\sigma W_u)
					\rd u
				\right|^{p\beta}
			\\
			&\quad
			+
				t^{p\beta-1}
					\int_{0}^{\eta_n(s)}
						\left|
							c(u,x+\sigma W_u)
							-c(\eta_n(u),x+\sigma W_{\eta_n(u)})
						\right|^{p\beta}
					\rd u\\
		&\leq
				\frac{T^{p\beta/2}}{n^{p\beta/2}}
			+
				|\sigma|^{p\beta}
				|W_s-W_{\eta_n(s)}|^{p\beta}
			+
				2^{p\beta-1}\|\zeta\|_{\gamma}^{p\beta}
				\left\{
					\frac{T^{p\beta \gamma/2}}{n^{p\beta \gamma/2}}
					+
					|\sigma|^{p\beta\gamma}
					\max_{0 \leq u \leq s}
						|W_u-W_{\eta_n(u)}|^{p\beta \gamma}
				\right\}
			\\
			&\quad
			+
				|\sigma|^{p\beta}
				\left(
					\|\theta\|_{\ell_1}
					\max_{0 \leq u \leq s}
					|W_{u}-W_{\eta_n(u)}|^{\beta}
					+
					\max_{0 \leq u \leq s}
					|x+W_{u}|^{\beta}
					\sum_{i=m+1}^{\infty}
					\theta_i
				\right)^{p}
			+
				\frac{T^{p\beta}}{n^{p\gamma}}
				\sup_{0 \leq s \leq T} |c(s,x+\sigma W_s)|^{p\beta}
			\\
			&\quad
			+
				t^{p\gamma}\|c\|_{\gamma}^{p\beta \gamma}
				\left\{
					\frac{T^{p \beta \gamma/2}}{n^{p \beta \gamma/2}}
					+
					|\sigma|^{p\beta \gamma}
					\int_{0}^{T}
						\left|W_u-W_{\eta_n(u)}\right|^{p \beta \gamma}
					\rd u
				\right\}.
	\end{align*}
	By using modulus continuity of Brownian motion (see, Theorem 2.9.25 in \cite{KS}), there exists $C_p>0$ such that
	\begin{align*}
		&
		\sup_{0 \leq s \leq t}
		\e
			\left[
				\left|
					\nu(A_s(x+\sigma W))
					-\nu(A_s^{(n,m)}(x+\sigma W))
				\right|^p
			\right]^{1/p}
		\leq
		C_{p}
			\left\{
				\left(
					\frac{\log n}{n}
				\right)^{\beta \gamma/2}
				+\sum_{i=m+1}^{\infty} \theta_i
			\right\},
	\end{align*}
	which concludes the proof.
\end{proof}

As an conclusion of the comparison Theorem \ref{main_comp_0} and Lemma \ref{err_BM}, we obtain the following error estimate.

\begin{Cor}\label{main_4}
	Suppose Assumption \ref{Ass_4} holds.
	Assume that $f:\real^d \to \real$ is a measurable function satisfies the exponentially bounded, that is, there exist $K>0$ such that for all $x\in \real^d$, $|f(x)| \leq K\exp(K|x|)$.
	Then for any $p_1,p_2,p_3>1$ with $1/p_1+1/p_2+1/p_3=1$ and $p_1 \in (1,d/(d-1))$, there exist $C_{+}>0$ and $c_{+}>0$ such that for any $(t,x,y) \in [0,T] \times \real^d \times \real^d$, if $t \in (0,t_{2p_2}]$, then 
	\begin{align*}
		|p_t(x,y)-p_t^{(n,m)}(x,y)|
		&\leq
			C_{+} 
			\exp
				\left(
					c_{+}
					(1+|x|^2)t
				\right) 
			(1+|x|)^2
			g_{p_1c_{+}t}(x,y)
			\left\{
				\left(
					\frac{\log n}{n}
				\right)^{\beta \gamma/2}
				+
				\sum_{i=m+1}^{\infty} \theta_i
			\right\},
			\\
			\left|
				\e[f(X_T^{x})]
				-\e[f(X_T^{(x,n,m)})]
			\right|
		&\leq
			C_{+} 
			\exp
				\left(
					c_{+}
					(1+|x|^2)t
				\right) 
				(1+|x|)^2
			\left\{
				\left(
					\frac{\log n}{n}
				\right)^{\beta \gamma/2}
				+
				\sum_{i=m+1}^{\infty}
					\theta_i
			\right\},
	\end{align*}
	and if $t \in (t_{2p_2},T]$,
	\begin{align*}
		|p_t(x,y)-p_t^{(n,m)}(x,y)|
		&\leq
			C_{+} 
			\exp\left(
				\frac{|x|^2}{16\widehat{c}_{+}p_{2}T}
			\right)
			(1+|x|)^2
			g_{p_1\widehat{c}t}(x,y)
			\left\{
				\left(
					\frac{\log n}{n}
				\right)^{\beta \gamma/2}
				+
				\sum_{i=m+1}^{\infty}
					\theta_i
			\right\},
			\\
			\left|
			\e[f(X_T^{x})]
			-\e[f(X_T^{(x,n,m)})]
			\right|
		&
		\leq
			C_{+} 
			\exp\left(
			\frac{|x|^2}{16\widehat{c}_{+} p_{2}T}
			\right)
			(1+|x|)^2
			\left\{
				\left(
					\frac{\log n}{n}
				\right)^{\beta \gamma/2}
				+
				\sum_{i=m+1}^{\infty}
					\theta_i
			\right\},
	\end{align*}
	where $t_{2p_2}$ is defined by \eqref{def_tr}.
\end{Cor}

\begin{Rem}
	The main idea of the proof for Corollary \ref{main_4} (and Theorem \ref{main_comp_0}) is to use Maruyama--Girsanov transform.
	This idea is inspired by Mackevi\v{c}ius \cite{Ma03} who study the weak rate of convergence of the Euler--Maruyama scheme for the SDE $\rd X_t=b(X_t) \rd t +\sigma \rd W_t$ under Lipschitz condition on $b$.
	However, we would like to point out that the proof in \cite{Ma03} contains several gaps (see, for instance Lemma 2 in \cite{Ma03}, see also Remark 3.3. in \cite{NT2}).
\end{Rem}

\subsection{Markovian SDEs with unbounded drift}\label{Sub_expa}

In this section, we consider the parametrix expansion similar to \eqref{para_expansion_0} on a pdf of a solution of the following Markov type SDE of the form
\begin{align}\label{SDE_markov}
X_t^{x}
=
x
+\int_{0}^{t}
b(X_s^{x})
\rd s
+\int_{0}^{t}
\sigma(X_s^{x})
\rd W_s,
~t \geq 0,~x\in \real^d.
\end{align}

Let us consider a frozen process $X^{s,x,z}=(X_t^{s,x,z})_{t \in [s,T]}$ for $x,z \in \real^d$ defined by
\begin{align} \label{Froz_pro}
\widetilde{X}_t^{s,x,z}
:=
x
+
\sigma(z)
(W_t-W_s).
\end{align}
We denote $p^z(s,x;t,\cdot)$ a pdf of $\widetilde{X}_t^{s,x,z}$ for $t \in (0,T]$ and $x,z\in \real^d$.

If the drift coefficients $b$ is bounded measurable, and the diffusion coefficient $\sigma$ satisfies Assumption \ref{Ass_1} (ii) and (iii), then by the same way as section 4.1 in \cite{LeMe} and approximation arguments (see Remark 4.1 in \cite{LeMe}), it holds that
\begin{align}\label{para_expansion_smooth}
p_t(x,y)
=
\sum_{n=0}^{\infty}
\overline{p} \otimes H^{\otimes n} (0,x;t,y),
\end{align}
where $\overline{p}(s,x;t,y):=p^y(s,x;t,y)$ and
\begin{align*}
H(s,x;t,y)
:=\langle
\nabla_x \overline{p}(s,x;t,y),  b(x)
\rangle
+
\sum_{i,j=1}^{d}
\frac{a_{i,j}(x)-a_{i,j}(y)}{2}
\partial_{x_ix_j}^2
\overline{p}(s,x;t,y).
\end{align*}

\begin{Rem}
	In \cite{LeMe}, the authors define a frozen process by $\widetilde{X}_t^{s,x,z}:=x+b(z)(t-s)+\sigma(z)(W_t-W_s)$. However, even if one define \eqref{Froz_pro}, the parametrix expansion \eqref{para_expansion_smooth} can be shown by the same way as in \cite{LeMe}.
\end{Rem}

\begin{Ass}\label{Ass_3}
	\begin{itemize}
		\item[(i)]
		the drift coefficient $b$ is of linear growth, that is, there exists $K>0$ such that for any $x \in \real^d$,
		\begin{align*}
		|b(x)|
		\leq
		K(1+|x|).
		\end{align*}
		
		\item[(ii)]
		$b$ is of sub-linear growth, that is, for any $\delta>0$, there exists $K(\delta)>0$ such that for all $x \in \real^d$
		\begin{align*}
		|b(x)|
		\leq
		\delta | x|+K(\delta).
		\end{align*}
		
	\end{itemize}
	
\end{Ass}

Under the above conditions, the parametrix expansion \eqref{para_expansion_smooth} still holds.

\begin{Thm}\label{main_expa_0}
	Suppose that Assumption \ref{Ass_1} (ii), (iii) and Assumption \ref{Ass_3} hold.
	Then the series $\sum_{n=0}^{\infty}\overline{p} \otimes H^{\otimes n} (0,x;t,y)$ converges absolutely and uniformly for $(t,y) \in (0,T] \times \real^d$, and a pdf of $X_t^{x}$ denoted by $p_t(x,y)$ satisfies the expansion
	\begin{align}\label{para_expansion_1}
		p_t(x,y)
		=
		\sum_{n=0}^{\infty}
		\overline{p} \otimes H^{\otimes n} (0,x;t,y).
	\end{align}
\end{Thm}

In order to provide the parametrix expansion for a pdf of a solution of SDE \eqref{SDE_markov}, under unbounded drift coefficients we first show that if $b$ satisfies Assumption  \ref{Ass_1} (ii), (iii) and Assumption \ref{Ass_3}, the expansion of right hand side in \eqref{para_expansion_smooth} convergences absolutely and uniformly in $x, y \in \real^d$. Then by taking bounded measurable approximation of $b$, we show the convergences of the parametrix expansion \eqref{para_expansion_1} by using comparison property (see, Theorem \ref{main_comp_0}).  

We denote $|H|(s,z;t,y):=|H(s,z;t,y)|$, and the following classical estimate will be used below.

\begin{Lem}\label{Basic_est_0}
	Let $A$ be a $d \times d$ matrix and suppose that there exists $\underline{A}$ and $\overline{A}$ such that for all $\xi \in \real^d$, $\underline{A} |\xi|^2 \leq \langle a \xi, \xi \rangle \leq \overline{A} |\xi|^2$.
	Then for $\alpha \in (0,1] $, there exists $C$ such that for all $(t,x,y) \in (0,T] \times \real^d \times \real^d$,  
	\begin{align*}
	|y-x|^{\alpha}
	\left|
	H_{tA}^{i,j}(y-x)
	\right|
	g_{tA}(x,y)
	\leq
	\frac{C}{t^{(1-\frac{\alpha}{2})}}
	g_{2t\overline{A}}
	(x ,y),
	\end{align*}
	where $H_{A}^{i,j}(y):=(A^{-1}y)^{i}(A^{-1}y)^{j}-(A^{-1})_{i,j}$.
\end{Lem}

The proof of Lemma \ref{Basic_est_0} follows by using  the classical estimation $\sup_{x \in \real^d} |x|^q e^{-|x|} < \infty$ for any $q \geq 0$.

\begin{Lem}\label{para_uni_bdd_1}
	Suppose that Assumption \ref{Ass_1} (ii), (iii) and Assumption \ref{Ass_3} (i) hold. Then there exists $C_{+}$ and $c_{+}$ such that for any $n \in \n \cup \{0\}$ and $(t,x,y) \in (0,T] \times \real^d \times \real^d$,
	\begin{align}
		\overline{p} \otimes |H|^{\otimes n} (0,x;t,y)
		&\leq
		\frac
			{\left( C_{+}(1+|x|)t^{\alpha/2} \Gamma(\alpha/2)\right)^n}
			{\Gamma(1+n\alpha/2)}
		g_{\widehat{c}_{+}t}(x,y)
		\label{para_expansion_2},\\
		\sum_{n=0}^{\infty}
		\overline{p} \otimes |H|^{\otimes n} (0,x;t,y)
		&< \infty. \label{para_expansion_3}
	\end{align}
\end{Lem}
\begin{proof}
	We first show \eqref{para_expansion_2}.
	For $n=0$, it is obvious from the definition of $\overline{p}$ and Assumption \ref{Ass_1} (ii).
	For $n=1$, we have
	\begin{align*}
		\overline{p} \otimes |H| (0,x;t,y)
		&=\int_{0}^{t} \int_{\real^d}
		\overline{p} (0,x;t_1,y_1) |H (t_1,y_1;t,y)| 
		\rd y_1 \rd t_1\\
		&\leq
		\overline{p} \otimes |H^b| (0,x;t,y)
		+
		\overline{p} \otimes |H^a| (0,x;t,y),
	\end{align*}
	where $|H^b|$ and $|H^a|$ are defined by
	\begin{align*}  
		|H^b| (s,z;t,y)
		&=
		\left|
		\langle
		\nabla_x \overline{p}(s,z;t,y),  b(z)
		\rangle
		\right|, 
		|H^a| (s,z;t,y)
		=
		\left|
		\sum_{i,j=1}^{d}
		\frac{a_{i,j}(z)-a_{i,j}(y)}{2}
		\partial_{z_{i}z_{j}}^2
		\overline{p}(s,z;t,y)
		\right|.
	\end{align*}	
	From the Gaussian upper bound \eqref{bound_qt} for $\overline{p}$, we obtain
	\begin{align*}  
		\overline{p} \otimes |H^b| (0,x;t,y)
		&\leq
		\widehat{C}_{+}
		\int_{0}^{t} \int_{\real^d} 
			g_{\widehat{c}_{+} s } (x,z)
			\left|
				\langle
					\nabla_x \overline{p}(s,z;t,y),  b(s,z)
				\rangle
			\right| 
		\rd z \rd s\\
		&=
		\widehat{C}_{+}
		\int_{0}^{t} 
			\e
				\left[
					\left|
						\langle
							\nabla_x \overline{p}(s,\sqrt{\widehat{c}_+}W_{s}+x;t,y),
							b(s,\sqrt{\widehat{c}_+}W_{s}+x)
						\rangle
					\right|
				\right]
		\rd s.
	\end{align*}
	It follows from H\"older's inequality, the Gaussian upper bound \eqref{bound_drev_2} for $\nabla_x \overline{p}$, Lemma \ref{moment_0} and analogy of Lemma \ref{key_1} (ii) that for $p_1, p_2 >1$ with $p_1 \in (1, \frac{d}{d-1})$ and $1/p_1 + 1/p_2 =1$, 
	\begin{align*}  
		\overline{p} \otimes |H^b| (0,x;t,y)
		&\leq 
		\widehat{C}_{+}
			\int_{0}^{t} 
				\e
					\left[
						\left|
							\nabla_{x} \overline{p}(s,\sqrt{\widehat{c}_+}W_{s}+x;t,y)
						\right|
						\left|
							b(t_1,\sqrt{\widehat{c}_+}W_{s}+x)
						\right|
					\right]
			\rd s\\
		&\leq
		\widehat{C}_{+}
			\int_{0}^{t} 
				\e
				\left[
					\left|
						\nabla_{x} \overline{p}(s,\sqrt{\widehat{c}_+}W_{s}+x;t,y)
					\right|^{p_1}
				\right]^{1/p_1}
				\e
					\left[
						\left|
							b(s,\sqrt{\widehat{c}_+}W_{s}+x)
						\right|^{p_2}
					\right]^{1/p_2}
			\rd s\\
		&\leq
		C
		(1+|x|)	
		\int_{0}^{t}
			\frac{1}{(t-s)^{1/2}}
			\e
				\left[
					\left|
						g_{\widehat{c}_{+}(t-s)}(\sqrt{\widehat{c}_+}W_{s}+x,y)
					\right|^{p_1}
			\right]^{1/p_1} 
		\rd s\\
		&\leq
		C'(1+|x|)	
		g_{\widehat{c}_{+}(t-s)}(x,y)
		\int_{0}^{t}
			\frac{1}{(t-s)^{1/2+d(p_1-1)/2p_1}} 
		\rd s,
	\end{align*}
	for some $C,C'>0$.
	By using Assumption \ref{Ass_1} (ii), Lemma \ref{Basic_est_0} and Chapman--Kolmogorov equation, we have
	\begin{align*}
		\overline{p} \otimes |H^a| (0,x;t,y)
		&\leq
		\frac{\|a\|_\alpha}{2}
		\int_{0}^{t} \int_{\real^d} 
			\overline{p} (0,x;s,z)
			|y-z|^{\alpha}
			\left|
			\partial_{z_{i}z_{j}}^2
			\overline{p}(s,x;t,y)
			\right|
		\rd z \rd s\\
		&=
		\frac{\|a\|_\alpha}{2}
		\int_{0}^{t} \int_{\real^d} 
			\overline{p} (0,x;s,z)
			|y-z|^{\alpha}
			\left|
			H_{(t-s)a(y)}^{i,j}(y-z)
			\right|
			\overline{p}(s,z;t,y)
		\rd z \rd s\\
		&\leq
		C
		\int_{0}^{t}\int_{\real^d}
			\frac{1}{(t-s)^{1-\alpha/2}}
			g_{\widehat{c}_{+}s} (x,z)
			g_{2\widehat{c}_{+}(t-s)}(z,y) 
		\rd z \rd s\\
		&\leq
		C'
		g_{c_{+}t}(x,y)
		\int_{0}^{t}
			\frac{1}{(t-s)^{1-\alpha/2}} 
		\rd s,
	\end{align*}
	for some $C,C'>0$ and $c_{+}>0$.
	By choosing $p_1=\frac{d}{d-(1-\alpha)}$,  \eqref{para_expansion_2} holds for $n=1$.
	We assume that \eqref{para_expansion_2} holds for $n-1$.
	Then we have 
	\begin{align*}
		\overline{p} \otimes |H|^{\otimes n} (0,x;t,y)
		&=
		\int_{0}^{t} \int_{\real^d}
			\overline{p} \otimes |H|^{\otimes n-1} (0,x;s,z)
			|H (s,z;t,y)| 
		\rd z \rd s\\
		&\leq 
		\frac
			{\left(C_{+}(1+|x|)\Gamma(\alpha/2)\right)^{n-1}}
			{\Gamma(1+(n-1)\alpha/2)}
		\int_{0}^{t} \int_{\real^d}
			s^{(n-1)\alpha/2}
			g_{c_{+}s}(x,z)
			|H (s,z;t,y)| 
		\rd z \rd s.
	\end{align*}
	Hence from the same arguments for $n=1$, it holds that
	\begin{align*}
		\overline{p} \otimes |H|^{\otimes n} (0,x;t,y)
		&\leq 
		\frac
			{ \left(C_{+}(1+|x|))^n \Gamma(\alpha/2 \right)^{n-1}}
			{\Gamma(1+(n-1)\alpha/2)}
			g_{\widehat{c}_+t}(x,y)
		\int_{0}^{t}
			\frac{s^{(n-1)\alpha/2}}{(t-s)^{1-\alpha/2}} 
		\rd s\\
		&=
		\frac
			{ \left(C_{+}(1+|x|)t^{\alpha/2} \Gamma(\alpha/2) \right)^n}
			{\Gamma(1+n\alpha/2)}
		g_{\widehat{c}_{+}t}(x,y).
	\end{align*}
	Hence \eqref{para_expansion_2} holds for every $n \in \n \cup \{0\}$. By using \eqref{para_expansion_2}, we obtain \eqref{para_expansion_3}.
\end{proof}

Now we define an approximation $b_N$ of $b$ by 
\begin{align*}
	b_N(x)
	=
	\left\{ \begin{array}{ll}
		\displaystyle
		b(x)
		&\text{ if } |x| \leq N,  \\
		\displaystyle
		b(Nx/|x|)
		&\text{ if } |x|  > N.
	\end{array}\right.
\end{align*}
Then  from Assumption \ref{Ass_1} (ii), (iii) and Assumption \ref{Ass_3}, $b_N$ satisfies the following conditions 
\begin{itemize}
	\item[(i)]
	$b_N$ is bounded, that is, $|b_N(x)| \leq K(1+N)$ for any 
	$x \in \real^d$.
	\item[(ii)]
	For any $x \in  \real^d$,  $b_N(x) \rightarrow b(x)$ as $N \rightarrow +\infty$.
	\item[(iii)]
	$b_N$ is of linear growth uniformly in $N \in \n$, that is, there exists $K>0$ such that for any $x \in \real^d$,
	\begin{align*}
		\sup_{N\in \n}
		|b_N(x)|
		\leq
		K(1+|x|).
	\end{align*}
	\item[(iv)]
	$b_N$ is of sub-linear growth uniformly in $N\in \n$, that is, for any $\delta>0$, there exists $K(\delta) >0$ such that  for any $x \in \real^d$
	\begin{align*}
		\sup_{N \in \n}
		|b_N(x)|
		\leq
		\delta | x|+K(\delta).
	\end{align*}
\end{itemize}

\begin{proof}[Proof of Theorem \ref{main_expa_0}]
We consider the following SDE with drift coefficient $b_N$
\begin{align*}
	X_t^{N,x}
	=
	x
	+\int_{0}^{t}
	b_N(X_s^{N,x})
	\rd s
	+\int_{0}^{t}
	\sigma(X_s^{N,x})
	\rd W_s,
	~t \geq 0,~x\in \real^d.
\end{align*}
Then $X_t^{N,x}$ admits a pdf, denoted by $p_t^N(x,\cdot)$, and from \eqref{para_expansion_smooth} the following parametrix expansion holds 
\begin{align*}
	p_t^N(x,y)
	=
	\sum_{n=0}^{\infty}
	\overline{p} \otimes H_N^{\otimes n} (s,x;t,y),
\end{align*}
where $H_N$ is defined by
\begin{align*}
	H_N(s,x;t,y)
	:=\langle
	\nabla_x \overline{p}(s,x;t,y),  b_N(x)
	\rangle
	+
	\sum_{i,j=1}^{d}
	\frac{a_{i,j}(x)-a_{i,j}(y)}{2}
	\partial_{x_ix_j}^2
	\overline{p}(s,x;t,y).
\end{align*}
Moreover, since $b_N$ is of linear growth uniformly in $N \in \n$, and from Lemma \ref{para_uni_bdd_1}, we have the following estimation uniformly in $N \in \n$
\begin{align*}
	\sum_{n=0}^{\infty}
	\overline{p} \otimes  |\sup_{N \in \n} H_{N}|^{\otimes n} (s,x;t,y)
	<
	\infty,
\end{align*}
for all $0 \leq s < t \leq T$ and $x, y \in \real^d$.
By using comparison Theorem \ref{main_comp_0} and Lemma \ref{para_uni_bdd_1},  \eqref{para_expansion_3} and dominated convergence theorem, we have
\begin{align*}
p_t(x,y)
=\lim_{N \rightarrow +\infty}p_t^N(x,y)
=\lim_{N \rightarrow +\infty}
\sum_{n=0}^{\infty}
\overline{p} \otimes H_N^{\otimes n} (s,x;t,y)
=\sum_{n=0}^{\infty}
\overline{p} \otimes H^{\otimes n} (s,x;t,y).
\end{align*}
for all $x \in \real$ and a.e. $y \in \real^d$.
Hence we conclude the statement \eqref{para_expansion_1}. 
\end{proof}

\subsection{Application to unbiased simulation scheme}\label{Sub_unbi}

In this subsection, we introduce a probabilistic representation of a pdf of a solution of  Markovian SDE \eqref{SDE_markov},  in order to provide an unbiased simulation scheme.

We introduce the define of counting process.
\begin{Def}
	Let $R_t:=\sum_{n=1}^{\infty} \1(\tau_n \leq t)$ where $(\tau_n-\tau_{n-1})_{n \in \n}$ with $\tau_0=0$ are independent and identically distributed random variables with pdf $\zeta$.
	Then we call $R=(R_t)_{t \geq 0}$ the counting process with $\pi:=(\tau_n)_{n \in \n}$ and $\zeta$.
\end{Def}

\begin{Eg}
	Let $\zeta(t):=\lambda e^{-\lambda t} \1_{[0,\infty)}(t)$.
	Then $R=(R_t)_{t \geq 0}$ is a Poisson process with parameter $\lambda >0$.
	Another choice of $\zeta$ is $\zeta(t):=\frac{A}{t^{\beta} }\1_{[0,2T]}(t)$ where $A:=(1-\beta)/(2T)^{1-\beta}$ and $\beta \in (0,1)$.
	For more on this, see \cite{AnKo}.
\end{Eg}

The following lemma plays a crucial role in our argument.

\begin{Lem}[Lemma 7.3 in \cite{KoTaZh}]\label{prob_rep_3}
	Let $R=(R_t)_{t \geq 0}$ be a counting process with jumps times $\pi:=(\tau_n)_{n \in \n}$ and $\zeta$.
	Then for any $t>0$, $n \in \n$ and any bounded measurable function $V_n:\real^n \to \real$,
	\begin{align*}
	&\e[\1_{\{R_{t}=n\}}V_n(\tau_1,\ldots,\tau_n)] \\
	&=
	\int_0^{t} \rd s_n \int_0^{s_n} \rd s_{n-1} \cdots \int_0^{s_2} \rd s_1
		V_n(s_1, \ldots, s_n) (1-F_{\zeta}(t-s_n)) \prod_{i=0}^{n-1}\zeta(s_{i+1}-s_{i}),
	\end{align*}
	where $F_{\zeta}(x):=\int_{-\infty}^x \zeta(y) dy$ and $s_0=0$.
\end{Lem}

We define $\phi_t^y(x):=\overline{p}(0,x;t,y)$ and
\begin{align*}
	\widehat{\theta}_{t}(x,y)
	&:=
	-\sum_{i=1}^d b_{i}(x) H_{ta(y)}^{i}(y-x)
	+\sum_{i,j=1}^{d} \frac{a_{i,j}(x)-a_{i,j}(y)}{2} H_{ta(y)}^{i,j} (y-x).
\end{align*}
Then it holds that $H(s,x;t,y)=\widehat{\theta}_{t-s}(x,y) \phi_{t-s}^y(x)$.
By using this fact, we obtain the following representation for a pdf, which provide an unbiased simulation scheme for an expectation $\e[f(X_t^x)]$.

\begin{Cor}\label{prob_rep}
	Suppose that Assumption \ref{Ass_1} (ii), (iii) and Assumption \ref{Ass_3} hold.
	Let $R=(R_t)_{t \geq 0}$ be a counting process with $\pi:=(\tau_n)_{n \in \n}$ and $\zeta$, which is independent from $W$.
	For any $t \in (0,T]$, a pdf $p_t(x,\cdot)$ of $X_t^x$ satisfies the following probabilistic representation
	\begin{align*}
		p_t(x,y)
		=
		\e
		\left[
			\frac{\overline{p}(t-\tau_t,x,X_{\tau_t}^{*, \pi}(y))} {1-F_{\zeta}(t-\tau_t)} \Gamma_t(y)
		\right],
	\end{align*}
	where $\tau_t:=\tau_{R_t}$, $F_{\zeta}(t):=\int_{0}^{t} \zeta(s) \rd s$ and
	\begin{align*}
		\Gamma_t(y)
		:=
		\1_{\{R_t=0\}}
		+\prod_{j=0}^{R_t-1}
			\frac
				{\widehat{\theta}_{\tau_{j+1}-\tau_j}(X_{\tau_{j+1}}^{*, \pi}(y), X_{\tau_{j}}^{*, \pi}(y))}
				{\zeta(\tau_{j+1}-\tau_{j})}
			\1_{\{R_t \geq 1\}},
	\end{align*}
	and $X^{*, \pi}(y)$ is the Euler-Maruyama scheme with $X_0^{*, \pi}(y) = y$ and a random partition $\pi$ which drift coefficient is zero and diffusion coefficient is $\sigma$, that is, $X_0^{*, \pi}(y):=y$ and for $j \geq 1$,
	\begin{align*}
	X_{\tau_j}^{*, \pi}(y)
	:=X_{\tau_{j-1}}^{*, \pi}(y)
	+\sigma(X_{\tau_{j-1}}^{*, \pi}(y)) (W_{\tau_j}-W_{\tau_{j-1}}).
	\end{align*}
	Moreover, for a random variable $Z$ with a pdf $g$, which is independent from $R$ and $W$, and for any measurable function $f:\real^d \to \real$ with $\e[|f(X_t^x)|]<\infty$, it holds that
	\begin{align*}
		\e[f(X_t^x)]
		&=
		\e
			\left[
				\frac{f(Z)}{g(Z)}\frac{\overline{p}(t-\tau_t,x,X_{\tau_t}^{*, \pi}(Z))}{1-F_{\zeta}(t-\tau_t)} \Gamma_t(Z)
			\right].
\end{align*}
\end{Cor}

Since we have the expansion \eqref{para_expansion_1} for unbounded drift coefficient $b$, we can prove Corollary \ref{prob_rep} by a similar argument to the proof of Theorem 5.7, Proposition 7.2 in \cite{BaKo15}, Theorem 7.4 in \cite{KoTaZh} and Theorem 3.2 in \cite{KoYu}.
For the convenience of the reader, we will give a proof below.

\begin{proof}[Proof of Corollary \ref{prob_rep}]
	It follows from \eqref{para_expansion_3} that the series $\sum_{n=0}^{\infty}\overline{p} \otimes H^{\otimes n} (0,x;t,y)$ converges absolutely and uniformly for $(t,y) \in (0,T] \times \real^d$, thus using Fubini's theorem and the equation  $H(s,x;t,y)=\widehat{\theta}_{t-s}(x,y) \phi_{t-s}^y(x)$, we have for each $n \in \n \cup \{0\}$,
	\begin{align*}
		&\overline{p} \otimes H^{\otimes n} (0,x;t,y)
		=
		\int_0^{t_0} \rd t_1 \cdots \int_0^{t_{n-1}} \rd t_n
			\int_{\real^n} \rd y_1 \cdots \rd y_n
				\prod_{i=0}^{n-1}
					\widehat{\theta}_{t_{i}-t_{i+1}}(y_{i+1}, y_{i})
					\phi_{t_{i}-t_{i+1}}^{y_{i}}(y_{i+1})
				\phi_{t_{n}}^{y_n}(x),\\
		&=
		\int_0^{t} \rd s_n \int_0^{s_n} \rd s_{n-1} \cdots \int_0^{s_2} \rd s_1
			\int_{\real^n} \rd y_1 \cdots \rd y_n
				\prod_{i=0}^{n-1}
					\widehat{\theta}_{s_{i+1}-s_{i}}(y_{i+1}, y_{i})
					\phi_{s_{i+1}-t_{i}}^{y_{i}}(y_{i+1})
				\phi_{t-s_{n}}^{y_n}(x),
	\end{align*}
	where $(t_0,y_{n+1}):=(t,y)$ and in the last equality we use the change of variables $s_n=t_0-t_n$.	
	For any partition $\pi_0=(s_i)_{i \in \n}$ with $0=:s_0 \leq s_1 < \cdots < s_n< \cdots <\infty$, we define a Markov chain $X^{*, \pi_0}(y)$ as follows:
	$X_0^{*, s_0}(y):=y$ and $\p(X_{s_i}^{*, \pi_0}(y) \in \rd y_{i+1} | X_{s_{i-1}}^{*, \pi_0}(y)=y_{i})=\varphi_{s_{i}-s_{i-1}}^{y_i} (y_{i+1}) \rd y_{i+1}$.
	Then, by using the Markov property of stochastic process $y+\sigma(y)W_t$ whose density is $\phi_t^y$, the parametrix expansion \eqref{para_expansion_1} and Lemma \ref{prob_rep_3} we conclude the statement.
\end{proof}

\subsection{H\"older continuity of pdf}\label{Sub_Hol}

In this subsection, by using a regularity of $q(s,x;t,y)$, we prove that if the diffusion matrix is smooth, then a pdf $p_t(x,\cdot)$ of a solution of SDE \eqref{SDE_1} is H\"older continuous.

\begin{Thm}\label{main_3}
	Suppose Assumption \ref{Ass_1} and \ref{Ass_2} hold, and $\sigma(t,\cdot) \in C_b^2(\real^d; \real^{d \times d})$ for all $t \in [0,T]$.
	Let $p_t(x,\cdot)$ be a continuous version of a pdf of a solution $X_t^{x}$ to SDE \eqref{SDE_1}, for $t \in (0,T]$ and $x \in \real^d$.
	Assume $\gamma \in (0,1)$ and $p_1,p_2,p_3>1$ with $p_1 \in (1,\frac{d}{d-(1-\gamma)})$, $1/p_1+1/p_2+1/p_3=1$.
	\begin{itemize}
		\item[(i)]
		There exists $C_{\gamma,p_1}>0$ and $c_{p_1}>0$ such that, for all $(t,x,y) \in (0,T] \times \real^d \times \real^d$
		\begin{align*}
			&|p_t(x,y)-p_t(x,y')|\notag\\
			&\leq
			C_{\gamma,p_1}
			\sup_{0 \leq s \leq t}
			\e\left[
			Z_t(1,Y^{0,x})^{p_2}
			\right]^{1/p_2}
			\e\left[
			|b(s,Y^{0,x})|^{p_3}
			\right]^{1/p_3}
			\frac{|y-y'|^{\gamma}}{t^{\gamma/2}} \left\{ g_{c_{p_1}t}(x,y) + g_{c_{p_1}t}(x,y') \right\}.
		\end{align*}
		
		\item[(ii)]
		Recall that $t_r$ is defined in \eqref{def_tr}.
		There exists $C_{\gamma,p_1}>0$ and $c_{p_1}>0$ such that, for all $(t,x,y) \in (0,T] \times \real^d \times \real^d$, it holds that if $t \in (0,t_{p_1}]$,
		\begin{align*}
		&|p_t(x,y)-p_t(x,y')|\\
		&\leq
		C_{\gamma,p_1}
		(1+|x|)
		\exp
		\left(
		c_{p_1}
		(1+|x|^2)t
		\right)
		\frac{|y-y'|^{\gamma}}{t^{\gamma/2}}
		\left
		\{ g_{c_{p_1}t}(x,y) + g_{c_{p_1}t}(x,y')
		\right\},
		\end{align*}
		and if $t \in (t_{p_1},T]$, then
		\begin{align*}
		&|p_t(x,y)-p_t(x,y')|\\
		&\leq
		C_{\gamma,p_1}
		(1+|x|)
		\exp
		\left(
		\frac{|x|^2}{8p_2 \widehat{c}_{+} T}
		\right)
		\frac{|y-y'|^{\gamma}}{t^{\gamma/2}}
		\left\{
		g_{c_{p_1}t}(x,y) + g_{c_{p_1}t}(x,y')
		\right\}.
		\end{align*}
	\end{itemize}
\end{Thm}

\begin{Rem}
	Theorem \ref{main_3} implies that the regularity of a pdf of a solution to SDE \eqref{SDE_1} does not depend on the regularity of drift coefficient.
\end{Rem}

From Theorem \ref{main_2} (i), the representation \eqref{density_1} is continuous.
In order to prove Theorem \ref{main_3}, we need to consider the H\"older continuity of the pdf $q(0,x,t;\cdot)$ and its derivative.

\begin{Lem}\label{holder_0}
	Suppose Assumption \ref{Ass_1} and \ref{Ass_2} hold, and $\sigma(t,\cdot) \in C_b^2(\real^d; \real^{d \times d})$ for all $t \in [0,T]$.
	Then there exist $\widetilde{C}_{+}>0$ and $\widetilde{c}_{+}>\widehat{c}_{+}>0$ such that for any $\gamma \in (0,1)$, $y,y'\in\real^d$, $0 \leq s < t \leq T$ and $i=1,\ldots,d$,
	\begin{align}
		|q(s,x;t,y)-q(s,x;t,y')|
		&\leq
		\frac{\widetilde{C}_{+}|y-y'|^{\gamma}}{t^{\gamma/2}} \left\{g_{\widetilde{c}_{+}(t-s)}(x,y)+g_{\widetilde{c}_{+}(t-s)}(x,y')\right\} \label{Holder_1},\\
		|\partial_{x_i} q(s,x;t,y)- \partial_{x_i} q(s,x;t,y')|
		&\leq
		\frac{\widetilde{C}_{+}\left|y-y'\right|^{\gamma}}{(t-s)^{1-(1-\gamma)/2}} \left\{g_{\widetilde{c}_{+}(t-s)}(x,y)+g_{\widetilde{c}_{+}(t-s)}(x,y')\right\} \label{Holder_2}.
	\end{align}
\end{Lem}
\begin{proof}
	We only prove \eqref{Holder_2}.
	The proof of \eqref{Holder_1} is similar.
	Note that if $\sigma(t,\cdot) \in C_b^2$, then it holds that there exist $\widehat{C}_{+}>0$ and $\widehat{c}_{+}>0$ such that for all $t \in (0,T]$, $x,y \in \real^d$ and $i,j=1,\ldots,d$,
	\begin{align}\label{upper_deriv_0}
		\left|\partial_{x_i} q(s,x;t,y) \right|
		\leq \frac{\widehat{C}_{+}}{(t-s)^{1/2}} g_{\widehat{c}_{+}(t-s)}(x,y)
		\text{ and }
		\left|\partial_{x_i}\partial_{y_j} q(s,x;t,y) \right|
		\leq \frac{\widehat{C}_{+}}{t-s} g_{\widehat{c}_{+}(t-s)}(x,y),
	\end{align}
	(see, Chapter 9, section 6, Theorem 7 in \cite{Fr64}).
	
	We first assume $|y-y'|^2 < t-s$.
	Then by the mean-value theorem and \eqref{upper_deriv_0}, we have
	\begin{align*}
	|\partial_{x_i} q(s,x;t,y)- \partial_{x_i} q(s,x;t,y')|
	&=\left| \int_{0}^{1} \langle y-y' , \nabla_{y} \partial_{x_i} q(s,x;t,\zeta(\theta)) \rangle \rd \theta \right|\\
	&\leq
	\frac{\sqrt{d}\widehat{C}_{+}\left|y-y'\right|}{t}
	\int_{0}^{1} g_{\widehat{c}_{+}(t-s)}(x,\zeta(\theta)) \rd \theta,
	\end{align*}
	where $\zeta(\theta) = \theta y+(1-\theta)y'$.
	Since $|y-y'|^2 < t-s$, by using $|x_1-x_2|^2 \geq \frac{1}{2}|x_1|^2-|x_2|^2$ for any $x_1,x_2 \in \real^d$, we have
	\begin{align*}
	-\frac{|\zeta(\theta)-x|^2}{\widehat{c}_{+}(t-s)}
	\leq-\frac{|y-x|^2}{2\widehat{c}_{+}(t-s)}+\frac{(1-\theta)|y-y'|^2}{\widehat{c}_{+}(t-s)}
	\leq-\frac{|y-x|^2}{2\widehat{c}_{+}(t-s)}+\frac{1}{\widehat{c}_{+}},~
	-\frac{|\zeta(\theta)-x|^2}{\widehat{c}_{+}(t-s)}\leq-\frac{|y'-x|^2}{2\widehat{c}_{+}(t-s)}+\frac{1}{\widehat{c}_{+}}.
	\end{align*}	
	Hence we have
	\begin{align*}
	&|\partial_{x_i} q(s,x;t,y)- \partial_{x_i} q(s,x;t,y')|\\
	&\leq
		\frac
			{C'\left|y-y'\right|^{\gamma}}
			{(t-s)^{1-(1-\gamma)/2}}
		\frac
			{\left( |y-x|^{1-\gamma}+|y'-x|^{1-\gamma}\right)}
			{(t-s)^{(1-\gamma)/2}}
		\exp
			\left(
				-\frac{|y-x|^2}{16\widehat{c}_{+}(t-s)}
			\right)
		\exp
			\left(
				-\frac{|y'-x|^2}{16\widehat{c}_{+}(t-s)}
			\right)\\
		&\quad\times
		\left\{
			g_{4\widehat{c}_{+}(t-s)}(x,y)
			+g_{4\widehat{c}_{+}(t-s)}(x,y')
		\right\} \\
	&\leq
		\frac
			{C''\left|y-y'\right|^{\gamma}}
			{(t-s)^{1-(1-\gamma)/2}}
		\left\{
			g_{4\widehat{c}_{+}(t-s)}(x,y)
			+g_{4\widehat{c}_{+}(t-s)}(x,y')
		\right\},
	\end{align*}
	for some $C',C''$.
	This concludes \eqref{Holder_2} for $|y-y'|^2<t-s$.
	
	If $|y-y'|^2\geq t$, by using \eqref{upper_deriv_0}, we have
	\begin{align*}
		&
			\left|
				\partial_{x_i} q(s,x;t,y)
				-
				\partial_{x_i} q(s,x;t,y')
			\right|
		\leq
			|\partial_{x_i} q(s,x;t,y)|
			+|\partial_{x_i} q(s,x;t,y')|\\
		&\leq
			\frac
				{2\widehat{C}_{+}(t-s)^{\gamma/2}}
				{(t-s)^{1-(1-\gamma)/2}}
			\left\{
				g_{\widehat{c}_{+}(t-s)}(x,y)
				+g_{\widehat{c}_{+}(t-s)}(x,y')
			\right\}
		\leq \frac{2\widehat{C}_{+}|y-y'|^{\gamma}}{(t-s)^{1-(1-\gamma)/2}} \left\{g_{\widehat{c}_{+}(t-s)}(x,y)+g_{\widehat{c}_{+}(t-s)}(x,y')\right\},
	\end{align*}
	which concludes \eqref{Holder_2} for $|y-y'|^2\geq t-s$.
\end{proof}

\begin{proof}[Proof of Theorem \ref{main_3}]
	From Theorem \ref{main_1} and Lemma \ref{holder_0}, it is sufficient to estimate
	\begin{align}\label{main_3_1}
		\left|
		\int_0^t
		\e\left[
		\langle
		\nabla_x q(x,X_s^{x};t,y)
		-\nabla_x q(s,X_s^{x};t,y')
		, b(s,X^{x})
		\rangle
		\right]
		\rd s
		\right|.
	\end{align}
	By using Theorem \ref{main_0}, H\"older's inequality, Jensen's inequality, Lemma \ref{moment_0}, Lemma \ref{holder_0} and \eqref{bound_qt}, for any $p_1,p_2,p_3>1$ with $1/p_1+1/p_2+1/p_3=1$, \eqref{main_3_1} is bounded by
	\begin{align}\label{main_3_2}
		&\int_0^t
		\e\left[
		\left|
		\nabla_x q(s,Y_s^{0,x};t,y)-\nabla_x q(s,Y_s^{0,x};t,y')
		\right|^{p_1}
		\right]^{1/p_1}
		\e\left[
		Z_s(1,Y^{0,x})^{p_2}
		\right]^{1/p_2}
		\e\left[
		|b(s,Y^{0,x})|^{p_3}
		\right]^{1/p_3}
		\rd s \notag\\
		&\leq
		\sup_{0 \leq s \leq t}
		\e\left[
		Z_t(1,Y^{0,x})^{p_2}
		\right]^{1/p_2}
		\e\left[
		|b(s,Y^{0,x})|^{p_3}
		\right]^{1/p_3}
		\notag\\
		&\quad \times
		\int_0^t
		\left(
		\int_{\real^d}
		\left|
		\nabla_x q(s,z;t,y)-\nabla_x q(s,z;t,y')
		\right|^{p_1}
		q(0,x;s,z)
		\rd z
		\right)^{1/p_1}
		\rd s \notag \notag\\
		&\leq
		C
		\sup_{0 \leq s \leq t}
		\e\left[
		Z_t(1,Y^{0,x})^{p_2}
		\right]^{1/p_2}
		\e\left[
		|b(s,Y^{0,x})|^{p_3}
		\right]^{1/p_3}
		\notag\\
		&\quad \times
		\int_0^t
		\left(
		\int_{\real^d}
		\frac
		{
			|y-y'|^{\gamma p_1}
			\left\{
			g_{\widetilde{c}_{+}(t-s)/p_1}(z,y)
			+g_{\widetilde{c}_{+}(t-s)/p_1}(z,y')
			\right\}
		}
		{(t-s)^{p_1(1-(1-\gamma)/2)+(p_1-1)d/2}}
		g_{\widetilde{c}_{+}s}(x,z)
		\rd z
		\right)^{1/p_1}
		\rd s,
	\end{align}
	for some $C>0$.
	Since $1/p_1<1$, $g_{c(t-s)/p_1}(z,y) \leq C_{p_1} g_{c(t-s)}(z,y)$ for some $C_{p_1}>0$ and since $\gamma \in (0,1)$, it holds that
	\begin{align*}
	\rho
	:=
	(1-(1-\gamma)/2)+\frac{(p_1-1)d}{2p_1}
	<1
	\quad \Leftrightarrow \quad
	1<p_1<\frac{d}{d-(1-\gamma)}.
	\end{align*}
	Therefore, by using Chapman--Kolmogorov equation, \eqref{main_3_2} is estimated by
	\begin{align*}
	&
	C'\sup_{0 \leq s \leq t}
	\e\left[
	Z_t(1,Y^{0,x})^{p_2}
	\right]^{1/p_2}
	\e\left[
	|b(s,Y^{0,x})|^{p_3}
	\right]^{1/p_3}
	\frac{C_{p_1}^{1/_{p_1}} t^{1-\rho}}{1-\rho}
	|y-y'|^{\gamma}
	\left\{
	g_{\widetilde{c}_{+}t}(x,y)
	+g_{\widetilde{c}_{+}t}(x,y')
	\right\}^{1/p_1}\\
	&\leq 
	C''
	\sup_{0 \leq s \leq t}
	\e\left[
	Z_t(1,Y^{0,x})^{p_2}
	\right]^{1/p_2}
	\e\left[
	|b(s,Y^{0,x})|^{p_3}
	\right]^{1/p_3}
	\frac{t^{1-\rho+\frac{d}{2} (1-\frac{1}{p_1})}}{1-\rho}
	|y-y'|^{\gamma}
	\left\{
	g_{\widetilde{c}_{+}p_1t}(x,y)+g_{\widetilde{c}_{+}p_1t}(x,y')
	\right\},
	\end{align*}
	for some $C',C''>0$.
	Since $t^{1-\rho+\frac{d}{2} (1-\frac{1}{p_1})}=t^{(1-\gamma)/2}\leq T t^{-\gamma/2}$, we conclude the statement (i).
	Lemma \ref{moment_0} and Lemma \ref{bdd_Girsanov_1} with $r=p_2$ imply the statement (ii).
\end{proof}

\section{One-dimensional SDEs with super--linear growth coefficients}\label{Sec_super}

In this section, inspired by \cite{HuJeKl12,Sa13,Sa16}, we use a ``one-step" tamed Euler--Maruyama approximation, in order to prove the existence of a pdf for a solution of one--dimensional Markovian SDEs
\begin{align}\label{super_SDE_0}
	X_t
	=x
	+\int_{0}^{t}
		b(s,X_s)
	\rd s
	+\int_{0}^{t}
		\sigma(X_s)
	\rd W_s,~(t,x) \in [0,T] \times \real,
\end{align}
under the assumption that the coefficients $b:[0,T]\times \real \to \real$ and $\sigma:\real \to \real$ satisfies the following conditions.

\begin{Ass}\label{Ass_super_0}
	We suppose that the coefficients $b:[0,T] \times \real \to \real$ and $\sigma:\real \to \real$ are measurable and satisfy the following conditions:
	\begin{itemize}
		\item[(i)]
		(Khasminskii and one-sided Lipschitz condition)
		There exist $K>0$, $p_0 >2$ and $p_1 >2$ such that for each $x,y \in \real$ and $s \in [0,T]$,
		\begin{align*}
		2 x b(s,x)
		+(p_0-1)|\sigma(x)|^2
		&\leq
		K(1+|x|^{2}),\\
		2(x-y) (b(s,x)-b(s,y))
		+(p_1-1)|\sigma(x)-\sigma(y)|^2
		&\leq
		K|x-y|^2.
		\end{align*}
		
		\item[(ii)]
		(locally Lipschitz continuous)
		There exist $K>0$ and $\ell \in (0,\frac{p_0-2}{4}]$ such that for each $x,y \in \real$ and $s \in [0,T]$,
		\begin{align*}
		|b(s,x)-b(s,y)|
		&\leq
		K(1+|x|^{\ell}+|y|^{\ell}) |x-y|,\\
		|b(s,x)|
		&\leq
		K(1+|x|^{\ell+1}).
		\end{align*}

	\end{itemize}
\end{Ass}

\begin{Eg}
		Let $b(s,x) \equiv b(x):=\lambda x(\mu-|x|)$ and $\sigma(x):=\xi |x|^{3/2}$ for some $\lambda, \mu, \xi>0$, then the SDE \eqref{super_SDE_0} is called Heston-$3/2$ volatility model (see, e.g. \cite{GoMa13}) and for any $p_0 \leq \frac{2\lambda+|\xi|^2}{|\xi|^2}$ and $p_1 \leq \frac{\lambda+|\xi|^2}{|\xi|^2}$, and $x,y \in \real$, it holds that
		\begin{align*}
		2 x b(x)
		+(p_0-1)|\sigma(x)|^2
		&\leq
		2\lambda \mu(1+|x|^{2}),\\
		2(x-y) (b(x)-b(y))
		+(p_1-1)|\sigma(x)-\sigma(y)|^2
		&\leq
		2\lambda \mu|x-y|^2,\\
		|b(x)-b(y)|
		&\leq
		\lambda (\mu \wedge 1)(1+|x|+|y|) |x-y|,
		\end{align*}
		(see, Appendix in \cite{Sa16}).
\end{Eg}

\begin{Rem}
	\begin{itemize}
		\item[(i)]
		From Assumption \ref{Ass_super_0}, we have there exist $K_0,K_1,K_2>0$ such that for any $x,y \in \real$ and $s \in [0,T]$,
		\begin{align}\label{local_sigma_1}
		(p_0-1)|\sigma(x)|^2
		&\leq
		K(1+|x|^2)
		-2xb(s,x)
		\leq
		K
		+2K|x|
		+K|x|^2
		+2K|x|^{\ell+2} \notag\\
		&\leq
		(p_0-1)K_0(1+|x|^{\ell+2}),
		\end{align}
		and
		\begin{align}
		(p_1-1)|\sigma(x)-\sigma(y)|^2
		&\leq
		K|x-y|^2
		-2(x-y) (b(s,x)-b(s,y)) \notag\\
		&\leq
		K(3+2|x|^{\ell}+2|y|^{\ell}) |x-y|^2 \notag\\
		&\leq
		(p_1-1)K_1(1+|x|^{\ell}+|y|^{\ell}) |x-y|^2 \notag\\
		&\leq
		(p_1-1)K_2(1+|x|^{2\ell}+|y|^{2\ell}) |x-y|^2\label{local_sigma_01}.
		\end{align}
		
		\item[(ii)]
		If the coefficients $b$ and $\sigma$ are local Lipschitz continuous in space,
		then there exists a unique strong solution to the equation \eqref{super_SDE_0}, (e.g., Theorem 5.2.5 in \cite{KS}).
	\end{itemize}
\end{Rem}

Under the above assumptions, we prove the existence of a pdf.

\begin{Thm}\label{super_density_0}
	Suppose that Assumption \ref{Ass_super_0} holds.
	Then for any $t \in (0,T]$, $X_t$ admits a pdf on the set $\{x \in \real ~;~\sigma(x) \neq 0\}$.
\end{Thm}

We derive a several lemmas for proving the above theorem.
The following lemma shows that the Assumption \ref{Ass_super_0} (ii) implies that the solution of SDE \eqref{super_SDE_0} has a moment and satisfies a Kolmogorov type condition.

\begin{Lem}\label{super_moment_0}
	Suppose that Assumption \ref{Ass_super_0} holds.
	Then, for any $p \in [2,p_0]$, there exists $C_p>0$ such that
	\begin{align}\label{super_moment_01}
	\sup_{t \in [0,T]}
	\e[|X_t|^{p}]
	\leq C_{p}.
	\end{align}
	Moreover, there exists $C>0$ such that for any $t,s \in [0,T]$ with $s<t$, we have
	\begin{align}\label{super_moment_02}
	\e[|X_t-X_s|^2]
	\leq C (t-s).
	\end{align}
\end{Lem}
\begin{proof}
	It suffices to prove the statement for $p=p_0$.
	Let $t \in [0,T]$ and $\tau_N:=\inf\{t \geq 0~;~|X_t| \geq N\}$ for $N \in \n$.
	By using It\^o's formula for $|x|^{p_0}$ and Assumption \ref{Ass_super_0} (i), we have
	\begin{align*}
	|X_{t \wedge \tau_N}|^{p_0}
	&=
	|x|^{p_0}
	+\int_{0}^{t \wedge \tau_N}
	\left\{
		p_0X_s^{p_0-1} b(s,X_s)
		+\frac{p_0(p_0-1)}{2} X_s^{p_0-2} \sigma(X_s)^2
	\right\}
	\rd s
	+\int_{0}^{t \wedge \tau_N}
	p_0X_s^{p_0-1} \sigma(X_s)
	\rd W_s\\
	&\leq
	|x|^{p_0}
	+
	\frac{p_0}{2}
	\int_{0}^{t \wedge \tau_N}
	X_s^{p_0-2}
	\{
	2X_sb(s,X_s)
	+(p_0-1) \sigma(X_s)^2
	\}
	\rd s
	+\int_{0}^{t \wedge \tau_N}
	p_0X_s^{p_0-1} \sigma(X_s)
	\rd W_s\\
	&\leq
	|x|^{p_0}
	+
	\frac{p_0K}{2}
	\int_{0}^{t \wedge \tau_N}
	X_s^{p_0-2}
	(1+X_s^2)
	\rd s
	+\int_{0}^{t \wedge \tau_N}
	p_0X_s^{p_0-1} \sigma(X_s)
	\rd W_s.
	\end{align*}
	By taking expectation, there exists $C>0$ such that we have
	\begin{align*}
	\e\left[
	|X_{t \wedge \tau_N}|^{p_0}
	\right]
	\leq
	C
	+C
	\int_{0}^{t}
	\e\left[
	|X_{s \wedge \tau_N}|^{p_0}
	\right]
	\rd s,
	\end{align*}
	and thus Gronwall's inequality implies
	\begin{align*}
	\e\left[
	|X_{t \wedge \tau_N}|^{p_0}
	\right]
	\leq
	Ce^{CT}.
	\end{align*}
	By letting $N \to \infty$, we concludes \eqref{super_moment_01}.
	
	From Assumption \ref{Ass_super_0} (ii), \eqref{local_sigma_1} and \eqref{super_moment_01}, we have
	\begin{align*}
	\int_{s}^{t} \e[|b(u,X_u)|^2] \rd u
	+\int_{s}^{t} \e[|\sigma(X_u)|^2] \rd u
	\leq C(t-s),
	\end{align*}
	for some $C>0$.
	Therefore, by using Jensen's inequality and It\^o's isometry, we conclude \eqref{super_moment_02}.
\end{proof}

Now inspired by \cite{FoPr10} and \cite{HuJeKl12,Sa13,Sa16}, we consider a one-step tamed Euler--Maruyama scheme.
For a fixed $t \in (0,T]$ and $\varepsilon \in (0,1)$, we define a one-step tamed Euler--Maruyama scheme $X^{(\varepsilon)}=(X_u^{(\varepsilon)})_{u \in [0,t]}$ by
\begin{align*}
X_u^{(\varepsilon)}
:=&
\left\{ \begin{array}{ll}
\displaystyle 
X_u,
&\text{ if } u \in [0,t-\varepsilon],  \\
\displaystyle
X_{t-\varepsilon}
+\int_{t-\varepsilon}^{u}
b_{\varepsilon}(s,X_{t-\varepsilon})
\rd s
+\int_{t-\varepsilon}^{u}
\sigma_{\varepsilon}(X_{t-\varepsilon})
\rd W_s,
&\text{ if } u \in (t-\varepsilon,t]
\end{array}\right.\\
=&
x
+\int_{0}^{u}
\overline{b}^{(\varepsilon)}_s
\rd s
+\int_{0}^{u}
\overline{\sigma}^{(\varepsilon)}_s
\rd W_s,
\end{align*}
where
\begin{align*}
b_{\varepsilon}(s,x)
:=\frac{b(s,x)}{1+\varepsilon^{1/2} |x|^{\ell}},
\quad
\sigma_{\varepsilon}(x)
:=\frac{\sigma(x)}{1+\varepsilon^{1/2} |x|^{\ell/2}},
\end{align*}
and
\begin{align*}
\overline{b}^{(\varepsilon)}_s
:=
\left\{ \begin{array}{ll}
\displaystyle
b(s,X_s)
&\text{ if } s \in [0,t-\varepsilon],  \\
\displaystyle
b_{\varepsilon}(s,X_{t-\varepsilon})
&\text{ if } s \in (t-\varepsilon,t],
\end{array}\right.
\quad
\overline{\sigma}^{(\varepsilon)}_s
:=
\left\{ \begin{array}{ll}
\displaystyle
\sigma(X_s)
&\text{ if } s \in [0,t-\varepsilon],  \\
\displaystyle
\sigma_{\varepsilon}(X_{t-\varepsilon})
&\text{ if } s \in (t-\varepsilon,t].
\end{array}\right.
\end{align*}
Then it holds that for any $x \in \real$, $s \in [0,T]$ and $\varepsilon \in (0,1)$,
\begin{align}
|b_{\varepsilon}(s,x)|
&\leq
\{K\varepsilon^{-1/2}(1+|x|)\} \wedge |b(s,x)|,
\label{tamed_coef_0}\\
|\sigma_{\varepsilon}(x)|^2
&\leq
\{K_0\varepsilon^{-1}(1+|x|^2)\} \wedge |\sigma(x)|^2,
\label{tamed_coef_01}
\end{align}
and
\begin{align}
	|b(s,x)-b_{\varepsilon}(s,x)|
	&=\frac{|b(s,x)| |x|^{\ell} \varepsilon^{1/2}}{1+\varepsilon^{1/2} |x|^{\ell}}
	\leq
	K(1+|x|^{\ell+1})|x|^{\ell}\varepsilon^{1/2},
	\label{tamed_coef_1}\\
	|\sigma(x)-\sigma_{\varepsilon}(x)|^2
	&=\frac{|\sigma(x)|^2 |x|^{\ell} \varepsilon}{(1+\varepsilon^{1/2} |x|^{\ell/2})^2}
	\leq 
	K_0(1+|x|^{\ell+2})|x|^{\ell} \varepsilon.
	\label{tamed_coef_2}
\end{align}
Indeed, it holds from Assumption 4.1 (ii) and the inequality $(1+|x|)(1+|x|^{\ell}) \geq 1+|x|^{\ell+1}$ that
\begin{align*}
	|b_{\varepsilon}(s,x)|
	&\leq
	\frac
		{K\varepsilon^{-1/2}(1+|x|^{\ell+1})}
		{1+ |x|^{\ell}}
	\leq
	K\varepsilon^{-1/2}
	(1+|x|),
\end{align*}
which implies \eqref{tamed_coef_0}.
By the similar way, , it holds from \eqref{local_sigma_1} and the inequality $(1+|x|^{l/2})^2 (1+|x|^2) \ge 1+|x|^{l+2}$ that
\begin{align*}
	|\sigma_{\varepsilon}(s,x)|^2
	\leq
	\frac{K_{0}\varepsilon^{-1} (1+|x|^{l+2})}{(1+|x|^{l/2})^2}
	\leq K_0 \varepsilon^{-1} (1+|x|^2),
\end{align*}
which implies \eqref{tamed_coef_01}.
\eqref{tamed_coef_1} and \eqref{tamed_coef_2} are followed from Assumption 4.1 (ii) and \eqref{local_sigma_1}, respectively.

\begin{Rem}
	As mentioned in introduction, Romito \cite{Ro18} studied the existence and Besov regularity of the probability density function of a solution of SDEs with locally bounded drift, and locally H\"older continuous, elliptic diffusion coefficient.
	Therefore, the statement of Theorem \ref{super_density_0} is　included in their results (see, Theorem 4.1 in \cite{Ro18}).
	In \cite{Ro18}, Romito use a localization argument and one-step Euler--Maruyama scheme.
	On the other hand, in our paper, we use directly one-step tamed Euler scheme, so the approach of proof is different.
\end{Rem}

By using Lemma \ref{super_moment_0}, $X^{(\varepsilon)}$ has a moment and some error estimate.

\begin{Lem}\label{super_moment_1}
	Suppose that Assumption \ref{Ass_super_0} holds.
	Then, for any $p \in [2,p_0]$, there exists $C_p>0$ such that
	\begin{align}\label{super_moment_2}
	\sup_{\varepsilon \in (0,1)} \sup_{u \in [0,t]}
	\e[|X_u^{(\varepsilon)}|^{p}]
	\leq C_{p}.
	\end{align}
	Moreover, for any $p \in [1,\frac{2p_0}{\ell+2}]$, there exists $C_p>0$ such that
	\begin{align}\label{super_moment_3}
	\sup_{u \in [t-\varepsilon,t]}
	\e\left[
	|X_u^{(\varepsilon)}-X_{t-\varepsilon}^{(\varepsilon)}|^p
	\right]
	\leq C_p \varepsilon^{p/2}.
	\end{align}
\end{Lem}
\begin{proof}
	If $u \in [0,t-\varepsilon]$, then $X_u^{(\varepsilon)}=X_u$ thus, we have the \eqref{super_moment_2} from Lemma \ref{super_moment_0}.
	
	Now we assume that $u \in [t-\varepsilon,t]$, then by using \eqref{tamed_coef_0} and \eqref{tamed_coef_01} with $|\sigma_{\varepsilon}(x)| \leq |\sigma(x)|$, since $\sigma(X_{t-\varepsilon})$ and $W_u-W_{t-\varepsilon}$ are independent,
	\begin{align*}
	\e\left[
	|X_u^{(\varepsilon)}|^{p}
	\right]
	&\leq
	3^{p-1}
	\e\left[
	|X_{t-\varepsilon}|^{p}
	+
	\varepsilon^{p-1}
	\int_{t-\varepsilon}^{t}
	|b_{\varepsilon}(s,X_{t-\varepsilon})|^{p}
	\rd s
	+
	|\sigma_{\varepsilon}(X_{t-\varepsilon})|^{p}
	|W_u-W_{t-\varepsilon}|^{p}
	\right]\\
	&\leq
	3^{p-1}
	\e\left[
	|X_{t-\varepsilon}|^{p}
	\right]
	+
	3^{p-1}K^{p}
	\varepsilon^{p/2}
	\e[(1+|X_{t-\varepsilon}|)^{p}]\\
	&\quad
	+
	3^{p-1}K_{0}^{p/2}
	\varepsilon^{-p/2}
	\e\left[
	(1+|X_{t-\varepsilon}|^2)^{p/2}
	\right]
	\e\left[
	|W_u-W_{t-\varepsilon}|^{p}
	\right].
	\end{align*}
	Since $\e\left[|W_u-W_{t-\varepsilon}|^{p}\right] \leq C \varepsilon^{p/2}$ for some $C>0$, thus we obtain \eqref{super_moment_2}.
	
	Now we prove \eqref{super_moment_3}.
	By using \eqref{local_sigma_1}, \eqref{tamed_coef_0} and \eqref{tamed_coef_01} since $\sigma(X_{t-\varepsilon})$ and $W_u-W_{t-\varepsilon}$ are independent, we have
	\begin{align}\label{Rem_drift_super}
	\e\left[
	|X_u^{(\varepsilon)}-X_{t-\varepsilon}^{(\varepsilon)}|^{p}
	\right]
	&\leq
	2^{p-1}
	\varepsilon^{p-1}
	\int_{t-\varepsilon}^{t}
	\e[|b_{\varepsilon}(s,X_{t-\varepsilon})|^{p}]
	\rd s
	+
	2^{p-1}
	\e\left[
	|\sigma(X_{t-\varepsilon})|^{p}
	\right]
	\e\left[
	|W_u-W_{t-\varepsilon}|^{p}
	\right]\\
	&\leq
	2^{p-1}K^{p}
	\varepsilon^{p/2}
	\e[(1+|X_{t-\varepsilon}|)^{p}]
	+
	2^{3p/2-2}K_0^{p/2}
	\e\left[
	1+|X_{t-\varepsilon}|^{\frac{p(\ell+2)}{2}}
	\right]
	\e\left[
	|W_u-W_{t-\varepsilon}|^{p}
	\right] \notag.
	\end{align}
	The assumption $p \in [1,\frac{2p_0}{\ell+2}]$ implies $\frac{p(\ell+2)}{2} \leq p_0$, thus from Lemma \ref{super_moment_0}, we conclude \eqref{super_moment_3}.
\end{proof}
\begin{Rem}
	Note that in the inequality \eqref{Rem_drift_super}, for the drift coefficient $b_{\varepsilon}$, if we use the estimate $|b_{\varepsilon}(x)| \leq |b(x)|$, we need a stronger assumption $p \in [1,\frac{p_0}{\ell+1}]$.
	
\end{Rem}

The following lemma plays a crucial role in our argument.
\begin{Lem}\label{key_super_0}
	Suppose that Assumption \ref{Ass_super_0} holds.
	Then there exists $C>0$ such that
	\begin{align*}
	\e[|X_t-X_t^{(\varepsilon)}|^2]
	\leq
	C \varepsilon^2.
	\end{align*}
\end{Lem}
\begin{proof}
	By It\^o's formula, we have for each $u \in [0,t]$
	\begin{align*}
	&|X_u-X_u^{(\varepsilon)}|^2\\
	&=
	\int_{0}^{u}
	\left\{
	2(X_s-X_s^{(\varepsilon)})
	(b(s,X_s)-\overline{b}^{(\varepsilon)}_s)
	+
	|\sigma(X_s)-\overline{\sigma}^{(\varepsilon)}_s|^2
	\right\}
	\rd s
	+\int_{0}^{u}
	(X_s-X_s^{(\varepsilon)})
	(\sigma(X_s)-\overline{\sigma}^{(\varepsilon)}_s)
	\rd W_s\\
	&=:
	I_u^{(\varepsilon)}
	+M_u^{(\varepsilon)}.
	\end{align*}
	Let $\tau_{N,\varepsilon}:=\inf\{s \geq 0~;~|X_s| \geq N\} \wedge \inf\{s \geq 0~;~|X_s^{(\varepsilon)}| \geq N\}$ for $N \in \n$.
	If $u \in [0,t-\varepsilon]$, then $I_u^{(\varepsilon)}=0$.
	For $u \in [t-\varepsilon,t]$, by using an elementally inequality $|a_1+a_2+a_3|^2 \leq (1+\rho) a_1^2+2(1+\rho^{-1})(a_2^2+a_3^2)$, for all $\rho>0, a_1,a_2,a_3  \geq 0$, we have
	\begin{align*}
	I_u^{(\varepsilon)}
	&=
	\int_{t-\varepsilon}^{u}
	\left\{
	2(X_s-X_s^{(\varepsilon)})
	(b(s,X_s)-b_{\varepsilon}(s,X_{t-\varepsilon}))
	+
	|\sigma(X_s)-\sigma_{\varepsilon}(X_{t-\varepsilon})|^2
	\right\}
	\rd s\\
	&\leq
	\int_{t-\varepsilon}^{u}
	\left\{
	2(X_s-X_s^{(\varepsilon)})
	(b(s,X_s)-b(s,X_s^{(\varepsilon)}))
	+
	(1+\rho)
	|\sigma(X_s)-\sigma(X_s^{(\varepsilon)})|^2
	\right\}
	\rd s\\
	&
	\quad+
	\int_{t-\varepsilon}^{u}
	\left\{
	2(X_s-X_s^{(\varepsilon)})
	(b(s,X_s^{(\varepsilon)})-b(s,X_{t-\varepsilon}^{(\varepsilon)})
	+
	2(1+\rho^{-1})
	|\sigma(X_s^{(\varepsilon)})-\sigma(X_{t-\varepsilon}^{(\varepsilon)})|^2
	\right\}
	\rd s\\
	&
	\quad+
	\int_{t-\varepsilon}^{u}
	\left\{
	2(X_s-X_s^{(\varepsilon)})
	(b(s,X_{t-\varepsilon}^{(\varepsilon)})-b_{\varepsilon}(s,X_{t-\varepsilon}^{(\varepsilon)}))
	+
	2(1+\rho^{-1})
	|\sigma(X_{t-\varepsilon}^{(\varepsilon)})-\sigma_{\varepsilon}(X_{t-\varepsilon}^{(\varepsilon)})|^2
	\right\}
	\rd s.
	\end{align*}
	We define $\rho:=(p_1-2)/2>0$, then since $p_1>2$, $1+\rho < p_1-1$.
	Therefore, from Assumption \ref{Ass_super_0} (i) and Young's inequality, there exists $C>0$ such that
	\begin{align*}
	I_u^{(\varepsilon)}
	&\leq
	(K+2)\int_{0}^{u}
	|X_s-X_s^{(\varepsilon)}|^2
	\rd s
	+CJ_u^{(\varepsilon)}
	+CK_u^{(\varepsilon)},
	\end{align*}
	where for $u \in [t-\varepsilon,t]$,
	\begin{align*}
	J_u^{(\varepsilon)}
	&:=
	\int_{t-\varepsilon}^{u}
	\left\{
	|
	b(s,X_s^{(\varepsilon)})
	-b(s,X_{t-\varepsilon}^{(\varepsilon)})
	|^2
	+
	|
	\sigma(X_s^{(\varepsilon)})
	-\sigma(X_{t-\varepsilon}^{(\varepsilon)})
	|^2
	\right\}
	\rd s\\
	K_u^{(\varepsilon)}
	&:=
	\int_{t-\varepsilon}^{u}
	\left\{
	|
	b(s,X_{t-\varepsilon}^{(\varepsilon)})
	-b_{\varepsilon}(s,X_{t-\varepsilon}^{(\varepsilon)})
	|^2
	+
	|
	\sigma(X_{t-\varepsilon}^{(\varepsilon)})
	-\sigma_{\varepsilon}(X_{t-\varepsilon}^{(\varepsilon)})
	|^2
	\right\}
	\rd s.
	\end{align*}
	
	Therefore, we have for any $u \in [0,t]$
	\begin{align*}
	\e[
	|
	X_{u\wedge \tau_{N,\varepsilon}}
	-X_{u\wedge \tau_{N,\varepsilon}}^{(\varepsilon)}
	|^2
	]
	\leq (K+2)\int_{0}^{u}
	\e[
	|
	X_{s\wedge \tau_{N,\varepsilon}}
	-X_{s\wedge \tau_{N,\varepsilon}}^{(\varepsilon)}
	|^2
	]
	\rd s
	+C\e[J_t^{(\varepsilon)}+K_t^{(\varepsilon)}].
	\end{align*}
	Hence, by Gronwall's inequality, and then by taking $N \to \infty$, we have
	\begin{align*}
	\e[|X_t-X_t^{(\varepsilon)}|^2]
	\leq
	e^{(K+2)T}
	\e[J_t^{(\varepsilon)}+K_t^{(\varepsilon)}].
	\end{align*}
	By Assumption \ref{Ass_super_0} (ii), \eqref{local_sigma_01} and H\"older's inequality, we have for any $p,q>1$ with $1/p+1/q=1$,
	\begin{align*}
	\e[J_t^{(\varepsilon)}]
	&\leq
	(3K^2+K_2)
	\int_{t-\varepsilon}^{t}
	\e\left[
	(
	1
	+|X_{s}^{(\varepsilon)}|^{2\ell}
	+|X_{t-\varepsilon}^{(\varepsilon)}|^{2\ell}
	)
	|
	X_{s}^{(\varepsilon)}
	-X_{t-\varepsilon}^{(\varepsilon)}
	|^2
	\right]
	\rd s\\
	&\leq
	(3K^2+K_2)
	\int_{t-\varepsilon}^{t}
	\e\left[
	(
	1
	+|X_{s}^{(\varepsilon)}|^{2\ell}
	+|X_{t-\varepsilon}^{(\varepsilon)}|^{2\ell}
	)^{p}
	\right]^{1/p}
	\e\left[
	|
	X_{s}^{(\varepsilon)}
	-X_{t-\varepsilon}^{(\varepsilon)}
	|^{2q}
	\right]^{1/q}
	\rd s.
	\end{align*}
	Since $\ell \leq \frac{p_0-2}{4}$, by choosing $p=\frac{4\ell +2 }{3\ell}$, then $q=\frac{p}{p-1}=\frac{4\ell+2}{\ell+2}$, $2\ell p \leq \frac{2p_0}{3}<p_0$ and $2q \leq \frac{2p_0}{\ell+2}$.
	Therefore, from Lemma \ref{super_moment_1}, there exists $C>0$ such that
	\begin{align*}
	\e[J_t^{(\varepsilon)}]
	\leq C\varepsilon^2.
	\end{align*}
	By Assumption \ref{Ass_super_0} (ii), \eqref{tamed_coef_1} and \eqref{tamed_coef_2}, we have
	\begin{align*}
	\e[K_t^{(\varepsilon)}]
	&\leq
	\varepsilon
	\int_{t-\varepsilon}^{t}
	\e\left[
	K^2
	(
	1+|X_{t-\varepsilon}^{(\varepsilon)}|^{\ell+1}
	)^2
	|X_{t-\varepsilon}^{(\varepsilon)}|^{2\ell}
	+K_0(1+|X_{t-\varepsilon}^{(\varepsilon)}|^{\ell+2})|X_{t-\varepsilon}^{(\varepsilon)}|^{\ell}
	\right]
	\rd s\\
	&=\varepsilon^2
	\e\left[
	K^2
	(
	1+|X_{t-\varepsilon}^{(\varepsilon)}|^{\ell+1}
	)^2
	|X_{t-\varepsilon}^{(\varepsilon)}|^{2\ell}
	+K_0(1+|X_{t-\varepsilon}^{(\varepsilon)}|^{\ell+2})|X_{t-\varepsilon}^{(\varepsilon)}|^{\ell}
	\right]
	\end{align*}
	Since $\ell \leq \frac{p_0-2}{4}$, $2\ell+2<4\ell+2 \leq p_0$.
	Therefore, from Lemma \ref{super_moment_1}, there exists $C>0$ such that
	\begin{align*}
	\e[K_t^{(\varepsilon)}]
	\leq C\varepsilon^2.
	\end{align*}
	This concludes the statement.
\end{proof}

\begin{proof}[Proof of Theorem \ref{super_density_0}]
	The proof is based on Theorem 2.1 in \cite{FoPr10}.
	Let $\delta>0$.
	We define a function $f_{\delta}:[0,\infty) \to [0,1]$ such that
	\begin{align*}
	f_{\delta}=0, \text{ on } [0,\delta], \text{ and }
	\sup_{x \neq y} \frac{|f_{\delta}(x)-f_{\delta}(y)|}{|x-y|} \leq 1.
	\end{align*}
	We denote
	\begin{align*}
	\mu_{X_t}(\rd x)
	&:=\p(X_t \in \rd x),~
	\mu_{\delta,X_t}(\rd x)
	:=f_{\delta} (|\sigma(x)|) \mu_{X_t}(\rd x),~
	\widehat{\mu}_{\delta,X_t}(\xi)
	:=\int_{\real} e^{\sqrt{-1}\xi x} \mu_{\delta,X_t}(\rd x),
	\end{align*}
	for $t \in (0,T]$ and $\xi \in \real$.
	Then it follows from Lemma 1.1 and Lemma 1.2 in \cite{FoPr10} that it is suffices to prove that $\widehat{\mu}_{\delta,X_t}$ satisfies
	\begin{align*}
	\int_{\real} |\widehat{\mu}_{\delta,X_t}(\xi)|^2 \rd \xi
	<\infty.
	\end{align*}
	
	For any $\xi \in \real$, it holds that
	\begin{align*}
	\left|
	\widehat{\mu}_{\delta,X_t}(\xi)
	\right|
	&=
	\left|
	\int_{\real}
	e^{\sqrt{-1} \xi x} f_{\delta}(|\sigma(x)|)
	\mu_{X_t}(\rd x)
	\right|
	=
	\left|
	\e\left[
	e^{\sqrt{-1} \xi X_t} f_{\delta}(|\sigma(X_t)|)
	\right]
	\right| \notag\\
	&\leq
	\left|
	\e\left[
	e^{\sqrt{-1} \xi X_t} f_{\delta}(|\sigma_{\varepsilon}(X_{t-\varepsilon})|)
	\right]
	\right|
	+
	\e\left[
	\left|
	f_{\delta}(|\sigma(X_{t})|)-f_{\delta}(|\sigma_{\varepsilon}(X_{t-\varepsilon})|)
	\right|
	\right].
	\end{align*}
	Since $X_t=X_t^{(\varepsilon)}+(X_t-X_t^{(\varepsilon)})$, by using the inequality $|e^{ \sqrt{-1}\xi x}-e^{\sqrt{-1} \xi y}| \leq |\xi| |x-y|$, $\|f_{\delta}\|_{\infty}\leq 1$ and Lipschitz continuity of $f_{\delta}$, we have
	\begin{align}\label{super_pr_1}
	\left|
	\widehat{\mu}_{\delta,X_t}(\xi)
	\right|
	&\leq
	\left|
	\e\left[
	e^{\sqrt{-1} \xi X_{t}^{(\varepsilon)}} f_{\delta}(|\sigma_{\varepsilon}(X_{t-\varepsilon})|)
	\right]
	\right|
	+|\xi|
	\e\left[
	\left|
	X_{t}
	-X^{(\varepsilon)}_{t}
	\right|
	\right] \notag\\
	&\quad
	+
	\e\left[
	\left|
	\sigma(X_{t})
	-\sigma(X_{t-\varepsilon})
	\right|
	\right]
	+
	\e\left[
	\left|
	\sigma(X_{t-\varepsilon})
	-\sigma_{\varepsilon}(X_{t-\varepsilon})
	\right|
	\right].
	\end{align}
	Note that
	\begin{align*}
	&\e[e^{\sqrt{-1} \xi X_t^{(\varepsilon)}} | \mathcal{F}_{t-\varepsilon}]\\
	&=
	\exp
	\left(
	\sqrt{-1} \xi
	\left\{
	X_{t-\varepsilon}
	+\int_{t-\varepsilon}^{t}
	b_{\varepsilon}(s,X_{t-\varepsilon})
	\rd s
	\right\}
	\right)
	\e
	\left[
	\exp
	\left(
	\sqrt{-1}
	\xi
	\sigma_{\varepsilon}(X_{t-\varepsilon})
	(W_t-W_{t-\varepsilon})
	\right)
	~|~
	\mathcal{F}_{t-\varepsilon}
	\right] \\
	&=
	\exp
	\left(
	\sqrt{-1} \xi
	\left\{
	X_{t-\varepsilon}
	+\int_{t-\varepsilon}^{t}
	b_{\varepsilon}(s,X_{t-\varepsilon})
	\rd s
	\right\}
	\right)
	\exp
	\left(
	-\frac{\varepsilon \sigma_{\varepsilon}(X_{t-\varepsilon})^2 \xi^2}{2}
	\right),
	\end{align*}
	thus, we obtain
	\begin{align*}
	\left|
	\e[e^{\sqrt{-1} \xi X_t^{(\varepsilon)}} | \mathcal{F}_{t-\varepsilon}]
	\right|
	&=
	\exp
	\left(
	-\frac{\varepsilon \sigma_{\varepsilon}(X_{t-\varepsilon})^2 \xi^2}{2}
	\right).
	\end{align*}
	Therefore, it holds that
	\begin{align}\label{super_pr_2}
		&
		\left|
			\e\left[
				e^{\sqrt{-1} \xi X_{t}^{(\varepsilon)}} f_{\delta}(|\sigma_{\varepsilon}(X_{t-\varepsilon})|)
			\right]
		\right|
		\leq
			\e\left[
				\left|
					\e\left[
						e^{\sqrt{-1} \xi X_{t}^{(\varepsilon)}}
						f_{\delta}(|\sigma_{\varepsilon}(X_{t-\varepsilon})|)
						~\middle|~\mathcal{F}_{t-\varepsilon}
					\right]
				\right|
			\right] \notag\\
		&=
		\e
			\left[
				f_{\delta}(|\sigma_{\varepsilon}(X_{t-\varepsilon})|)
				\left|
					\e
					\left[
						e^{\sqrt{-1} \xi X_{t}^{(\varepsilon)}}
						~\middle|~
						\mathcal{F}_{t-\varepsilon}
					\right]
				\right|
			\right]
		=
		\e
			\left[
				f_{\delta}(|\sigma_{\varepsilon}(X_{t-\varepsilon})|)
				\exp
					\left(
						-\frac{\varepsilon \sigma_{\varepsilon}(X_{t-\varepsilon})^2 \xi^2}{2}
					\right)
			\right] \notag\\
		&
		=
		\e
			\left[
				f_{\delta}(|\sigma_{\varepsilon}(X_{t-\varepsilon})|)
				\exp
					\left(
						-\frac{\varepsilon \sigma_{\varepsilon}(X_{t-\varepsilon})^2 \xi^2}{2}
					\right)
				\1_{\{|\sigma_{\varepsilon}(X_{t-\varepsilon})|\geq \delta\} }
			\right]
		\leq
		\exp
			\left(
				-\frac{\varepsilon \delta^2 \xi^2}{2}
			\right).
	\end{align}
	From Lemma \ref{key_super_0}, there exists $\overline{C}_1>0$ such that
	\begin{align}\label{super_pr_3}
	|\xi| \e[|X_{t}-X^{(\varepsilon)}_{t}|]
	\leq |\xi| \e[|X_{t}-X^{(\varepsilon)}_{t}|^2]^{1/2}
	\leq \overline{C}_1 |\xi| \varepsilon.
	\end{align}
	By using Schwarz's inequality, \eqref{local_sigma_01} and Lemma \ref{super_density_0}, there exists $\overline{C}_2>0$ such that
	\begin{align}\label{super_pr_4}
	\e\left[
	\left|
	\sigma(X_{t})
	-\sigma(X_{t-\varepsilon})
	\right|
	\right]
	&\leq
	K_0
	\e[(1+|X_{t}|^{2\ell}+|X_{t-\varepsilon}|^{2\ell})]^{1/2}
	\e[|X_{t}-X_{t-\varepsilon}|^2]^{1/2}
	\leq
	\overline{C}_2 \varepsilon^{1/2}.
	\end{align}
	By using Jensen's inequality, \eqref{tamed_coef_2} and Lemma \ref{super_moment_0}, there exists $\overline{C}_3>0$ such that
	\begin{align}\label{super_pr_5}
	\e\left[
	\left|
	\sigma(X_{t-\varepsilon})
	-\sigma_{\varepsilon}(X_{t-\varepsilon})
	\right|
	\right]
	&
	\leq
	K_{0}\varepsilon^{1/2}
	\e\left[
	(1+|X_{t-\varepsilon}|^{\ell+2})|X_{t-\varepsilon}|^{\ell}
	\right]^{1/2}
	\leq
	\overline{C}_3 \varepsilon^{1/2}.
	\end{align}
	
	From \eqref{super_pr_1}, \eqref{super_pr_2}, \eqref{super_pr_3}, \eqref{super_pr_4} and \eqref{super_pr_5} we obtain
	\begin{align*}
	\left|
	\widehat{\mu}_{\delta,X_t}(\xi)
	\right|
	\leq
	\exp
	\left(
	-\frac{\varepsilon \delta^2 \xi^2}{2}
	\right)
	+ \overline{C}_1|\xi| \varepsilon
	+(\overline{C}_2+\overline{C}_3) \varepsilon^{1/2}.
	\end{align*}
	For $|\xi| \geq 2$, we choose $\varepsilon=(\log |\xi|/|\xi|)^2 \in (0,1)$, then
	\begin{align*}
	\left|
	\widehat{\mu}_{\delta,X_t}(\xi)
	\right|
	\leq
	\exp\left(-\frac{\delta^2 (\log |\xi|)^2 }{2}\right)
	+\overline{C}_1 \frac{(\log |\xi|)^{2}}{|\xi|}
	+ (\overline{C}_2+\overline{C}_3) \frac{\log |\xi|}{|\xi|},
	\end{align*}
	and thus
	\begin{align*}
	\int_{\real} |\widehat{\mu}_{\delta,X_t}(\xi)|^2 \rd \xi
	&\leq
	4
	+3
	\int_{|\xi| \geq 2}
	\left\{
	\exp\left(-\delta^2 (\log |\xi|)^2\right)
	+
	\frac
	{\overline{C}_1^2(\log |\xi|)^{4}}
	{|\xi|^2}
	+
	\frac
	{(\overline{C}_2+\overline{C}_3)^2(\log |\xi|)^{2}}
	{|\xi|^2}
	\right\}
	\rd \xi
	<\infty,
	\end{align*}
	which concludes the proof.
\end{proof}

\section{Appendix}\label{Sec_app}
\subsection{On some Beta type integral}

\begin{Lem}\label{pa_App_1}
	Let $b>-1$ and $a \in [0,1)$.
	Then for any $t_0>0$,
	\begin{align*}
	\int_0^{t_0} \mathrm{d}t_1 \cdots \int_0^{t_{m-1}} \mathrm{d}t_m t_m^{b} \prod_{j=0}^{m-1}(t_j-t_{j+1})^{-a}
	= \frac{t_0^{b+m(1-a)} \Gamma^m(1-a) \Gamma(1+b)}{\Gamma(1+b+m(1-a))}.
	\end{align*}
\end{Lem}
\begin{proof}
	Let $b>-1$ and $a \in [0,1)$.
	Using the change of variables $s=ut$, we have
	\begin{align*}
	\int_0^ts^b(t-s)^{-a}\mathrm{d}s
	=t^{b+1-a}\int_0^1 u^b(1-u)^{-a} \mathrm{d}u
	=t^{b+1-a}B(1+b,1-a),
	\end{align*}
	where $B(x,y)=\int_0^1 t^{x-1}(1-t)^{y-1}$ is the standard Beta function.
	Using this repeatedly, we obtain the statement.
\end{proof}

\section*{Acknowledgements}
The authors would like to thank an anonymous referee for his/her careful readings and comments.
The first author was supported by JSPS KAKENHI Grant Number 17H06833 and 19K14552.
The second author was supported by Sumitomo Mitsui Banking Corporation.


\begin{thebibliography}{99}
	
	
	\bibitem{AnKo}
	{Andersson, P.} and {Kohatsu-Higa, A.}
	{\it Unbiased simulation of stochastic differential equations using parametrix expansions.}
	Bernoulli
	{\bf 23}(3),
	2028--2057,
	(2017).
	
	\bibitem{Ar67}
	{Aronson, D.~G.}
	{\it Bounds for the fundamental solution of a parabolic equation.}
	Bull. Amer. Math. Soc.
	{\bf 73},
	890--896,
	(1967).
	
	\bibitem{BaCa}
	{Bally, V.} and {Caramellino, L.}
	{\it Convergence and regularity of probability laws by using an interpolation method.}
	Ann. Probab.
	{\bf 45}(2),
	1110--1159,
	(2017).
	
	\bibitem{BaKo15}
	{Bally, V.} and {Kohatsu-Higa, A.}
	{\it A probabilistic interpretation of the parametrix method.}
	Ann. Appl. Probab.
	{\bf 25}(6),
	3095--3138,
	(2015).
	
	\bibitem{BT}
	{Bally, V.} and {Talay, D.}
	{\it The law of the Euler scheme for stochastic differential equations: I. Convergence rate of the distribution function.}
	Probab. Theory Relat. Fields.
	104,
	43--60,
	(1996).
	
	\bibitem{BaKr16}
	{Ba\~{n}os, D.} and {Kr\"uhner, P.}
	{\it Optimal density bounds for marginals of It\^o processes.}
	Commun. Stoch. Anal.
	{\bf 10}(2),
	131--150,
	(2016).
	
	\bibitem{BaKr17}
	{Ba\~{n}os, D.} and {Kr\"uhner, P.}
	{\it H\"older continuous densities of solutions of SDEs with measurable and path dependent drift coefficients.}
	Stochastic Process. Appl. 
	{\bf 127},
	1785–-1799,
	(2017).
	
	\bibitem{BaSh18}
	{Bao, J.} and {Shao, J.}
	{\it Weak convergence of path--dependent SDEs with irregular coefficients.}
	arXiv:1809.03088.
	
	\bibitem{Bass}
	{Bass, R.~F.}
	{\it Diffusions and elliptic operators.}
	Springer Science and Business Media
	(1998).
	
	
	
	
	
	
	\bibitem{DM11}
	{De~Marco, S.}
	{ \it Smoothness and asymptotic estimates of densities for {SDE}s with locally smooth coefficients and applications to square root-type diffusions.}
	Ann. Appl. Probab.,
	{\bf 21}(4),
	1282--1321,
	(2011).
	
	\bibitem{DeKr02}
	{Deck, T.} and {Kruse, S.}
	{\it Parabolic differential equations with unbounded coefficients -- a generalization of the parametrix method.}
	Acta Appl. Math.
	{\bf 74},
	71--91,
	(2002).
	
	\bibitem{DoOuWa17}
	{Doumbia, M.}, {Oudjane, N.} and {Warin, X.}
	{\it Unbiased Monte Carlo estimate of stochastic differential equations expectations.} ESAIM Probab. Stat.
	{\bf 21},
	56--87,
	(2017).
	
	
	\bibitem{FaKe81}
	{Fabes, E.~B.} and {Kenig, C.}
	{\it Examples of singular parabolic measures and singular transition probability densities.}
	Duke Math. J.
	{\bf 48}(4),
	845--855,
	(1981).
	
	\bibitem{FoPr10}
	{Fournier, N.} and {Printems, J.}
	{\it Absolute continuity of some one-dimensional processes.}
	Bernoulli
	{\bf 16},
	343-–360,
	(2010).
	
	\bibitem{Fr64}
	{Friedman, A.}
	{\it Partial Differential Equations of Parabolic Type.}
	Dover Publications Inc.,
	(1964).
	
	\bibitem{Fr75}
	{Friedman, A.}
	{\it Stochastic differential equations and applications. Courier Corporation.}
	Academic Press,
	New York,
	(1975).
	
	\bibitem{FrLi}
	{Frikha, N.} and  {Li, L.}
	{\it Weak uniqueness and density estimates for sdes with coefficients depending on some path-functionals.}
	arXiv:1707.01295.
	
	
	\bibitem{GoMa13}
	{Goard, J. and Mazur, M.}
	{\it Stochastic volatility models and the pricing of VIX options.}
	Math. Finance 
	{\bf 23}(3),
	439--458,
	(2013).
	
	\bibitem{G08}
	{Gobet, E.} and {Labart, C.}
	{\it Sharp estimates for the convergence of the density of the Euler scheme in small time.}
	Elect. Comm. in Probab.,
	{\bf 13},
	352--363,
	(2008).
	
	\bibitem{Gu}
	{Guyon, J.}
	{\it Euler schemes and tempered distributions.}
	Stochastic. Process. Appl. 
	{\bf 116},
	887--904,
	(2006).
	
	\bibitem{HaKoYu13}
	{Hayashi, M.}, {Kohatsu-Higa, A.} and {Y\^uki, G.}
	{\it Local H\"older continuity property of the densities of solutions of SDEs with singular coefficients.}
	J. Theor. Probab.
	{\bf 26},
	1117--1134,
	(2013).
	
	\bibitem{HaKoYu14}
	{Hayashi, M.}, {Kohatsu-Higa, A.} and {Y\^uki, G.}
	{\it H\"older continuity property of the densities of SDEs with singular drift coefficients.}
	Electorn. J. Probab.
	{\bf 19}(77),
	1--22,
	(2014).
	
	
	\bibitem{HeTaTo17}
	{Henry-Labordere, P.}, {Tan, X.} and {Touzi, N.}
	{\it Unbiased simulation of stochastic differential equations.}
	Ann. Appl. Probab.,
	{\bf 27}(6),
	3305--3341,
	(2017)
	
	\bibitem{HuJeKl11}
	{Hutzenthalerm M.}, {Jentzen, A.} and {Kloeden, P.~E.}
	{\it Strong and weak divergence in finite time of Euler's method for stochastic differential equations with non-globally Lipschitz continuous coefficients.}
	Proc. R. Soc. A
	{\bf 467},
	1563--1576,
	(2011).
	
	\bibitem{HuJeKl12}
	{Hutzenthalerm M.}, {Jentzen, A.} and {Kloeden, P.~E.}
	{\it Strong convergence of an explicit numerical method for SDEs with non-globally Lipschitz continuous coefficients.}
	Ann. Appl. Probab.,
	{\bf 22}(4),
	1611--1641,
	(2012).
	
	\bibitem{IdIm}
	{Ida, Y.} and {Imamura, Y.}
	{\it Towards the exact simulation using hyperbolic Brownian motion.}
	Jpn. J. Ind. Appl. Math.
	{\bf 34}(3),
	833--843,
	(2017).
	
	\bibitem{IkWa}
	{Ikeda, N.} and {Watanabe, S.}
	{\it Stochastic differential equations and diffusion processes, second ed.}
	volume~24 of North-Holland Mathematical Library,
	North-Holland Publishing Co., Amsterdam-New York; Kodansha, Ltd.,
	Tokyo,
	(1981).
	
	\bibitem{KS}
	{Karatzas, I.}  and {Shreve, S.~E.}
	{\it Brownian motion and stochastic calculus. Second edition.}
	Springer (1991).
	
	\bibitem{KiSo}
	{Kim, P.} and {Song, S.}
	{\it Two-sided estimates on the density of Brownian motion with singular drift.}
	Illinois J. Math.
	{\bf 50}(3),
	635--688,
	(2006).
	
	\bibitem{KP}
	{Kloeden, P.} and  {Platen, E.}
	{\it Numerical Solution of Stochastic Differential Equations.}
	Springer
	(1995).
	
	
	
	\bibitem{KoTaZh}
	{Kohatsu-Higa, A.}, {Taguchi, T.} and {Zhong, J.}
	{\it The Parametrix Method for Skew Diffusions.}
	Potential Anal.
	{\bf 45},
	299–-329,
	(2016).
	
	\bibitem{KoTa12}
	{Kohatsu-Higa, A.} and {Tanaka, A.}
	{\it A Malliavin calculus method to study densities of additive functionals of SDE's with irregular drifts.}
	Ann. Inst. Henri Poincar\'e
	{\bf 48}(3),
	871--883,
	(2012).
	
	\bibitem{KLY}
	{Kohatsu-Higa, A.}, {Lejay, A.} and {Yasuda, K.}
	{\it On Weak Approximation of Stochastic Differential Equations with Discontinuous Drift Coefficient.}
	J. Comput. Appl. Math.,
	{\bf 326},
	138--158,
	(2017).
	
	\bibitem{KoYu}
	{Kohatsu-Higa, A.} and {Y\^uki, G.}
	{\it Stochastic formulations of the parametrix method.}
	to appear in ESAIM Probab. Stat.
	
	\bibitem{KoMa}
	{Konakov, V.} and {Mammen, E.}
	{\it Edgeworth type expansions for Euler schemes for stochastic differential equations.}
	Monte Carlo Methods Appl.
	{\bf 8}(3),
	271--285,
	(2002).
	
	\bibitem{KoKoMe}
	{Konakov, V.}, {Kozhina, A.} and {Menozzi, S.}
	{\it Stability of densities for perturbed Diffusions and Markov Chains.}
	ESAIM Probab. Stat.
	{\bf 21}, 
	88--112,
	(2017).
	
	\bibitem{KuSt85}
	{Kusuoka, S.} and {Stroock, D.~W.}
	{\it Applications of the Malliavin calculus, I.}
	Taniguchi Symposium on Stochastic Analysis, Katata 1982 (K. It\^o editor)
	271--306,
	North--Holland, Amsterdam
	(1983).
	
	\bibitem{Ku17}
	{Kusuoka, S.}
	{\it Continuity and Gaussian two--sided bounds of the density functions of the solutions to path--dependent stochastic differential equations via perturbation.}
	Stochastic Process. Appl.,
	{\bf 127},
	359--384,
	(2017).
	
	
	\bibitem{LeMe}
	{Lemaire, V.} and {Menozzi, S.}
	{\it On some Non-Asymptotic Bounds for the Euler Scheme.}
	Electron J. Probab.
	{\bf 15},
	1645--1681,
	(2010).
	
	\bibitem{LeSz17b}
	{Leobacher, G.} and  {Sz\"olgyenyi, M.}
	{\it Convergence of the Euler-Maruyama method for multidimensional SDEs with discontinuous drift and degenerate diffusion coefficient.}
	Numer. Math.
	{\bf 138}(1),
	219--239,
	(2018).
	
	\bibitem{Ma03}
	{Mackevi\v{c}ius, V.}
	{\it On the convergence rate of Euler scheme for SDE with Lipschitz drift and constant diffusion.}
	Acta Appl. Math.
	{\bf 78},
	301--310,
	(2003).
	
	\bibitem{Ma15}
	{Makhlouf, A.}
	{\it Existence and Gaussian bounds for the density of Brownian motion with random drift.}
	Commun. Stoch. Anal.
	{\bf 10}(2),
	151--162,
	(2016).
	
	\bibitem{Mar54}
	{Maruyama, G.}
	{\it On the transition probability functions of the Markov process.}
	Nat. Sci. Rep. Ochanomizu Univ.
	{\bf 5},
	10--20,
	(1954).
	
	\bibitem{Ma55}
	{Maruyama, G.}
	{\it Continuous Markov processes and stochastic equations.}
	Rend. Circ. Matem. Palermo
	{\bf 10},
	48--90,
	(1955).
	
	\bibitem{McSi67}
	{McKean, H.~P.} and {Singer, I.~M.}
	{\it Curvature and the eigenvalues of the Laplacian.}
	J. Differential Geom.,
	{\bf 1},
	43--69,
	(1967).
	
	\bibitem{MeTa}
	{Menoukeu Pamen, O.} and {Taguchi, D.}
	{\it Strong rate of convergence for the Euler-Maruyama approximation of SDEs with H\"older continuous drift coefficient.}
	hastic Process. Appl. 
	{\bf 127},
	2542--2559,
	(2017).
	
	\bibitem{MP}
	{Mikulevicius, R.} and {Platen, E.}
	{\it Rate of convergence of the Euler approximation for diffusion processes.}
	Math. Nachr.,
	{\bf 151},
	233--239,
	(1991).
	
	
	
	\bibitem{Na58}
	{Nash, J.}
	{\it Continuity of solutions of parabolic and elliptec equations.}
	Amer. J. Math.
	{\bf 80},
	931--954,
	(1958).
	
	\bibitem{NT1}
	{Ngo, H-L.} and {Taguchi, D.}
	{\it Strong rate of convergence for the Euler--Maruyama approximation of stochastic differential equations with irregular coefficients.}
	Math. Comp.
	{\bf 85}(300),
	1793--1819,
	(2016).
	
	\bibitem{NT2}
	{Ngo, H-L.} and {Taguchi, D.}
	{\it Approximation for non-smooth functionals of stochastic differential equations  with irregular drift.}
	J. Math. Anal. Appl.
	{\bf 457}(1),
	361--388,
	(2018).
	
	\bibitem{Nu06}
	{Nualart, D.}
	{\it The Mallivain Calculus and Related Topics.}
	Springer,
	(2006).
	
	\bibitem{OlTu}
	{Olivera, C.} and {Tudor, C.~A.}
	{\it Density for solutions to stochastic differential equations with unbounded drift.}
	arXiv:1805.0671.
	
	\bibitem{Po}
	{Portenko, N.~I.}
	{\it Generalized diffusion processes.}
	Naukova Dumka, Kiev, 1982,
	English transl.: American Mathematical Society, Providence, R.I.,
	(1990).
	
	\bibitem{QiZh02}
	{Qian, Z.} and {Zheng, W.}
	{\it Sharp bounds for transition probability densities of a class of diffusions.}
	C. R. Acad. Sci. Paris, Ser
	{\bf 335}(11),
	953--957,
	(2002).
	
	\bibitem{QiZh03}
	{Qian, Z.} and {Zheng, W.}
	{\it Comparison theorem and estimates for transition probability densities of diffusion processes.}
	Probab. Theory Relat. Fields.
	{\bf 127},
	388--406,
	(2003).
	
	\bibitem{QiZh04}
	{Qian, Z.} and {Zheng, W.}
	{\it A representation formula for transition probability densities of diffusions and applications.}
	Stoch. Anal. Appl.
	{\bf 111},
	57-76,
	(2004).
	
	\bibitem{Ro18}
	{Romito, M.}
	{\it A simple method for the existence of a density for stochastic evolutions with rough coefficients.}
	Electorn. J. Probab.
	{\bf 23}(113),
	1--43,
	(2018).
	
	
	\bibitem{Sa13}
	{Sabanis, S.}
	{\it A note on tamed Euler approximations.}
	Electron. Commun. Probab.
	{\bf 18}(47),
	1--10,
	(2013).
	
	\bibitem{Sa16}
	{Sabanis, S.}
	{\it Euler approximations with varying coefficients: the case of superlinearly growing diffusion coefficients.}
	Ann. Appl. Probab.,
	{\bf 25}(4),
	2083--2105,
	(2016).
	
	\bibitem{Sh91}
	{Sheu, S.~J.}
	{\it Some estimates of the transition density of a nondegenerate diffusion Markov process.}
	Ann. Probab.,
	{\bf 19}(2),
	538--561,
	(1991).
	
	\bibitem{Shi04}
	{Shigekawa, I.}
	{\it Stochastic analysis},
	volume 224 of {\em Translations of Mathematical Monographs}.
	\newblock American Mathematical Society, Providence, RI, 2004.
	\newblock Translated from the 1998 Japanese original by the author, Iwanami
	Series in Modern Mathematics.
	
	
	\bibitem{StVa}
	{Stroock, D.~W.} and {Varadhan, S.~R.~S.}
	{\it Diffusion processes with continuous coefficients, I,II.}
	Communications in Pure and Applied Mathematics,
	{\bf 22},
	345--400, 479--530,
	(1969).
	
	
	\bibitem{Take16}
	{Takeuchi, A.}
	{\it Joint distributions for stochastic functional differential equations.}
	Stochastics
	{\bf 88}(5),
	711--736,
	(2016).
	
	\bibitem{TT}
	{Talay, D.} and {Tubaro, L.}
	{\it Expansion of the global error for numerical schemes solving stochastic differential equations.}
	Stochastic Anal. Appl.,
	{\bf 8},
	94--120,
	(1990).
	
	\bibitem{Ve}
	{Veretennikov, A.~YU.}
	{\it On strong solution and explicit formulas for solutions of stochastic integral equations.}
	Math. USSR Sb. 
	{\bf 39},
	387-- 403,
	(1981).
	
	
	
	\bibitem{Yan}
	{Yan, B.~L.}
	{\it The Euler scheme with irregular coefficients.}
	Ann. Probab.
	{\bf 30}(3),
	1172--1194,
	(2002).
	
	\bibitem{Zv}
	{Zvonkin, A.~K.}
	{\it A transformation of the phase space of a diffusion process that removes the drift.}
	Math. USSR Sbornik,
	{\bf 22},
	129--148,
	(1974).
	
\end{thebibliography}
\end{document}